\documentclass[11pt]{article}
\usepackage{fullpage}
\usepackage{amsmath,amssymb,graphicx,float,epsf}
\def\RR{\hbox{I\kern-.2em\hbox{R}}}
\numberwithin{equation}{section}
\restylefloat{figure}
\usepackage{float}
\usepackage{color}
\usepackage{xcolor}
\usepackage{amsmath}
\usepackage{amsthm}
\usepackage{amsfonts}
\usepackage[caption = false]{subfig} 
\usepackage{enumerate}
\usepackage{authblk} 

\usepackage[utf8]{inputenc}
\numberwithin{equation}{section}

\newtheorem{Def}{Definition}
\newtheorem{Lemma}{Lemma}
\newtheorem{Th}{Theorem}

\numberwithin{equation}{section}

\usepackage[ruled,vlined]{algorithm2e}
\newcommand{\bx}{\boldsymbol{x}}

\usepackage{graphicx}
\usepackage{varioref}
\usepackage{hyperref}
\usepackage{comment}

\begin{document}
	
	\date{}
	
\title{The role of harvesting and growth rate for spatially heterogeneous populations}
	
	\author[1]{\small Md. Mashih Ibn Yasin Adan \thanks{Email: mdadan1081@gmail.com}} 
	\author[2]{\small Md. Kamrujjaman \thanks{Corresponding author: M. Kamrujjaman, email: kamrujjaman@du.ac.bd, \\ ORCID ID: \url{https://orcid.org/0000-0002-4892-745X}}}
	\author[3,4]{\small Md. Mamun Molla\thanks{Email: mamun.molla@northsouth.edu}} 
	\author[5]{\small  Muhammad Mohebujjaman\thanks{Email: m.mohebujjaman@tamiu.edu}} 	
	\author[6]{\small  Clarisa Buenrostro\thanks{Email:  clarisabuenrostro@dusty.tamiu.edu}} 	
	
	\affil[1,2]{\footnotesize Department of Mathematics, University of Dhaka, Dhaka 1000, Bangladesh}
	\affil[3]{\footnotesize Department of Mathematics \& Physics, 		North South University, Dhaka-1229, Bangladesh}
    \affil[4]{\footnotesize Center for Applied Scientific Computing (CASC), North South University, Dhaka-1229, Bangladesh}	
    \affil[5,6]{\footnotesize 	Department of Mathematics and Physics, Texas A\&M International University, Laredo, TX 78041, USA}	

	\maketitle
	
	\vspace{-0.5cm}
	\noindent\rule{6.35in}{0.02in}\\
{\bf Abstract.}	
This paper investigates the competition of two species in a heterogeneous environment subject to the effect of harvesting. The most realistic harvesting case is connected with the intrinsic growth rate, and the harvesting functions are developed based on this clause instead of random choice.   We prove the existence and uniqueness of the solution to the model we consider. Theoretically, we state that when species coexist, one may drive the other to die out, and both species extinct, considering all possible rational values of parameters. These results highlight a comparative study between two harvesting coefficients. Finally, we solve the model using a backward-Euler, decoupled, and linearized time-stepping fully discrete algorithm and observe a match between the theoretical and numerical findings. 
	\\[-3mm]
	\noindent\rule{6.35in}{0.02in}

\section{Introduction}
In population dynamics, harvesting is quite common and is always visible in ecology. In the natural or human haphazardness, harvesting reduces species due to hunting, fishing, disease, war, environmental effects like natural disasters, competition among the species for the same resources, limited living space, and limited food supply. To study the two species competition model, harvesting is an important term. To know the ecological system, harvesting must be considered for species because species are reducing continuously. To protect the species and maintain the balance of the ecological system, we should know the threshold of harvesting, so that species can not go extinct. The study of harvesting is very effective not only in ecology but also in economics. The population model with harvesting greatly impacts the economy like fisheries, forestry, plants, and poultry.

In \cite{d1}, delineated two species harvesting where harvested independently with constant rates, and the highest secure harvesting may be much less than what would be considered from a local analysis for the equilibrium point. In \cite{d2}, investigated the global behavior of predator-prey systems in the presence of continuous harvesting and preserving of either or both species. This is analogous to the characteristics of an unharvested system with several parameters. In \cite{d3}, studied the combined impacts of harvesting and discrete-time delay on the predator-prey system. A comparative examination of stability behavior has been offered in the absence of time delay. The study \cite{d4} emphasized the crucial concept in the ecological system that a perfect mathematical model cannot be gained since we cannot include all of the effective parameters in the model. Moreover, the model will never be able to forecast ecological catastrophes. As a result, we can analyze the models which describe and represent the reality of population harvesting. 
		\begin{itemize}
			\item This study aims to illustrate the comparative study between the prey and predator species harvesting rate, where we have established the result when they can coexist or when one species derive to other species to extinction, or when both species die out. This study gives a translucent idea about real-life scenarios of predator and prey species in the population ecology. 
		\end{itemize}
	
In this paper, we study the impact of harvesting on the consequence of the interaction, like the competition of two species in a spatially non-homogeneous environment. Here competition arises for the same resources, limited food supply, and limited living space; predators make predation prey species for their food. Taking into account harvesting rate is proportional to the intrinsic growth rate such that harvesting functions be $E_{1}(\bx)\propto r(\bx)$ and $ E_{2}(\bx) \propto r(\bx) $ which implies $E_{1}(\bx)=\mu r(\bx),\; E_{2}(\bx)=\nu r(\bx)$, where $\mu,\nu$ are coefficient of proportionality which are non-negative. 
	The model equations are
	\begin{equation}\label{equ_1}
	\begin{cases}
	& \displaystyle\frac{\partial u}{\partial t}= d_{1}\Delta u(t,\bx)+r(\bx)u(t,\bx)\left(1-\dfrac{u(t,\bx)+v(t,\bx)}{K(\bx)}\right)-\mu r(\bx)u(t,\bx),\\
	& \displaystyle\frac{\partial v}{\partial t}=d_{2}\Delta v(t,\bx)+r(\bx)v(t,\bx)\left(1-\dfrac{u(t,\bx)+v(t,\bx)}{K(\bx)}\right)-\nu r(\bx)v(t,\bx),\\ 
	& \displaystyle t>0,\;\;\; \bx\in\Omega,\\
	& \displaystyle\dfrac{\partial u}{\partial \eta}=\dfrac{\partial v}{\partial \eta}=0,\;\;\; \bx\in \partial\Omega,\\
	& u(0,\bx)=u_0(\bx),\;v(0,\bx)=v_0(\bx),\;\;\; \bx\in\Omega.
	\end{cases}
	\end{equation}
	where $u(t,\bx)$, $v(t,\bx)$ represent the population densities of two competing species which are non-negative, with corresponding dispersal rates $d_{1}, d_{2}$, respectively. Note that the analogous model is discussed in \cite{Ref_12}. Moreover, in \cite{Ref_12} studied directed diffusion strategies with harvesting but in this paper, we investigate for regular (random) diffusion strategy. Regular diffusion strategy quite challenging to analysis, see \cite{Ref_18, Ref_19, Ref_20, Ref_21, Ref_22, Ref_23, Ref_24} and references therein.
	There are several scenarios can happen when harvesting is applied to single or more of various interacting species and different diffusive strategy \cite{Ref_12, Ref_15}. 
In \cite{Ref_15} investigated single species with harvesting function where harvesting function is time and space {dependent} but in \cite{Ref_12} investigated competitive two species with harvesting effort where harvesting function is time-independent. In the study, \cite{Ref_16} showed a non-homogeneous Gilpin–Ayala diffusive equation for single species with harvesting where the harvesting function is space {dependent}. 

Consider the initial conditions $u_{0}(\bx)\geq 0$, $v_{0}(\bx)\geq 0$, $\bx\in\overline{\Omega}$ and these initial conditions are positive in an open nonempty subdomain of $\Omega$. 
	Carrying capacity and intrinsic growth rate are denoted by $K(\bx)$ and $r(\bx)$, respectively. 
	 The function $K(\bx)$ is continuous as well as positive on $\overline{\Omega}$ and $r(\bx)\geqslant 0$ where $\bx\in\overline\Omega$, moreover $r(\bx)$ is positive in an open nonempty subdomain of $\Omega$. The notation $\Omega$ is a bounded region in $\mathbb{R}^{n}$, typically $ n=\{1,2,3\} $, with smooth boundary $\partial \Omega\in C^{2+\alpha},\;0<\alpha<1$ and $\eta $ represents the unit normal vector on $\partial \Omega$.
	The zero Neumann boundary condition indicates that no individual crosses the boundary of the habitat or individuals going in and out at any location from the boundary stay equal at all times. The Laplace operator $\Delta:=\sum_{i=1}^n \partial^{2}/ \partial x_{i}^{2}$ in $\mathbb{R}^{n}$ implies that the random motion of the species.
	
	Now we modify the system (\ref{equ_1}) in such a way that no harvesting rate is present. The first equation of the model (\ref{equ_1}) can be written in the following way
	\begin{align*}
	\frac{\partial u}{\partial t} &= d_{1}\Delta u(t,\bx)+r(\bx)u(t,\bx)\left(1-\dfrac{u(t,\bx)+v(t,\bx)}{K(\bx)}\right)-\mu r(\bx)u(t,\bx)\\
		&= d_{1}\Delta u(t,\bx)+r(\bx)u(t,\bx)(1-\mu)\left(1 -\dfrac{u(t,\bx)+v(t,\bx)}{(1-\mu)K(\bx)}\right).
	\end{align*}
	Let $ K_{1}(\bx)=(1-\mu)K(\bx)$ and $r_{1}=1-\mu$. Then we obtain 
	\begin{align*}
	\frac{\partial u}{\partial t} = d_{1}\Delta u(t,\bx)+r_{1}r(\bx)u(t,\bx)\left(1 -\frac{u(t,\bx)+v(t,\bx)}{K_{1}(\bx)}\right).
	\end{align*}
	After same work for the second equation of system (\ref{equ_1}), finally we obtain
	\begin{equation}\label{equ_2}
	\begin{cases}
	& \displaystyle\frac{\partial u}{\partial t}= d_{1}\Delta u(t,\bx)+r_{1}r(\bx)u(t,\bx)\left(1-\frac{u(t,\bx)+v(t,\bx)}{K_{1}(\bx)}\right),\\
	& \displaystyle\frac{\partial v}{\partial t}=d_{2}\Delta v(t,\bx)+r_{2}r(\bx)v(t,\bx)\left(1-\frac{u(t,\bx)+v(t,\bx)}{K_{2}(\bx)}\right),\\ 
	& \displaystyle t>0,\;\;\; \bx\in\Omega,\\
	& \displaystyle\frac{\partial u}{\partial \eta}=\frac{\partial v}{\partial \eta}=0,\;\;\; \bx\in\ \partial\Omega,\\
	& u(0,\bx)=u_0(\bx),\;v(0,\bx)=v_0(\bx),\;\;\; \bx\in\Omega,
	\end{cases}
	\end{equation}
	here $K_{1}(\bx)=(1-\mu)K(\bx), K_{2}(\bx)=(1-\nu)K(\bx), r_{1}=1-\mu$, and $ r_{2}=1-\nu$.
		
	We solve \eqref{equ_1} numerically using a stable backward-Euler, decoupled, and linearized fully discrete time-stepping algorithm in a finite element setting and examine whether the theoretical results are supported by giving several numerical experiments.
	
	The rest of the paper is organized as below: In Section \ref{Existence-and-uniqueness}, the existence and uniqueness of the solution of equation \ref{equ_1} are proven. The necessary preliminary discussions are provided in Section \ref{Preliminaries}. In Section \ref{stability-analysis}, stability analysis of the equilibrium points is given when the intrinsic growth rate exceeds harvesting rates, one harvesting rate exceeds the intrinsic growth rate, and both harvesting rates exceed the intrinsic growth rate. To support the theoretical findings, several numerical experiments are given in Section \ref{Numerical-experiments}. Finally, a concluding summary and future research directions are discussed in Section \ref{conclusion}.
	
	\section{Existence and uniqueness}\label{Existence-and-uniqueness}
	Now we detach each equation to delineate the existence as well as the uniqueness of the paired system. Consider the following system
	\begin{equation}\label{equ_ex1}
	\begin{cases}
	& \displaystyle\frac{\partial u}{\partial t}= d_{1}\Delta u(t,\bx)+r_{1}r(\bx)u(t,\bx)\left(1-\frac{u(t,\bx)}{K_{1}(\bx)}\right),\;\;t>0,\;\; \bx\in\Omega,\\
	& u(0,\bx)=u_0(\bx),\;\;\; \bx\in\Omega,\\
	& \displaystyle\frac{\partial u}{\partial \eta}=0,\;\; \bx\in \partial\Omega.
	
	\end{cases}
	\end{equation}
	The following results also discussed in \cite{Ref_3,Ref_5,Ref_6, Ref_66}. Note that the proofs of Lemma \ref{lemma_ex1} and Lemma \ref{lemma_ex2} are analogous to the proofs of [\cite{Ref_3} Theorem 1.14, Proposition 3.2-3.3].

	\begin{Lemma}\label{lemma_ex1}
		Consider the parameters are positive on $\overline{\Omega}$ and the initial condition of (\ref{equ_ex1}) be a nonnegative continuous function $u_0(\bx)\in C(\Omega),\;u_0(\bx)\geq 0$ in $\Omega$ and $u_0(\bx)>0$ in some open bounded nonempty domain $\Omega_1 \subset \Omega$. Thus there exists a unique positive solution of the system (\ref{equ_ex1}).	
		\end{Lemma}
	\begin{proof}
		Take into account
		\begin{align*}
		f(\bx,u)=g(\bx,u)u=r_{1}r(\bx)u(t,\bx)\left(1-\dfrac{u(t,\bx)}{K_{1}(\bx)}\right)
		\end{align*}
		where, $g(\bx,u)=r_{1}r(\bx)\left(1-\dfrac{u(t,\bx)}{K_{1}(\bx)}\right)$. The system (\ref{equ_ex1}) becomes
		\begin{equation}\label{equ_ex1.1}
		\begin{cases}
		& \displaystyle\frac{\partial u}{\partial t}= d_{1}\Delta u(t,\bx)+g(\bx,u)u,\;\;t>0,\;\; \bx\in\Omega,\\
		& u(0,\bx)=u_0(\bx),\;\;\; \bx\in\Omega,\\
		& \displaystyle\frac{\partial u}{\partial \eta}=0,\;\; \bx\in \partial\Omega.
		\end{cases}
		\end{equation}
 Here, $f(\bx, u)$ is Lipschitz in $u$ as well as a measurable function in $\bx$, moreover bounded since $u$ constrain to a bounded set, here $\Omega$ is bounded and $\partial \Omega$ is member of class $C^{2+\alpha}$. Assume $f(\bx,u)=g(\bx,u)u$ where in $u$, $g(\bx,u)$ is member of class $C^2$, further there exists $K_1>0$ which implies that $g(\bx,u)<0$ when $u>K_1$. The corresponding eigenvalue problem of (\ref{equ_ex1.1}) is represented in the below
		\begin{equation}\label{equ_ex1.2}
		\sigma\psi = d_{1}\Delta \psi+g(\bx,0)\psi ,\;\; \bx\in \Omega,\;\; \frac{\partial \psi}{\partial \eta}=0,\;\; \bx\in\partial\Omega.
		\end{equation}
The following can be written on the assumptions of $f(\bx,u)$
		such that $f(\bx,u)=\left(g(\bx,0)+ g_1(\bx,u) u \right)u $. If this problem has a positive principal eigenvalue $ \sigma_{1} $.
	Let $\Psi$ be an eigenfunction for (\ref{equ_ex1.2}) with $\Psi>0$ on $\Omega$. For $\epsilon>0$  sufficiently small,
	\begin{align*}
	 d_{1}\Delta (\epsilon \Psi)+f(\bx,\epsilon \Psi)&=\epsilon[d_1\Delta (\Psi)+g(\bx,0)\Psi]+ g_1(\bx,\epsilon\Psi)\epsilon^2\Psi^2 \\
	 & = \epsilon\sigma_1  \Psi+ g_1(\bx,\epsilon\Psi)\epsilon^2\Psi^2 \\
	 & = \epsilon \Psi \left\{\sigma_1+ g_1(\bx,\epsilon\Psi)\epsilon\Psi\right\}>0.
	\end{align*}
	Thus, $\epsilon\Psi$ is a subsolution of the elliptic equation when $\epsilon>0$ is small.
		\begin{align*}
		\begin{cases}
		&  d_{1}\Delta u+f(\bx,u)=0,\;\; \bx\in\Omega,\\
		& \displaystyle\frac{\partial u}{\partial \eta}=0,\;\; \bx\in \partial\Omega		
		\end{cases}
		\end{align*}
		corresponding to (\ref{equ_ex1}). 
When $\underline{u}(\bx,0)=\epsilon\Psi$, then $\underline{u}(\bx,t)$ is a solution of (\ref{equ_ex1}). At $t=0$, $\partial \underline{u}/\partial t>0$ on $\Omega$ as well as supersolutions' and sub-solutions' general features delineates that in $t$, $\underline{u}(\bx,t)$ is increasing. If $K_1>\underline{u}$ is a supersolution and $u^*$ is the minimal positive solution to (\ref{equ_ex1}) then we have $\underline{u}(\bx,t) \uparrow  u^*(\bx)$ where $t\rightarrow\infty$. (When $\Psi$  be a strict subsolution for each sufficiently small  $\epsilon> 0$, then $u^*(\bx)$ is minimal.) Since $u(x, t)$ is positive but initially nonnegative then $u(x, t)$ be a solution of (\ref{equ_ex1}), hence when $t > 0$, the strong maximum principle exposes $u(\bx, t) > 0$ on $\overline{\Omega}$, which completes the proof.
	\end{proof}

	\begin{Lemma}\label{lemma_ex2}
		 Consider the problem  (\ref{equ_ex1}), then there exists a function $u^*(\bx)>0$ that is a unique equilibrium solution of (\ref{equ_ex1}). Further, for any initial condition $u_0(\bx)\geq 0,\; u_0(\bx)\not\equiv 0$ the solution $u(t, \bx)$ gratifies the condition 
		\begin{align*}
		\lim_{t\rightarrow\infty} u(t, \bx)=u^*(\bx)
		\end{align*}
		uniformly for $\bx\in \overline{\Omega}$.
	\end{Lemma}
\begin{proof}
 Assume the  hypotheses of Lemma \ref{lemma_ex1} are contented, which implies that $g(\bx,u)$ strictly decreasing in $u$ where $u \geq 0$ and $f(\bx,u) = g(\bx,u)u$. Hence the minimal positive steady state $u^*$ is the sole positive steady state of (\ref{equ_ex1}). Now take into account $u^{**}$  is a another positive steady state of (\ref{equ_ex1}) where $u^{*}\neq u^{**}$, therefore when $u^*$ is minimal positive steady state then we have $u^{**}>u^*$ someplace on $\Omega$. When $u^*> 0$ is a steady state of (\ref{equ_ex1}) then it would be a positive solution of the following equation
	\begin{equation}\label{equ_ex1.3}
	\sigma\psi = d_{1}\Delta \psi+g(\bx,u^*)\psi ,\;\; \bx\in \Omega,\;\; \frac{\partial \psi}{\partial \eta}=0,\;\; \bx\in\partial\Omega,
	\end{equation}
	letting $\sigma=0$ for any $\psi$, so that $\sigma_1=0$  would be the principal eigenvalue of (\ref{equ_ex1.3}). Analogously $u^{**}>0$ gratifies 
	\begin{equation}\label{equ_ex1.4}
	\sigma\psi = d_{1}\Delta \psi+g(\bx,u^{**})\psi ,\;\; \bx\in \Omega,\;\; \frac{\partial \psi}{\partial \eta}=0,\;\; \bx\in\partial\Omega,
	\end{equation}
	with $\sigma=0$ for any $\psi$, so $\sigma_1=0$ in (\ref{equ_ex1.4}) also. 
The principal eigenvalue of (\ref{equ_ex1.4}) obviously less than the principal eigenvalue of (\ref{equ_ex1.3}) because $u^{**}>u^*$ on at least part of $\Omega$ as well as $g(\bx,u)$ is strictly decreasing in $u$. Thus $\sigma_1 = 0$ cannot have in both (\ref{equ_ex1.3}) and (\ref{equ_ex1.4}), hence (\ref{equ_ex1}) cannot be any steady state other than the minimal steady sate $u^*$, which completes the proof.
\end{proof}

	Now, take into account the next following problem for population density $v=v(t,\bx)$
	\begin{equation}\label{equ_ex2}
	\begin{cases}
	& \displaystyle\frac{\partial v}{\partial t}= d_{2}\Delta v(t,\bx)+r_{2}r(\bx)v(t,\bx)\left(1-\frac{v(t,\bx)}{K_{2}(\bx)}\right),\;\;t>0,\;\; \bx\in\Omega,\\
	& v(0,\bx)=v_0(\bx),\;\;\; \bx\in\Omega,\\
	& \displaystyle\frac{\partial v}{\partial \eta}=0,\;\; \bx\in \partial\Omega.
	\end{cases}
	\end{equation}
	\begin{Lemma}\label{lemma_ex3}
		  Consider the initial condition of  (\ref{equ_ex2}) be a continuous non-negative function $v_0(\bx)\in C(\Omega),\;v_0(\bx)\geq 0$ in $\Omega$ and $v_0(\bx)>0$ in some open bounded nonempty domain $\Omega_1 \subset \Omega$. Therefore there exists a unique positive solution of the system (\ref{equ_ex2}).
		\end{Lemma}
	\begin{proof}
		The proof is analogous of Lemma \ref{lemma_ex1}.
	\end{proof}
	\begin{Lemma}\label{lemma_ex4}
	 Consider the problem  (\ref{equ_ex2}), thus there exists a function $v^*(\bx)>0$ which is a unique stationary solution of (\ref{equ_ex2}). Further, for any initial condition $v_0(\bx)\geq 0,\; v_0(\bx)\not\equiv 0$ the solution $v(t, \bx)$ satisfies the condition 
	\begin{align*}
	\lim_{t\rightarrow\infty} v(t, \bx)=v^*(\bx)
	\end{align*}
	uniformly for $\bx\in \overline{\Omega}$.
			\end{Lemma}
	\begin{proof}
		The proof is analogous of Lemma \ref{lemma_ex2}.
	\end{proof}

	The final result demonstrates the existence as well as the uniqueness of solutions to a paired system (\ref{equ_2}). Note that the following proof is analogous with [\cite{Ref_6}, Theorem 5]. 
	\begin{Th}\label{lemma_ex5}
		 Let $K_1(\bx), K_2(\bx)>0$ which implies $\mu,\nu\in[0,1)$, and $r(\bx)>0$ on $\bx\in \overline{\Omega}$. When $u_0(\bx), \;v_0(\bx)\in C(\Omega)$ the model (\ref{equ_2}) has a unique solution $(u, v)$. Further, if both initial functions $u_0$ and $v_0$ are non-negative as well as nontrivial, thus $u(t, \bx) > 0$ and $v(t, \bx) > 0$ for $ t > 0$.
	\end{Th}
	\begin{proof}
		Take into account the following system with  $\mu,\nu\in[0,1)$
		\begin{equation}\label{equ_5.1}
		\begin{cases}
		& \displaystyle\frac{\partial u}{\partial t}= d_{1}\Delta u(t,\bx)+r_{1}r(\bx)u(t,\bx)\left(1-\frac{u(t,\bx)+v(t,\bx)}{K_{1}(\bx)}\right),\\
		& \displaystyle\frac{\partial v}{\partial t}=d_{2}\Delta v(t,\bx)+r_{2}r(\bx)v(t,\bx)\left(1-\frac{u(t,\bx)+v(t,\bx)}{K_{2}(\bx)}\right),\\ 
		& \displaystyle t>0,\;\;\; \bx\in\Omega,\\
		& \displaystyle\frac{\partial u}{\partial \eta}=\frac{\partial v}{\partial \eta}=0,\;\;\; \bx\in \partial\Omega,\\
		& u(0,\bx)=u_0(\bx),\;v(0,\bx)=v_0(\bx),\;\;\; \bx\in\Omega,
		\end{cases}
		\end{equation}
		where $K_{1}(\bx)=(1-\mu)K(\bx)>0,\; K_{2}(\bx)=(1-\nu)K(\bx)>0,\; r_{1}=1-\mu>0,\; r_{2}=1-\nu>0$, since $\mu,\nu\in[0,1)$.
		
		We utilize Theorem \ref{th_A.2} from Appendix \ref{A} and methods which is analogous to the proof of \cite{Ref_6}, to show existence of nontrivial time-dependent solutions. We choose the following constants
		\begin{align*}
		\rho_u >\sup_{\bx\in\Omega} u_0(\bx)>0,\;\text{and}\; \rho_v > \sup_{\bx\in\Omega} v_0(\bx)>0,
		\end{align*}
		and use the notations of Theorem \ref{th_A.2} from Appendix \ref{A} and denote
		\begin{align*}
		& f_1(t,\bx,u,v)=r_1r(\bx)u(t,\bx)\left(1-\frac{u(t,\bx)+v(t,\bx)}{K_1(\bx)}\right),\\
		& f_2(t,\bx,u,v)=r_2r(\bx)v(t,\bx)\left(1-\frac{u(t,\bx)+v(t,\bx)}{K_2(\bx)}\right).
		\end{align*}
		Then it is simple to examine that the following conditions of the theorem are satisfied 
		\begin{equation}\label{equ_5.2}
		\begin{cases}
		&	f_1(t,\bx,\rho_u,0)\leq 0\leq f_1(t,\bx,0,\rho_v),\\
		& f_2(t,\bx,0,\rho_v)\leq 0\leq f_2(t,\bx,\rho_u,0).
		\end{cases}
		\end{equation}
		The conditions (\ref{equ_5.2}) satisfy the conditions of Theorem \ref{th_A.2} from Appendix \ref{A} for the functions $f_1$  and $f_2$ deﬁned above. Therefore we arrive at the conclusion of the theorem that nontrivial $(u_0(\bx),v_0(\bx))$ such that 
		\begin{equation}\label{equ_5.3}
		(u_0,v_0)\in \mathbf{S}_\rho\equiv  \left\{\left(u_1,v_1\right)\in C \left([0,\infty) \times\overline{\Omega}\right)\times C \left([0,\infty) \times\overline{\Omega}\right) ; 0\leq u_1\leq \rho_u,\; 0\leq v_1\leq \rho_v\right\} 
		\end{equation}
		where $ C \left([0,\infty) \times\overline{\Omega}\right)$  is the class of continuous functions on $[0,\infty) \times\overline{\Omega}$, a unique solution $(u(t,\bx),v(t,\bx))$ for the system (\ref{equ_5.1}) exists and remains in $\mathbf{S}_\rho$ for all $(t,\bx)\in [0,\infty) \times\overline{\Omega}$. 
		Thus, $(u(t,\bx), v(t,\bx))$ is unique and positive solution.
		\end{proof}

Let us establish the existence result for (\ref{equ_1}) for $0\leq\nu< 1\leq\mu$. Note that the proof is analogous with \cite{Ref_66}.
\begin{Th}\label{lemma_10}
	Assume $0\leq\nu< 1\leq\mu$ and let the initial conditions be $ u_0,v_0\geq 0$, therefore the model (\ref{equ_1}) has a unique positive time-dependent nontrivial solution.
\end{Th}
\begin{proof}
	Rewrite the system (\ref{equ_1}) in the following way, yields
	\begin{equation}\label{equ_3.8}
		\begin{cases}
			& \displaystyle\frac{\partial u}{\partial t}= d_{1}\Delta u(t,\bx)+r(\bx)u(t,\bx)\left(1-\mu-\frac{u(t,\bx)+v(t,\bx)}{K(\bx)}\right),\\
			& \displaystyle\frac{\partial v}{\partial t}=d_{2}\Delta v(t,\bx)+r(\bx)v(t,\bx)\left(1-\nu-\frac{u(t,\bx)+v(t,\bx)}{K(\bx)}\right),\\ 
			& \displaystyle t>0,\;\;\; \bx\in\Omega,\\
			& \displaystyle\frac{\partial u}{\partial \eta}=\frac{\partial v}{\partial \eta}=0,\;\;\; \bx\in \partial\Omega,\\
			& u(0,\bx)=u_0(\bx),\;v(0,\bx)=v_0(\bx),\;\;\; \bx\in\Omega.
		\end{cases}
	\end{equation}
	Utilizing the Theorem \ref{th_A.2} from Appendix \ref{A} and methods which is analogous to the proof of \cite{Ref_66}, to show the existence of nontrivial time-dependent solutions. Let us choose the following constants
	\begin{align*}
	\rho_u = \max \left\{\sup_{\bx\in\Omega} u_0(\bx),1 \right\}, \;\; \rho_v =\max \left\{ \sup_{\bx\in\Omega} v_0(\bx),\; \sup_{\bx\in\Omega} K(\bx)\right\}.
\end{align*}
Note that the chosen of $\rho_u $ and $\rho_v$ analogous with [\cite{Ref_66}, Theorem 2.5]. 
Let us use the notations of Theorem \ref{th_A.2} from Appendix \ref{A} and denote
	\begin{align*}
		& f_1(t,\bx,u,v)=r(\bx)u(t,\bx)\left(1-\mu-\frac{u(t,\bx)+v(t,\bx)}{K(\bx)}\right),\\
		& f_2(t,\bx,u,v)=r(\bx)v(t,\bx)\left(1-\nu-\frac{u(t,\bx)+v(t,\bx)}{K(\bx)}\right).
	\end{align*}
	Then it is simple to examine that the following conditions of the theorem are satisfied
	
	\begin{equation}\label{equ_3.9}
		\begin{cases}
			&	f_1(t,\bx,\rho_u,0)\leq 0\leq f_1(t,\bx,0,\rho_v),\\
			& f_2(t,\bx,0,\rho_v)\leq 0\leq f_2(t,\bx,\rho_u,0).
		\end{cases}
	\end{equation}
	The conditions (\ref{equ_3.9}) satisfy the conditions of Theorem \ref{th_A.2} in Appendix \ref{A} for the functions $f_1$  and $f_2$ defined above. Therefore we arrive at the conclusion of the theorem that nontrivial $(u_0(\bx),v_0(\bx))$ such that 
	\begin{equation}\label{equ_s}
		(u_0,v_0)\in \mathbf{S}_\rho\equiv  \left\{\left(u_1,v_1\right)\in C \left([0,\infty) \times\overline{\Omega}\right)\times C \left([0,\infty) \times\overline{\Omega}\right) ; 0\leq u_1\leq \rho_u,\; 0\leq v_1\leq \rho_v\right\}, 
	\end{equation}
	where $ C \left([0,\infty) \times\overline{\Omega}\right)$  is the class of continuous functions on $[0,\infty) \times\overline{\Omega}$, a unique solution $(u(t,\bx),v(t,\bx))$ for the system (\ref{equ_3.8}) exists and remains in $\mathbf{S}_\rho$ for all $(t,\bx)\in [0,\infty) \times\overline{\Omega}$. 
	Thus, $(u(t,\bx); v(t,\bx))$ is unique and positive solution.
\end{proof}

Let us establish the existence result for (\ref{equ_1}) for $0\leq\mu< 1\leq\nu$. 
\begin{Th}\label{lemma_18.1}
	Assume $0\leq\mu< 1\leq\nu$ and the initial conditions be $u_0,v_0\geq 0$, therefore the model (\ref{equ_1}) has a unique positive time-dependent nontrivial solution.
\end{Th}
\begin{proof}
	The proof is analogous of Theorem \ref{lemma_10}.	
\end{proof}

	\section{Preliminaries}\label{Preliminaries}
	Let the following problem has stationary solution $u^*(\bx)$ where $v$ is zero in (\ref{equ_2}) 
	\begin{equation}\label{equ_2.1}
	d_{1}\Delta u^{*}(\bx)+r_{1}r(\bx)u^{*}(\bx)\left(1-\frac{u^{*}(\bx)}{K_{1}(\bx)}\right)=0,\;\;\bx\in\Omega,\;\;\; \frac{\partial u^{*}}{\partial \eta}=0,\;\; \bx\in\partial\Omega.
	\end{equation}
	 Analogously, the following problem has stationary solution $v^*(\bx)$ when $u$ is zero in (\ref{equ_2})
	\begin{equation}\label{equ_2.2}
	d_{2}\Delta v^{*}(\bx)+r_{2}r(\bx)v^{*}(\bx)\left(1-\frac{v^{*}(\bx)}{K_{2}(\bx)}\right)=0,\;\;\bx\in\Omega,\;\;\; \frac{\partial v^{*}}{\partial \eta}=0,\;\; \bx\in\partial\Omega.
	\end{equation}
	The following preliminaries results also discussed in \cite{ Ref_12,Ref_1, Ref_13, Ref_14}.
	
	\begin{Lemma}\label{lemma_2} Let $u^{*}$ be a positive solution of (\ref{equ_2.1}) and $v^{*}$ be a positive solution of (\ref{equ_2.2}), let $K_{1}(\bx)$ satisfy  $d_{1}\Delta K_{1}(\bx)\not\equiv 0$  and  $K_{2}(\bx)$ satisfy $d_{2}\Delta K_{2}(\bx)\not\equiv 0$ on $\Omega$. Thus
		\begin{equation}\label{equ_2.9}
		\int_{\Omega}r(\bx)K_{1}(\bx)\;d\bx > \int_{\Omega}r(\bx)u^{*}(\bx)\;d\bx,
		\end{equation}
	and	
		\begin{equation}\label{equ_2.10}
		\int_{\Omega}r(\bx)K_{2}(\bx)\;d\bx > \int_{\Omega}r(\bx)v^{*}(\bx)\;d\bx,
		\end{equation}
		respectively.
	\end{Lemma}

\begin{proof} 
	First, we put $v=0$ as well as $u=u^{*}$ in the first equation of (\ref{equ_2}), and utilizing the boundary conditions as well as integrating over $\Omega$, we obtain
	\begin{align}
	& r_{1}\int_{\Omega}ru^{*}\left(1-\dfrac{u^{*}}{K_{1}}\right)d\bx=0.
	\label{lemma-equn-1}
	\end{align}
{Adding and subtracting $K_1$ in equation \eqref{lemma-equn-1}, we have}
\begin{align}
    \int_{\Omega}r\left(u^{*}-K_1+K_1\right)\left(1-\dfrac{u^{*}}{K_{1}}\right)d\bx=0.
\end{align}
{Rewriting}
	\begin{align}
	 \int_{\Omega}rK_1\left(1-\dfrac{u^{*}}{K_{1}}\right)d\bx=\int_{\Omega}rK_1\left(1-\dfrac{u^{*}}{K_{1}}\right)^2d\bx>0,
	\end{align}
 which gives 
\begin{equation}\label{equ_2.11}
\int_{\Omega}rK_{1}\left(1-\dfrac{u^{*}}{K_{1}}\right)d\bx>0.
\end{equation}
{Simplifying  \eqref{equ_2.11}, we obtain} 
\begin{align*}
\int_{\Omega}r(\bx)K_{1}(\bx)d\bx > \int_{\Omega}r(\bx)u^{*}(\bx)d\bx.
\end{align*}
Analogously, the result (\ref{equ_2.10}) is justified.
\end{proof}

\begin{Lemma}\label{lemma_7.11}
		Assume $u^{*}(\bx)$ is a positive solution of (\ref{equ_2.1}). Moreover, if $K_{1}(\bx)\not\equiv const$.
		Then
		\begin{equation}\label{equ_2.12}
		\int_{\Omega}rK_{1}\left(1-\dfrac{u^{*}}{K_{1}}\right)d\bx>0.
		\end{equation}
	\end{Lemma}
	\begin{proof}
		The proof is analogous of Lemma \ref{lemma_2}.
	\end{proof}

	\begin{Lemma}\label{lemma_7.1}
		Assume $v^{*}(\bx)$ is a positive solution of (\ref{equ_2.2}). Moreover, if $K_{2}(\bx)\not\equiv const$.
		 Then
		\begin{equation}\label{equ_2.13}
		\int_{\Omega}rK_{2}\left(1-\dfrac{v^{*}}{K_{2}}\right)d\bx>0.
		\end{equation}
		\end{Lemma}
	\begin{proof}
		The proof is analogous of Lemma \ref{lemma_2}.
	\end{proof}

	\section{Stability analysis of equilibrium points}\label{stability-analysis}
Investigating the consequences of competition of two competitive species, it is crucial to stability analysis of semi-trivial equilibrium namely $(u^*(\bx), 0), \;(0, v^*(\bx))$, trivial solution $(0, 0)$ and nontrivial stationary solution which implies coexistence $(u_s, v_s)$. 
\subsection{When intrinsic growth rate transcending harvesting rate }
The following section organize by the case of intrinsic growth rate transcending harvesting rate such that $\mu,\nu \in [0,1)$. Since $E_{1}(\bx)=\mu r(\bx),\; E_{2}(\bx)=\nu r(\bx)$, which implies if $\mu,\nu \in [0,1)$ then obviously $0\leq E_1(\bx)< r(\bx)$ and $0\leq E_2(\bx)< r(\bx)$. 
In this section, we investigate two possible cases namely when $\mu\leq\nu$ and the other case when $\mu\geq \nu$. 
	
	\begin{Lemma}\label{lemma_3}
	  Assume $\mu,\nu \in [0,1)$ which implies $K_{1}(\bx), K_{2}(\bx), r_{1}>0$, and $r_{2}>0$ be positive on $\overline{\Omega}$. Therefore the trivial steady state $(0,0)$ of the model (\ref{equ_2}) is an unstable repelling equilibrium by the second definition of Theorem \ref{th_A.1} from Appendix \ref{A}.
	\end{Lemma}
    \begin{proof}
    	Let the linearized system (\ref{equ_2}) near the trivial equilibrium
    	\begin{equation}\label{equ_3}
    		\begin{cases}
    			& \displaystyle\dfrac{\partial u}{\partial t}= d_{1}\Delta u(t,\bx)+r_{1}r(\bx)u(t,\bx),\quad t>0,\;\bx\in \Omega,\vspace{-5mm}\\\\
    			& \displaystyle\dfrac{\partial v}{\partial t}=d_{2}\Delta v(t,\bx)+r_2(\bx)v(t,\bx),\quad t>0,\;\bx\in \Omega,\vspace{-5mm}\\\\
    			& \displaystyle\dfrac{\partial u}{\partial \eta}=\dfrac{\partial v}{\partial \eta}=0,\;\;\; \bx\in \partial\Omega,\\
    			& u(0,\bx)=u_0(\bx),\;v(0,\bx)=v_0(\bx),\;\;\; \bx\in\Omega.
    		\end{cases}
    	\end{equation}
    	Corresponding eigenvalue problems are given below
    	\begin{equation}\label{equ_3.111}
    		\begin{cases}
    			& \gamma \psi=d_{1}\Delta \psi+r_{1}r(\bx)\psi,\;\bx\in \Omega,\\	
    			& \sigma \phi= d_{2}\Delta \phi+r_2(\bx)\phi,\;\bx\in \Omega,\\
    			&\dfrac{\partial \psi}{\partial \eta}=\dfrac{\partial \phi}{\partial \eta}=0,\;\; \bx\in \partial\Omega.
    		\end{cases}
    	\end{equation}
    	Using variational characterization of eigenvalues according to \cite{Ref_3}, we obtain the principal eigenvalue by choosing the eigenfunction $\psi=1$ 
    	\begin{align*}
    		&\gamma_{1}\geq \dfrac{1}{|\Omega|} \int_{\Omega}r_{1}r(\bx)d\bx = \dfrac{1}{|\Omega|} \int_{\Omega}(1-\mu)r(\bx)d\bx>0,\;\; \mu\in[0,1).
    	\end{align*}
    	Analogously utilizing the variational characterization of eigenvalues according to \cite{Ref_3}, we obtain the principal eigenvalue by the eigenfunction choosing $\phi=1$ 
    	\begin{align*}
    		&\sigma_{1}\geq \dfrac{1}{|\Omega|} \int_{\Omega}(1-\nu)r(\bx)d\bx>0,\;\; \nu\in[0,1).
    	\end{align*}
    	Thus, the trivial equilibrium $(0,0)$ is unstable. 
    	Now, we prove the trivial steady state $(0,0)$ is repeller. The proof is the same as [\cite{Ref_5}, Theorem 5].
    \end{proof}

The following case demonstrates the result on the outcome of the competition when intrinsic growth rate transcending harvesting rate for $ 0\leq\mu \leq\nu< 1$.
\subsubsection{Case \texorpdfstring{$\mu\leq\nu$}{Lg}}\label{3.1}
The semi-trivial steady state $(u^*, 0)$ is unstable whenever $\mu\leq \nu$, as shown in the following lemma. 
Note that the following proof is analogous with \cite{Ref_12}.

\begin{Lemma}\label{lemma_4}
Let $\mu, \nu \in [0, 1)$ where $\mu \leq \nu$. Thus there exists $\nu_1$  for a certain $\mu$, such that for all $\nu \in [\mu, \nu_1)$, the equilibrium $(u^*, 0)$ is unstable of the system (\ref{equ_2}).
\end{Lemma}
\begin{proof}
	The analogous case $\mu = \nu$ was discussed in \cite{Ref_1}, where species have common carrying capacity. Thus, we discuss the case $\mu < \nu$. 
	Linearization of the second equation from (\ref{equ_2}) near the stationary solution $(u^*(\bx), 0)$, we have
	\begin{align*}
	& \dfrac{\partial v(t,\bx)}{\partial t} = d_{2}\Delta v(t,\bx)+r_{2}r(\bx)v(t,\bx)\left(1 -\dfrac{u^*(\bx)}{K_{2}(\bx)}\right),\;\;t>0,\;\bx\in\Omega,\\
	& v(0,\bx)=v_0(\bx),\;\;\; \bx\in\Omega,\; \dfrac{\partial v}{\partial \eta}=0,\;\; \bx\in\partial\Omega.
	\end{align*}
	The corresponding eigenvalue problem is represented as follows
	\begin{equation}\label{equ_3.1}
	\sigma\psi = d_{2}\Delta \psi+r_{2}r(\bx)\psi\left(1 -\frac{u^*(\bx)}{K_{2}(\bx)}\right) ,\;\; \bx\in \Omega,\;\; \dfrac{\partial \psi}{\partial \eta}=0,\;\; \bx\in\partial\Omega.
	\end{equation}
	The principal eigenvalue of this system is given by \cite{Ref_3}
	\begin{equation}\label{equ_3.2}
	\sigma_{1}=\sup_{\psi \neq 0, \psi\in W^{1,2}}\left\{\dfrac{-d_2 \int_{\Omega}|\nabla \psi|^{2}\;d\bx+ \int_{\Omega} r_{2}r(\bx)\psi^{2}\left(1-\dfrac{u^{*}(\bx)}{ K_2(\bx)}\right)\;d\bx}{\int_\Omega\psi^2\;d\bx}\right\}.
	\end{equation}
	We assume that $\Psi$ is the principal eigenfunction for the problem (\ref{equ_3.1}) with the principal eigenvalue $\sigma_1$. 
	This value is positive whenever the numerator of (\ref{equ_3.2}) is positive, leading to
	\begin{align}
	\sigma_{1}=\dfrac{-d_2 \int_{\Omega}|\nabla \Psi|^{2}\;d\bx+ \int_{\Omega} r_{2}r(\bx)\Psi^{2}\left(1-\dfrac{u^{*}(\bx)}{ K_2(\bx)}\right)\;d\bx}{\int_\Omega\Psi^2\;d\bx}.\label{sigma-1}
	\end{align}
{Since the denominator on the right-hand-side of equation \eqref{sigma-1} is positive, to have a positive eigenvalue, we assume}
		\begin{align*}
   	& -d_2 \int_{\Omega}|\nabla \Psi|^{2}\;dx+ \int_{\Omega} r_{2}r(\bx)\Psi^2\left(1-\dfrac{u^{*}(\bx)}{ K_2(\bx)}\right)\;d\bx>0,
   	\end{align*}
   	which gives
   	\begin{align*}
   	 \int_{\Omega} r_{2}r(\bx)\Psi^2\;d\bx> d_2 \int_{\Omega}|\nabla \Psi|^{2}\;d\bx+ \int_{\Omega} r_{2}r(\bx)\Psi^2\dfrac{u^{*}(\bx)}{ K_2(\bx)}\;d\bx.
	\end{align*}
	 Multiplying both sides by $(1-\nu),$ { this reduces to}
	\begin{align*}
 (1-\nu)\int_{\Omega}r_2 r(\bx)\Psi^2\;d\bx> (1-\nu) d_2 \int_{\Omega}|\nabla \Psi|^{2}\;d\bx+ \int_{\Omega}r_2 r(\bx)\Psi^2\dfrac{u^{*}(\bx)}{K(\bx)}\;d\bx,
 \end{align*}
 which gives
 \begin{align*}
	 1-\nu > \dfrac{r_2 d_2 \int_{\Omega}|\nabla \Psi|^{2}\;d\bx+r_2\int_{\Omega} r(\bx)\Psi^2\dfrac{u^{*}(\bx)}{K(\bx)}\;d\bx}{r_2\int_{\Omega} r(\bx)\Psi^2\;d\bx}.
	\end{align*}
	Therefore, we have
	\begin{align*}
	 1-\nu>\dfrac{ d_2 \int_{\Omega}|\nabla \Psi|^{2}\;d\bx+\int_{\Omega} r(\bx)\Psi^{2}\dfrac{u^{*}(\bx)}{K(\bx)}\;d\bx}{\int_{\Omega} r(\bx)\Psi^{2}\;d\bx}.
	\end{align*}
	Here, {in the above inequality,} the left-hand-side $(1-\nu)\in (0,1]$ whenever $\nu\in [0,1)$ and the right-hand-side is positive since numerator and denominator have square term {and} there is no negative term as the parameter $d_2$ is positive. {Therefore}, we can say that the {right-hand-side} of the above inequality belongs to $(0,1)$ since the {right-hand-side} is less than the left-hand-side and positive, where the left-hand-side of the above inequality is $(1-\nu)\in (0,1]$.
	
    Rearrange the above inequality to obtain
    \begin{align*}
    1-\dfrac{ d_2 \int_{\Omega}|\nabla \Psi|^{2}\;d\bx+\int_{\Omega} r(\bx)\Psi^{2}\dfrac{u^{*}(x)}{K(x)}\;d\bx}{\int_{\Omega} r(\bx)\Psi^{2}\;d\bx}>\nu.
    \end{align*}
    Note that, left-hand-side of the above inequality belongs to the $(0,1)$ explanation is given above. We define
\begin{align*}
 \nu_1: =1-\dfrac{ d_2 \int_{\Omega}|\nabla \Psi|^{2}\;d\bx+\int_{\Omega} r(\bx)\Psi^{2}\dfrac{u^{*}(\bx)}{K(\bx)}\;d\bx}{\int_{\Omega} r(\bx)\Psi^{2}\;d\bx},
\end{align*}	
which implies $\nu_1>\nu$ since right-side is greater than $\nu$. Hence, we obtain $\mu\leq\nu<\nu_1$. Therefore, there exists $\nu_1$  for a ﬁxed $\mu$, such that for all $\nu \in [\mu, \nu_1)$, the equilibrium $(u^*, 0)$ is unstable.
\end{proof}

In the following lemma we prove that the steady state $(0,v^*)$ is unstable whenever $\mu\leq \nu$.
\begin{Lemma}\label{lemma_5}
	 Let $\mu, \nu\in[0, 1)$ where $ \mu \leq \nu$ and there exists $\nu_1$  for a certain $\mu$, such that for all $\nu \in [\mu, \nu_1)$. Thus steady state $(0, v^*)$ is unstable of the system (\ref{equ_2}).
\end{Lemma}
\begin{proof}
	 The case $\mu = \nu$ was discussed in \cite{Ref_1}. 
	 Thus, here we only discuss the case $\mu < \nu$. Linearization of the first equation of (\ref{equ_2}) in the neighborhood of $(0, v^*)$ by the following way
	 	\begin{align*}
	 & \frac{\partial u(t,\bx)}{\partial t} = d_{1}\Delta u(t,\bx)+r_{1}r(\bx)u(t,\bx)\left(1 -\frac{v^*(\bx)}{K_{1}(\bx)}\right),\;\;\bx\in\Omega,\\
	 & u(0,\bx)=u_0(\bx),\;\;\; \bx\in\Omega,\; \frac{\partial u}{\partial \eta}=0,\;\; \bx\in\partial\Omega.
	 \end{align*}
	 The corresponding eigenvalue problem be
	 \begin{equation}\label{equ_3.3}
	 \sigma\psi = d_{1}\Delta \psi+r_{1}r(\bx)\psi\left(1 -\dfrac{v^*(\bx)}{K_{1}(\bx)}\right) ,\;\; \bx\in \Omega,\;\; \dfrac{\partial \psi}{\partial \eta}=0,\;\; \bx\in\partial\Omega.
	 \end{equation}
	 The principal eigenvalue of this problem is given by \cite{Ref_3}
	 \begin{equation}\label{equ_3.4}
	 \sigma_{1}=\sup_{\psi \neq 0, \psi\in W^{1,2}}\left\{\dfrac{-d_1 \int_{\Omega}|\nabla \psi|^{2}\;d\bx+ \int_{\Omega} r_{1}r(\bx)\psi^{2}\left(1-\dfrac{v^{*}(\bx)}{ K_1(\bx)}\right)\;d\bx}{\int_\Omega\psi^2\;d\bx}\right\}.
	 \end{equation}
	 For $\nu\in[0,1)$, take into account the eigenfunction $\psi(\bx) =\sqrt{(1-\nu)K(\bx)}=\sqrt{K_2(\bx)}$. Note that analogous eigenfunction used in [\cite{Ref_12}, Lemma 7]. Then the principle eigenvalue becomes 
	 \begin{align*}
	 \sigma_{1}\geq \dfrac{ \int_{\Omega} r_{1}r(\bx)K_2(\bx)\left(1-\dfrac{v^{*}(\bx)}{ K_1(\bx)}\right)\;d\bx}{\int_\Omega K_2(\bx) \;d\bx}.
	 	 \end{align*}
	 Note that $\nu \in [\mu,\nu_1)$, $\nu_1$ is defined in Lemma \ref{lemma_4}. We introduce a constant $c:=\dfrac{1-\mu}{1-\nu}>1$ as long as $\mu<\nu$, further it is true for every $0\leq \mu \leq \nu <1$, then it is definitely true for $0\leq \mu \leq \nu <\nu_1 <1$ due to $\nu_1 \in [0,1)$ which means $\mu, \nu,\nu_1 \in [0,1)$, it implies $c>1$ for any values in $[0,1)$. 
	 
	 Now estimate the principal eigenvalue by the following way
	 \begin{align*}
	 \sigma_{1} \geq \dfrac{ \int_{\Omega} r_{1}r(\bx)K_2(\bx)\left(1-\dfrac{v^{*}(\bx)}{(1-\mu) K(\bx)}\right)\;d\bx}{\int_\Omega K_2(\bx)\;d\bx},
	 \end{align*}
	 {which can be rewritten as}
	 \begin{align*}
	 \sigma_{1}\geq \dfrac{ \int_{\Omega}r_{1} r(\bx)K_2(\bx)\left(1-\frac{v^{*}(\bx)}{\frac{(1-\mu)}{(1-\nu)}(1-\nu)K}\right)\;d\bx}{\int_\Omega K_2(\bx)\;d\bx}.
	 \end{align*}
{Introducing the constant $c$, we have}
	 
 \begin{equation}\label{equ_4.5}
 	\sigma_{1}\geq \dfrac{ r_1\int_{\Omega} r(\bx)K_2(\bx)\left(1-\dfrac{v^{*}(\bx)}{c K_2(\bx)}\right)\;d\bx}{\int_\Omega K_2(\bx)\;d\bx}.
 \end{equation}
	 We want to show the numerator of the {right-hand-side fraction in (\ref{equ_4.5}) is} positive. From Lemma \ref{lemma_7.1}, we have 
	 \begin{align*}
	 	\int_{\Omega} r(\bx)K_2(\bx) \left(1-\dfrac{v^{*}(\bx)}{K_2(\bx)}\right) d\bx>0.
	 \end{align*}
 Note that since $c>1$, we obtain
 \begin{align*}
 	\int_{\Omega} r(\bx) K_2(\bx) \left(1-\dfrac{v^{*}(\bx)}{cK_2(\bx)}\right) d\bx> 
 		\int_{\Omega} r(\bx) K_2(\bx) \left(1-\dfrac{v^{*}(\bx)}{K_2(\bx)}\right) d\bx>0.
 \end{align*}
Therefore, numerator of the inequality (\ref{equ_4.5}) is positive. Thus, the principal eigenvalue is positive
 \begin{align*}
 	\sigma_{1}\geq \dfrac{ r_1\int_{\Omega} r(\bx)K_2(\bx)\left(1-\dfrac{v^{*}(\bx)}{c K_2(\bx)}\right)\;d\bx}{\int_\Omega K_2(\bx)\;d\bx}>0.
 \end{align*}
which completes the proof. 
\end{proof}
In the following theorem we prove that the equilibrium $(u_s,v_s)$ is globally stable for the system (\ref{equ_2}) whenever $\mu\leq \nu$ using Lemma \ref{lemma_3}, Lemma \ref{lemma_4}, and Lemma \ref{lemma_5}.

\begin{Th}\label{th_1}
	Let $\mu, \nu \in [0, 1)$ where $\mu \leq \nu$. Thus there exists $\nu_1$  for a certain $\mu$, such that for all $\nu \in [\mu, \nu_1)$, the equilibrium $(u_s,v_s)$ of the system (\ref{equ_2}) is globally stable.
\end{Th}
\begin{proof}
{We consider} $0\leq \mu \leq\nu < 1$. Lemma \ref{lemma_4} demonstrates that there is a number $\nu_1\in[0,1)$ such that whenever $\nu\in [\mu,\nu_1)$ the steady state $(u^*,0)$ is unstable. At the same time, Lemma \ref{lemma_5} illustrates that the steady state $(0,v^*)$ is unstable. Lemma \ref{lemma_3} demonstrates that the trivial steady state $(0,0)$ is unstable, moreover repeller. We extract two options of Theorem \ref{th_A.1} in Appendix \ref{A}. Hence, there exists a globally stable coexistence solution, which confirms the first statement of Theorem \ref{th_A.1} from Appendix \ref{A}. 
\end{proof}

The next case demonstrates the result on the outcome of the competition when growth function exceeding harvesting for $ 0\leq\nu \leq\mu< 1$.

\subsubsection{Case \texorpdfstring{$\mu \geq \nu$}{Lg}}\label{3.2}
This subsection contains lemmata that are symmetrical are proved in Subsection \ref{3.1}. Hence, we ignore the proofs and instead we mention the corresponding lemmata in Subsection \ref{3.1}. In the following lemma we prove that the steady state $(0,v^*)$ is unstable whenever $\mu\geq \nu$.
\begin{Lemma}\label{lemma_7}
	Let $\mu \geq \nu$, where $\mu, \nu\in [0,1)$. There exists a value $\mu_1$  for a certain $\nu$, such that for all $\mu \in [\nu, \mu_1)$, the steady state $(0,v^*(\bx))$ of the model (\ref{equ_2}) is unstable.
	\end{Lemma}
\begin{proof}
 The proof is analogous to Lemma \ref{lemma_4}. Basically, Lemma \ref{lemma_7} demonstrates that there is a range of values for $\mu$ when $\nu$ is fixed and such that $\mu \geq \nu$, where $(0, v^*(\bx))$ is unstable. 
\end{proof}

The following lemma proves that the steady state $(u^*,0)$ is unstable whenever $\mu\geq \nu$.
\begin{Lemma}\label{lemma_8}
	Let $\mu \geq \nu$, where $\mu, \nu\in [0,1)$, thus there exists a value $\mu_1$  for a certain $\nu$, for all $\mu \in [\nu, \mu_1)$. Thus the steady state $(u^*(\bx), 0)$ of the system (\ref{equ_2}) is unstable.
\end{Lemma}
\begin{proof}
	The proof is analogous of Lemma \ref{lemma_5}. This, Lemma \ref{lemma_8} represents the steady state $(u^*,0)$ is unstable whenever $\mu \geq \nu$.
\end{proof}

In the following theorem we prove that the coexistence solution $(u_s,v_s)$ of the system (\ref{equ_2}) is globally stable whenever $\mu\geq \nu$ using Lemma \ref{lemma_3}, Lemma \ref{lemma_7}, and Lemma \ref{lemma_8}.

\begin{Th}\label{th_3}
Let $\mu \geq \nu$, where $\mu, \nu\in [0,1)$. Thus there exists a value $\mu_1$  for a certain $\nu$, for all $\mu \in [\nu, \mu_1)$, the coexistence steady state of the system (\ref{equ_2}) is a globally stable.
\end{Th}
\begin{proof}
	The proof is analogous with Theorem \ref{th_1}. Take into account $0\leq \nu \leq\mu < 1$. By Lemma \ref{lemma_7} there exists a value $\mu_1$ for all $\mu\in (\nu, \mu_1)$ the solution $(0,v^*)$ is unstable. At the same time, Lemma \ref{lemma_8} shows that the $(u^*,0)$ is unstable whenever $\mu\geq \nu$. Moreover, Lemma \ref{lemma_3} demonstrates that the steady state $(0,0)$ is unstable and repeller. This excludes two respective options in Theorem \ref{th_A.1} from Appendix \ref{A}. Thus, $(u_s,v_s)$ is a globally stable.
\end{proof}

\subsection{When one harvesting rate transcending intrinsic growth rate}\label{4.2}
In this section, we examine the outcomes of two competitive species when one harvesting rate in the system (\ref{equ_1}) overpass respective intrinsic growth rates which means there are two possible scenarios can arise namely, $E_1(\bx) \geq r(\bx)$ or $E_2(\bx) \geq r(\bx)$ such that $0\leq \nu < 1\leq\mu$ or $ 0\leq\mu < 1\leq\nu,$ respectively.

First, we depict the result on the impact of competition when one harvesting function exceeds the respective intrinsic growth function for the case $ 0\leq\nu < 1\leq\mu$.
\subsubsection{Case \texorpdfstring{$\nu < 1\leq\mu$}{Lg}}\label{3.3}
The following lemma shows that there is no coexistence state when $ 0\leq\nu < 1\leq\mu$.
\begin{Lemma}\label{lemma_11}
	Suppose $0\leq\nu< 1\leq\mu$ thus there is no nontrivial stationary solution $(u_s,v_s)$ for the model (\ref{equ_3.8}) as well as (\ref{equ_1}).
	\end{Lemma}
	\begin{proof}
	 Take into account that there is a nontrivial stationary solution $(u_s(\bx),v_s(\bx))$ where $u_s\geq 0,\;v_s\geq 0$ for all $\bx\in\Omega$. The coexistence solution is to satisfy the following system of equations 
		 \begin{equation}\label{equ_3.10}
		 \begin{cases}
		 & 0= d_{1}\Delta u_s(\bx)+r(\bx)u_s(\bx)\left(1-\mu-\dfrac{u_s(\bx)+v_s(\bx)}{K(\bx)}\right),\;\bx\in \Omega,\\
		 & 0=d_{2}\Delta v_s(\bx)+r(\bx)v_s(\bx)\left(1-\nu-\dfrac{u_s(\bx)+v_s(\bx)}{K(\bx)}\right),\;\bx\in\Omega,\\
		 & \displaystyle\dfrac{\partial u_s}{\partial \eta}=\dfrac{\partial v_s}{\partial \eta}=0,\;\; \bx\in \partial\Omega.		
		 \end{cases}
		 \end{equation}
		Now, integrating the first equation over $\Omega$ and utilizing the boundary conditions, yields
		\begin{align*}
		\int_\Omega ru_s\left(1-\mu-\dfrac{u_s+v_s}{K}\right)\; d\bx=0.
		\end{align*}
		The integrand is non-positive for all $\bx \in\Omega$ whenever $\mu\geq 1$ and $u_s\not\equiv 0$ (which holds by our assumption on $u_s$  being a nontrivial coexistence solution). Assume that $\mu = 1$, then, since $u_s\not\equiv 0$,  the integrand is non-positive, unless $u_s+v_s\equiv 0$  which cannot happen for a
		nontrivial non-negative coexistence solution, hence contradiction. Next, let $\mu>1$, and if $u_s+v_s\equiv K(1-\mu)$, the system (\ref{equ_3.10}) becomes
		\begin{align}
	&	\begin{cases}
		&  d_{1}\Delta u_s(\bx)+r(\bx)u_s(\bx)\left(1-\mu-\dfrac{K(1-\mu)}{K}\right)=0,\;\bx\in \Omega,\\
		& d_{2}\Delta v_s(\bx)+r(\bx)v_s(\bx)\left(1-\nu-\dfrac{K(1-\mu)}{K}\right)=0,\;\bx\in\Omega,\\
		& \dfrac{\partial u_s}{\partial \eta}=\dfrac{\partial v_s}{\partial \eta}=0,\;\; \bx\in \partial\Omega.
		\end{cases}
		\end{align}
		Simplifying
	\begin{align}
		\begin{cases}
		&  d_{1}\Delta u_s(\bx)=0,\;\;\bx\in\Omega,\\
		& \dfrac{\partial u_s}{\partial \eta}=0,\;\; \bx\in \partial\Omega,\\
		& d_{2}\Delta v_s(\bx)+r(\bx)v_s(\bx)\left(\mu-\nu\right)=0,\;\;\bx\in\Omega,\\ 
		& \dfrac{\partial v_s}{\partial \eta}=0,\;\; \bx\in \partial\Omega.
		\end{cases}
		\end{align}
		This leads to the solution $u_s\equiv const$ on $\overline{\Omega}$ by the maximum principle \cite{Ref_7}. Integrating the second equation utilizing the boundary conditions, we get
		\begin{align*}
		\int_\Omega r(\bx)v_s(\bx)\left(\mu-\nu\right)d\bx=0,
		\end{align*}
		which is not true unless $v_s(\bx)$ is trivial, leading to $u_s\equiv K(1-\mu)$ and contradicting the assumption on the pair $(u_s,v_s)$ being nontrivial. Hence, there is no coexistence state $(u_s,v_s)$, which proves the lemma.
	\end{proof}

Next, we delineate only possible nontrivial stationary solution for the system
(\ref{equ_3.8}) is $(0, v^*)$ for any nontrivial non-negative initial conditions.
\begin{Lemma}\label{lemma_not use}
	Let $\mu \geq 1$, then $(0,v^*(\bx))$ is the only nontrivial stationary solution to (\ref{equ_3.8}) as well as (\ref{equ_1}).
	\end{Lemma}
\begin{proof}
	 We assume that there exists a nontrivial steady state other than $(0, v^*(\bx))$. Since there is no coexistence in the system by Lemma \ref{lemma_11}, the other possible solution of such type is $(u^*(\bx), 0)$ where $u^*(\bx) \geq 0$ on $\Omega$ and satisfies the following boundary value problem for $\mu=1$
	 	\begin{align*}
 &	 \begin{cases}
	 &  d_{1}\Delta u^*(\bx)-r(\bx)u^*(\bx)\dfrac{u^*(\bx)}{K(\bx)}=0,\;\; \bx\in\Omega,\\
	 & \displaystyle\dfrac{\partial u^*}{\partial \eta}=0,\;\; \bx\in \partial\Omega.
	 \end{cases}
	 \end{align*}
	 Now, integrating and utilizing the boundary condition, yields
	  \begin{align*}
	 \int_\Omega r(\bx)\dfrac{(u^*(\bx))^2}{K(\bx)}\;d\bx=0,
	 \end{align*}
	 which is not true for a nontrivial $u^*(\bx) \geq 0$. Therefore we arrive at a contradiction, and the only nontrivial stationary solution is $(0, v^*(\bx))$ where the function $v^*(\bx)$ satisfies the second equation by Lemma \ref{lemma_7.1}. Same procedure is applicable for $\mu>1$. Thus, $(0, v^*(\bx))$ is the only nontrivial stationary solution of (\ref{equ_1}) for $\mu\geq 1$.
\end{proof}

The following lemma proves that $(0, 0)$ of the system (\ref{equ_3.8}) as well as (\ref{equ_1}) is unstable but is not a repeller by the second definition of Theorem \ref{th_A.1} from Appendix \ref{A}, when the harvesting rate $E_1(\bx)$ surpasses or equal to the intrinsic growth rate $r(\bx)$. Note that the proof is analogous with [\cite{Ref_6}, Theorem 9].

\begin{Lemma}\label{lemma_12}
{Consider the case} $ 0\leq \nu < 1\leq\mu$. Thus, the trivial steady state $(0,0)$ of the model (\ref{equ_3.8}) as well as (\ref{equ_1}) is unstable, but is not a repeller by the second definition of Theorem \ref{th_A.1} from Appendix \ref{A}.
	\end{Lemma}
\begin{proof}
	 First, we assume $\mu > 1$ and linearized the system  (\ref{equ_2}) near the trivial equilibrium
	 \begin{align}
	\begin{cases}
	 & \displaystyle\dfrac{\partial u}{\partial t}= d_{1}\Delta u+r_1r(\bx)u,\;\; t>0,\;\bx\in\Omega,\vspace{-3ex}\\\\
	 & \displaystyle\dfrac{\partial v}{\partial t}=d_{2}\Delta v+r_2r(\bx)v,\;\; t>0,\;\bx\in\Omega,\vspace{-3ex}\\\\
	 & \displaystyle\dfrac{\partial u}{\partial \eta}=\dfrac{\partial v}{\partial \eta}=0,\;\;\; \bx\in \partial\Omega,\\
	 & u(0,\bx)=u_0(\bx),\;v(0,\bx)=v_0(\bx),\;\;\; \bx\in\Omega.	
	 \end{cases} 
	 \end{align}
	 The corresponding eigenvalue problems are
	 \begin{equation}\label{equ_3.11}
	\begin{cases}
	 & \gamma \psi=d_{1}\Delta \psi+r_1r(\bx)\psi,\;\;\bx\in\Omega,\\	
	 & \sigma \phi= d_{2}\Delta \phi+r_2r(\bx)\phi,\;\;\bx\in\Omega,\\	
	 &\dfrac{\partial \psi}{\partial \eta}=\dfrac{\partial \phi}{\partial \eta}=0,\;\; \bx\in \partial\Omega.
	 \end{cases}
	 \end{equation}
	 Consider $\psi_1$ and $\phi_1$ be two eigenfunctions (that can be chosen positive) and corresponding principal eigenvalues of (\ref{equ_3.11}) $\gamma_1$ and $\sigma_1$, respectively \cite{Ref_3}. Integrating
	 (\ref{equ_3.11}) using the boundary condition, yields
	 \begin{align*}
	 \gamma_1= \dfrac{\int_\Omega r_1r(\bx)\psi_1\;d\bx}{\int_\Omega \psi_1\;d\bx},
	 \end{align*}
	 which implies
	 \begin{equation}\label{equ_3.12}
 \gamma_1= \dfrac{\int_\Omega (1-\mu)r(\bx)\psi_1\;d\bx}{\int_\Omega \psi_1\;d\bx}<0,\;\;\mu>1,
	 \end{equation}
	 and
	 \begin{align*}
	 \sigma_1= \dfrac{\int_\Omega r_2r(\bx)\phi_1\;d\bx}{\int_\Omega \phi_1\;d\bx},
	 \end{align*}
	 implies
	 \begin{equation}\label{equ_3.13}
	  \sigma_1= \dfrac{\int_\Omega (1-\nu)r(\bx)\phi_1\;d\bx}{\int_\Omega \phi_1\;d\bx}>0,\;\;\nu<1,
	 \end{equation}
	 respectively. Thus, the steady state $(0,0)$ is unstable. For the first equation of (\ref{equ_3.8}) note that when $\mu> 1$ parameters are negative. By Lemma \ref{lemma_10}, the time-dependent solutions $(u(t,\bx),v(t,\bx))$ are positive for $u_0\not\equiv 0$ or $v_0\not\equiv 0$. We recall $K_1(\bx)=(1-\mu)K(\bx)$ and establish the following inequality whenever $\mu>1$ 
	 \begin{align*}
	 1-\dfrac{u(t,\bx)+v(t,\bx)}{K_1(\bx)}=1+\dfrac{u(t,\bx)+v(t,\bx)}{\mid K_1(\bx)\mid}	\geq 1.
	 \end{align*}
	 Multiplying each side by $r_1$ whenever $\mu>1$, where $r_1=1-\mu$, we obtain
	 \begin{align*}
	 	 r_1\left(1-\frac{u(t,\bx)+v(t,\bx)}{K_1(\bx)}\right)=r_1\left(1+\frac{u(t,\bx)+v(t,\bx)}{\mid K_1(\bx)\mid}\right)	\leq r_1.
	 \end{align*}
	 Thus, we obtain from first equation in (\ref{equ_2})
	 \begin{align*}
	& \frac{\partial u}{\partial t}= d_{1}\Delta u+r_1r(\bx)u\left(1-\frac{u(t,\bx)+v(t,\bx)}{K_{1}(\bx)}\right)\leq d_{1}\Delta u+r_1r(\bx)u.
	 \end{align*}
	 Therefore,
	 \begin{align*}
	 & \frac{\partial u}{\partial t}\leq d_{1}\Delta u+r_1r(\bx)u(t,\bx),\\
	 & \frac{\partial v}{\partial t}\geq d_{2}\Delta v+r_2r(\bx)v(t,\bx)\left(1-\frac{u(t,\bx)+v(t,\bx)}{K_{2}(\bx)}\right).
	 \end{align*}
	 Now, integrating {over $\Omega$} and utilizing the boundary condition, yields
	 \begin{align*}
	 & \frac{ d}{ dt}\int_\Omega u(t,\bx)\;d\bx\leq \int_\Omega r_1r(\bx)u(t,\bx)\;d\bx,\\
	 & \frac{d}{dt}\int_\Omega v(t,\bx)\;d\bx\geq \int_\Omega r_2r(\bx)v(t,\bx)\left(1-\frac{u(t,\bx)+v(t,\bx)}{K_{2}(\bx)}\right)\;d\bx.
	 \end{align*}
		We consider the positive numbers 
	$	0<\rho \leq \inf\limits_{\bx\in\Omega} r_2 r(\bx)\left(1-\dfrac{2\delta}{K_2(\bx)}\right) $
	and
	$	0< \delta \leq \inf\limits_{\bx\in\Omega} \left(\dfrac{K_2(\bx)}{4}\right) $
	(see [\cite{Ref_5} Theorem 5, \cite{Ref_6} Theorem 9]) such that for initial conditions satisfying $u_0(\bx)+v_0(\bx)<\delta,\; u_0\not\equiv 0,\; v_0\not\equiv 0,\; u_0\geq 0,$ and $ v_0\geq 0$, yields
	 \begin{align*}
	 \frac{d}{dt}\int_\Omega v(t,\bx)\;d\bx>\int_\Omega r_2r(\bx)v(t,\bx)\left(1-\frac{2\delta}{K_{2}(\bx)}\right)\;d\bx.
	 \end{align*}
	Finally, we get
	 \begin{align*}
	 \frac{d}{dt}\int_\Omega v(t,\bx)\;d\bx>\rho \int_\Omega v(t,\bx)\;d\bx.
	 \end{align*}
	 Utilizing the Gr\"onwall inequality from Theorem \ref{th_A.4} in Appendix \ref{A}, yields
	 \begin{align*}
	 \int_\Omega v(t,\bx)\;d\bx\geq e^{\rho t} \int_\Omega v(0,\bx)\;d\bx,
	 \end{align*}
	 where $t > 0$. Note that $\rho$ is positive which implies the integral on the right side grows exponentially. 
	 Now, consider the first equation
	 \begin{align*}
	 \frac{d}{dt}\int_\Omega u(t,\bx)\;d\bx\leq\int_\Omega r_1r(\bx)u(t,\bx)\;d\bx.
	 \end{align*}
	 Since $r_{1}r(\bx)<0$ whenever $\mu>1$, there exists a real number
	 $\varepsilon =\sup\limits_{\bx\in\Omega}r_{1}r(\bx)< 0,$ for all $\mu> 1$ (see [\cite{Ref_6} Theorem 9, \cite{Ref_66} Theorem 3.4])
	 such that $r_{1}r(\bx)<-\mid \varepsilon\mid<0$ which yields
	  \begin{align*}
	 & \frac{d}{dt}\int_\Omega u(t,\bx)\;d\bx \leq\int_\Omega r_1r(\bx)u(t,\bx)\;d\bx <-\mid \varepsilon\mid \int_\Omega u(t,\bx)\;d\bx.
	  \end{align*}
	  Now, utilizing the Gr\"onwall inequality (see Theorem \ref{th_A.4} in Appendix \ref{A}), yields
	  \begin{align*}
	  \int_\Omega u(t,\bx)\;d\bx\leq e^{-\mid \varepsilon\mid t} \int_\Omega u(0,\bx)\;d\bx.
	  \end{align*}
	  In the right-hand-side { of the above equation}, there is an exponential term which converges to zero {as time grows}. Thus, the solution $(0, 0)$ is repelling in $v(t, \bx)$ and attracting in $u(t, \bx)$ which does not satisfy the second definition of Theorem \ref{th_A.1} from Appendix \ref{A}.
	  
	  Now, take into account $\mu = 1$, instability of $(0, 0)$ follows from the inequality (\ref{equ_3.13}). The first equation of (\ref{equ_3.8}) becomes 
	    \begin{align}
	  \begin{cases}
	  & \displaystyle\dfrac{\partial u}{\partial t}= d_{1}\Delta u(t,\bx)+r(\bx)u(t,\bx)\left(1-1-\dfrac{u(t,\bx)+v(t,\bx)}{K(\bx)}\right),\\
	  & \displaystyle\dfrac{\partial v}{\partial t}=d_{2}\Delta v(t,\bx)+r(\bx)v(t,\bx)\left(1-\nu-\dfrac{u(t,\bx)+v(t,\bx)}{K(\bx)}\right),\\ 
	  & \displaystyle t>0,\;\;\; \bx\in\Omega,\\
	  & \displaystyle\dfrac{\partial u}{\partial \eta}=\dfrac{\partial v}{\partial \eta}=0,\;\;\; \bx\in\ \partial\Omega,\\
	  & u(0,\bx)=u_0(\bx),\;v(0,\bx)=v_0(\bx),\;\;\; \bx\in\Omega,
	  \end{cases}
	  \end{align}
	  which implies
	  \begin{equation}\label{equ_3.14}
	  \begin{cases}
	  & \displaystyle\dfrac{\partial u}{\partial t}= d_{1}\Delta u(t,\bx)-r(\bx)u(t,\bx)\left(\dfrac{u(t,\bx)+v(t,\bx)}{K(\bx)}\right),\\
	  & \displaystyle\dfrac{\partial v}{\partial t}=d_{2}\Delta v(t,\bx)+r(\bx)v(t,\bx)\left(1-\nu-\dfrac{u(t,\bx)+v(t,\bx)}{K(\bx)}\right),\\ 
	  & \displaystyle t>0,\;\;\; \bx\in\Omega,\\
	  & \displaystyle\dfrac{\partial u}{\partial \eta}=\dfrac{\partial v}{\partial \eta}=0,\;\;\; \bx\in\ \partial\Omega,\\
	  & u(0,\bx)=u_0(\bx),\;v(0,\bx)=v_0(\bx),\;\;\; \bx\in\Omega.
	  \end{cases}
	  \end{equation}
	  The rest of the proof follows the same procedures which are discussed above and we omit this proof for $\mu=1$. Therefore, for $ 0\leq \nu < 1\leq\mu$ the steady state $(0,0)$ is unstable but not repeller.
\end{proof}

This next result shows global asymptotic stability for the steady state $(0, v^*(\bx))$ of the system (\ref{equ_1}) when the harvesting coefficient satisfies $0\leq\nu < 1\leq \mu$ using Lemma \ref{lemma_11}, Lemma \ref{lemma_not use}, and Lemma \ref{lemma_12}. 

\begin{Th}\label{th_5}
Let $0\leq\nu<1\leq \mu$. Thus the stationary solution $(0,v^*(\bx))$ of the system (\ref{equ_3.8}) as well as (\ref{equ_1}) be globally asymptotically stable.		
	\end{Th}
\begin{proof}
{From the} Lemma \ref{lemma_12}, the solution $(0, 0)$ of this model is unstable. At the same time, {from the} Lemma \ref{lemma_11}, there is no coexistence solution. The remaining non-negative steady state is the solution $(0, v^*(\bx))$, see Lemma \ref{lemma_not use}. This solution is unique by uniqueness of $v^*(\bx)$, see Lemma \ref{lemma_ex3} and Lemma \ref{lemma_ex4}. Recall the definition of $\mathbf{S}_\rho$ from (\ref{equ_s})
	\begin{align*}
	\mathbf{S}_\rho\equiv  \left\{\left(u_1,v_1\right)\in C \left([0,\infty) \times\overline{\Omega}\right)\times C \left([0,\infty) \times\overline{\Omega}\right) ; 0\leq u_1\leq \rho_u,\; 0\leq v_1\leq \rho_v\right\} 
	\end{align*}
	where
	\begin{align*}
	\rho_u = \max \left\{\sup_{\bx\in\Omega} u_0(\bx),1 \right\}, \;\; \rho_v =\max \left\{ \sup_{\bx\in\Omega} v_0(\bx),\; \sup_{\bx\in\Omega} K(\bx)\right\}.
\end{align*}
 Utilizing the Theorem \ref{th_A.3} from Appendix \ref{A}, we obtain the time dependent solution $(u(t, \bx), v(t, \bx))$ of (\ref{equ_3.8}) as well as (\ref{equ_1}) will converge to the unique equilibrium $(0, v^*(\bx))$ for any initial condition from $\mathbf{S}_\rho$, which complete the proof.
	\end{proof}
Now, we demonstrate the result on the outcome of the competition when one harvesting function exceeds respective intrinsic growth rates for $ \mu < 1\leq\nu$.

\subsubsection{Case \texorpdfstring{$\mu < 1\leq\nu$}{Lg}}\label{3.4}
This subsection contains lemmata which are symmetrical and proven in Subsection \ref{3.3}. Therefore, we ignore the proofs and instead mention to corresponding lemmata from Subsection \ref{3.3}. 

Investigating the case when harvesting rate $E_2(\bx)$ surpasses or identical to the intrinsic growth rate $ r(\bx)$ for all $\bx \in \Omega$.

In the following lemma, we prove that there exists no coexistence whenever $0\leq \mu < 1\leq\nu$.
\begin{Lemma}\label{lemma_18}
	Let $ 0\leq\mu < 1\leq\nu$, there exists no coexistence solution of the system (\ref{equ_1}). 
	\end{Lemma}
	\begin{proof}
	The proof is analogous to Lemma \ref{lemma_11}.
		\end{proof}

	Next, we will show that the only possible nontrivial stationary solution for the system (\ref{equ_1}) is $( u^*,0)$ for any nontrivial non-negative initial conditions.
	\begin{Lemma}\label{lemma_not use_1}
		Assume $\nu\geq 1$, thus $(u^*(\bx),0)$ is the only nontrivial steady state of the model (\ref{equ_1}).
	\end{Lemma}
	\begin{proof}
		The proof is analogous of Lemma \ref{lemma_not use} with using Lemma \ref{lemma_7.11}. 
	\end{proof}

The following lemma proves that $(0, 0)$ of the system (\ref{equ_1}) is unstable but is not a repeller whenever $ 0\leq \mu < 1\leq\nu$ when the harvesting rate exceeds or equal to the intrinsic growth rate.

\begin{Lemma}\label{lemma_19}
	Assume $ 0\leq \mu < 1\leq\nu$. Thus, the steady state $(0,0)$ of the model (\ref{equ_1}) is unstable, but is not a repeller by the second definition of Theorem  \ref{th_A.1} from Appendix \ref{A}.
\end{Lemma}
\begin{proof}
	The proof is analogous of Lemma \ref{lemma_12} by the second definition of Theorem \ref{th_A.1} from Appendix \ref{A}.
\end{proof}

In the following theorem we demonstrates that global asymptotic stability for the semi-trivial steady state $(u^*,0)$ of the model (\ref{equ_1}) when the harvesting rate satisfies $0\leq\mu < 1\leq\nu$ using Lemma \ref{lemma_18}, Lemma \ref{lemma_not use_1} and Lemma \ref{lemma_19}.
\begin{Th}\label{th_6}
	Let $0\leq\mu<1\leq\nu$. Thus the steady state $(u^*(\bx),0)$ of the model (\ref{equ_1}) is globally asymptotically stable.
	\end{Th}
\begin{proof}
 Lemma \ref{lemma_19} shows that the solution $(0,0)$ is unstable. At the same, Lemma \ref{lemma_18} shows that there is no coexistence solution. The remaining non-negative steady state is the solution $(u^*(\bx),0)$, see Lemma \ref{lemma_not use_1}. 
 Now, utilizing Theorem \ref{th_A.3} from Appendix \ref{A}, we
see that the time-dependent solution $(u(t, \bx), v(t, \bx))$ of (\ref{equ_3.8}) as well as (\ref{equ_1}) with $\nu>1$ will converge to the unique steady state $( u^*(\bx),0)$ for any initial condition from $\mathbf{S}_\rho$. The proof is complete.
\end{proof}
\subsection{When both harvesting rate transcending intrinsic growth rate}
In this section, we examine the case when both harvesting rates transcending intrinsic growth rates namely $\mu,\nu\geq 1$.

\subsubsection{Case \texorpdfstring{$\mu,\nu\geq 1$}{Lg}}\label{3.5}
In the following theorem we demonstrates that global asymptotic stability for the steady state $(0,0)$ using Lemma \ref{lemma_12} from Subsection \ref{3.3} (or, Lemma \ref{lemma_19} from Subsection \ref{3.4} ).

\begin{Th}\label{th_7}
	Let $ \mu, \nu\geq 1$. Thus the trivial solution $(0, 0)$ of the model (\ref{equ_1}) is globally asymptotically stable.
	\end{Th}
\begin{proof}
	The reasoning from the proof of Lemma \ref{lemma_12} (or, Lemma \ref{lemma_19}) applies here directly and it shows convergence to the trivial solution. The proof is complete.
\end{proof}

\section{Numerical results}\label{Numerical-experiments}
In this section, we represent  numerical experiments using finite element method to support the theoretical results.  {The usual $L^2(\Omega)$ inner product are denoted by  $(.,.)$. We define the Hilbert space for our problem as $$X:=H^1(\Omega)=\big\{u\in L^2(\Omega):\nabla u\in L^2(\Omega)^{n}\big\}.$$
The conforming finite element space is denoted by $X_h\subset X$, and we assume a
regular triangulation $\tau_h(\Omega)$, where $h$ is the maximum triangle diameter. We consider the following fully-discrete, decoupled and linearized scheme of the system \eqref{equ_1}:}

 {\begin{algorithm}[H]\label{Algn1}
  \caption{Fully discrete and decoupled ensemble scheme} Given time-step $\Delta t>0$, end time $T>0$, initial conditions $u^0$, $v^0\in X_h$. Set $M=T/\Delta t$ and for $n=1,\cdots\hspace{-0.35mm},M-1$, compute:
 Find $u_{h}^{n+1}\in X_h$ satisfying, for all $\chi_h\in X_h$:
\begin{align}
    \left(\frac{u_{h}^{n+1}-u_{h}^{n}}{\Delta t},\chi_h\right)=-d_1\big(\nabla u_{h}^{n+1},\nabla \chi_h\big)&+\bigg(r(\bx) u_{h}^{n+1} \left(1-\frac{u_{h}^n+v_{h}^n}{K(\bx)}\right), \chi_h\bigg)\nonumber\\&-\left(\mu r(\bx) u_{h}^{n+1},\chi_h\right). \label{disweak1}
\end{align}
Find $v_{h}^{n+1}\in X_h$ satisfying, for all $l_h\in X_h$:
\begin{align}
    \left(\frac{v_{h}^{n+1}-v_{h}^{n}}{\Delta t},l_h\right)=-d_2\big(\nabla v_{h}^{n+1},\nabla l_h\big)&+\bigg( r(\bx) v_{h}^{n+1} \left(1-\frac{u_{h}^n+v_{h}^n}{K(\bx)}\right), l_h\bigg)\nonumber\\&-\left(\nu r(\bx) v_{h}^{n+1},l_h\right).\label{disweak2}
\end{align}
\end{algorithm}}

{For all experiments, we consider the diffusion coefficients $d_{1} = d_{2}=1$, a unit square domain $\Omega=(0,1)\times(0,1)$, $P_2$ finite element, and structured triangular meshes. We define the energy of the system at time $t$ for the species density $u$, and $v$ as $$\frac12\int_\Omega u^2(t,\bx)d\bx,\hspace{1mm}\text{and}\hspace{1mm}\frac12\int_\Omega v^2(t,\bx)d\bx,$$ respectively. The 2D code is written in Freefem++ \cite{hecht2012new}.}

{\subsection{Stationary carrying capacity}
In this section, we will consider stationary carry capacity together with both constant and space-dependent intrinsic growth rates.
\subsubsection{Experiment 1: Constant intrinsic growth rate}
In this experiment, we consider the carrying capacity of the system $$K(\bx)\equiv 2.1+\cos(\pi x)\cos(\pi y),$$ and a constant intrinsic growth rate $r(\bx)\equiv 1.2$. We run several simulations for various values of the harvesting coefficients $\mu$, and $\nu$. In Figures \ref{exp-1-mu-1-5-nu-0-08}-\ref{energy-mu-0-0009-nu-varies-bet-0-0005-to-0-0025}, we considered the initial population densities $u_0=v_0=1.8$ with time-step size $\Delta t=0.1$.
\begin{figure} [H]
		\centering
		\subfloat[]{\includegraphics[scale=.18]{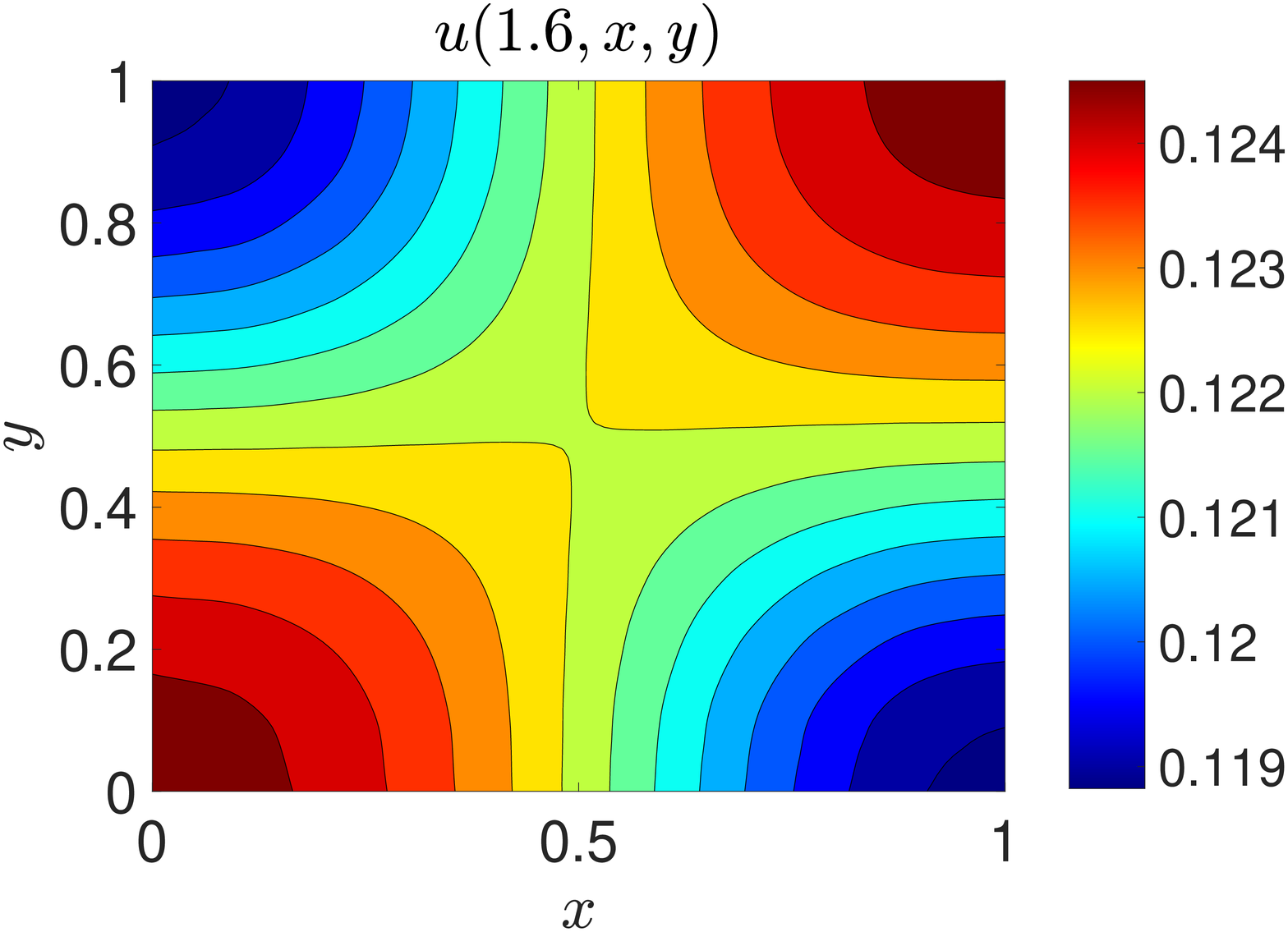}}
		\subfloat[]{\includegraphics[scale=.18]{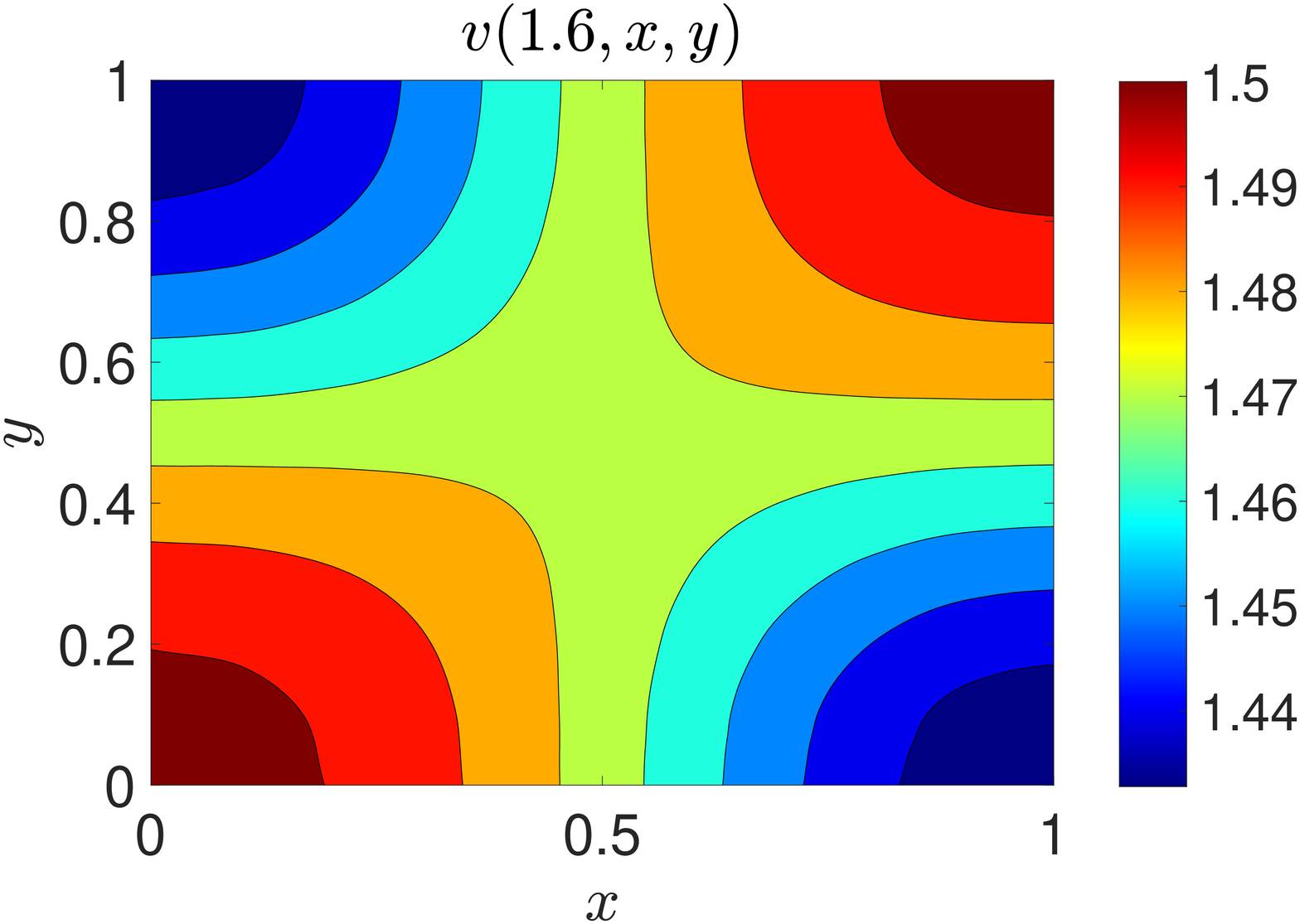}}
		\caption{Population density (a) $u(t,\bx)$, and (b) $v(t,\bx)$ at time $t=1.6$ with the harvesting coefficients $\mu=1.5$, and $\nu=0.08$.}
		\label{exp-1-mu-1-5-nu-0-08}
	\end{figure}
	In Figure \ref{exp-1-mu-1-5-nu-0-08}, we represent the contour plot of the species density $u$, and $v$ at time $t=1.6$ with fixed harvesting coefficients $\mu=1.5$, and $\nu=0.08$. A co-existence is observed at the moment.}
	
{We also plot the energy of the system for the species density $u$, and $v$ versus time for three different combinations of the harvesting coefficient pairs $(\mu,\nu)$ in Figure \ref{energy-mu-nu-varies-bet-1-5-to-0-08}. We consider the harvesting parameter $\mu=1.5>\nu=0.08$ in Figure \ref{energy-mu-nu-varies-bet-1-5-to-0-08}(a) and thus observe the species $u$ dies away shortly but the species $v$ survives. A opposite scenario is observed in Figure \ref{energy-mu-nu-varies-bet-1-5-to-0-08}(b) where $\mu=0.08<\nu=1.15$ is considered. This is because one harvesting coefficient is significantly bigger than the other and exceeds the intrinsic growth rate, that is why one species extincts in a short period of time. The results in Figure \ref{energy-mu-nu-varies-bet-1-5-to-0-08} (a), and Figure \ref{energy-mu-nu-varies-bet-1-5-to-0-08} (b) support the Theorem \ref{lemma_10}, and the Theorem \ref{lemma_18.1}, respectively.  In Figure \ref{energy-mu-nu-varies-bet-1-5-to-0-08}(c), though the harvesting coefficients are the same ($\mu=\nu=1.5$) both exceeds the intrinsic growth rate and thus an extinction in both species is observed in short-time evolution. 
\begin{figure} [H]
		\centering
		\subfloat[]{\includegraphics[scale=.15]{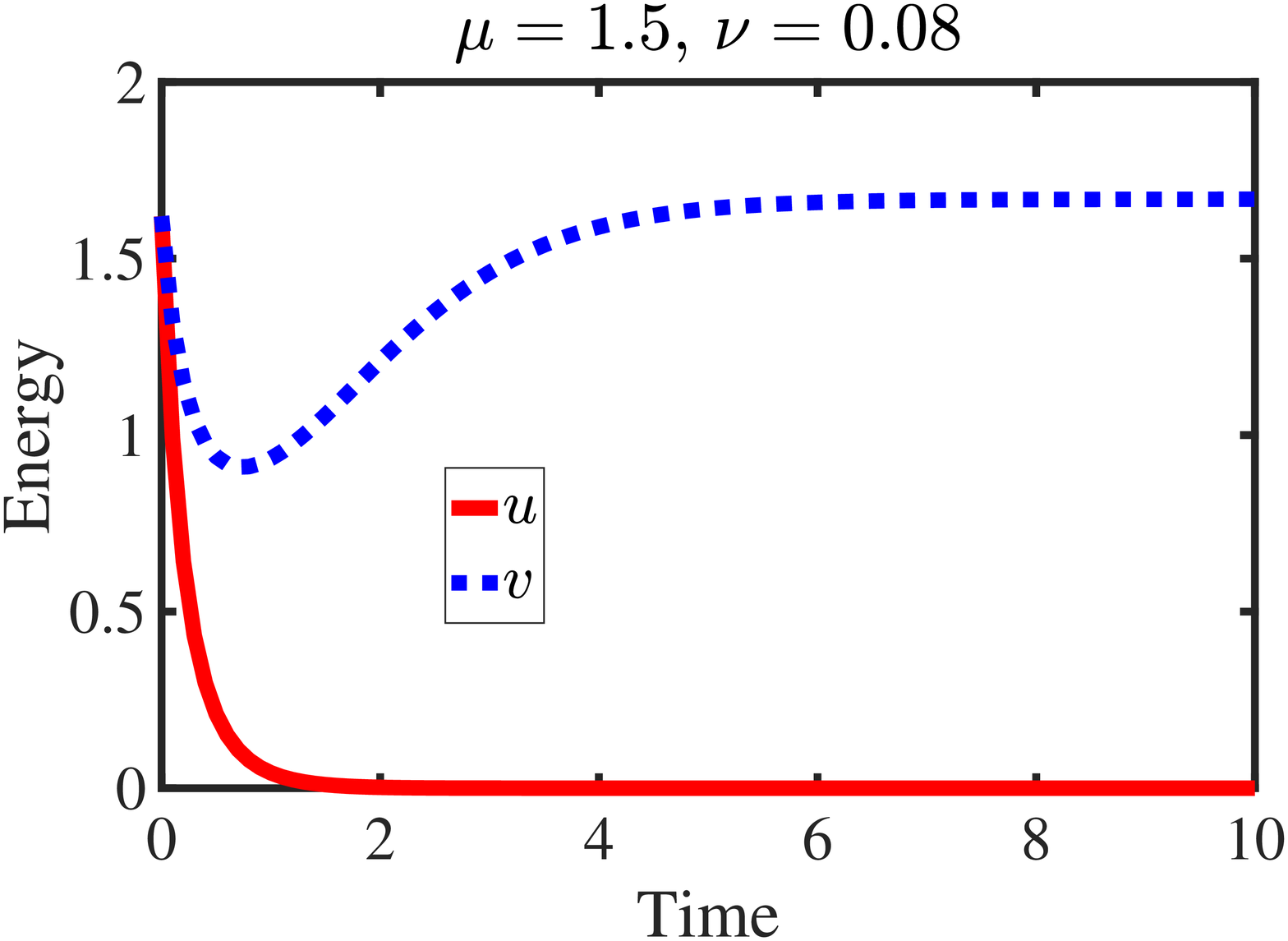}}
		\subfloat[]{\includegraphics[scale=.15]{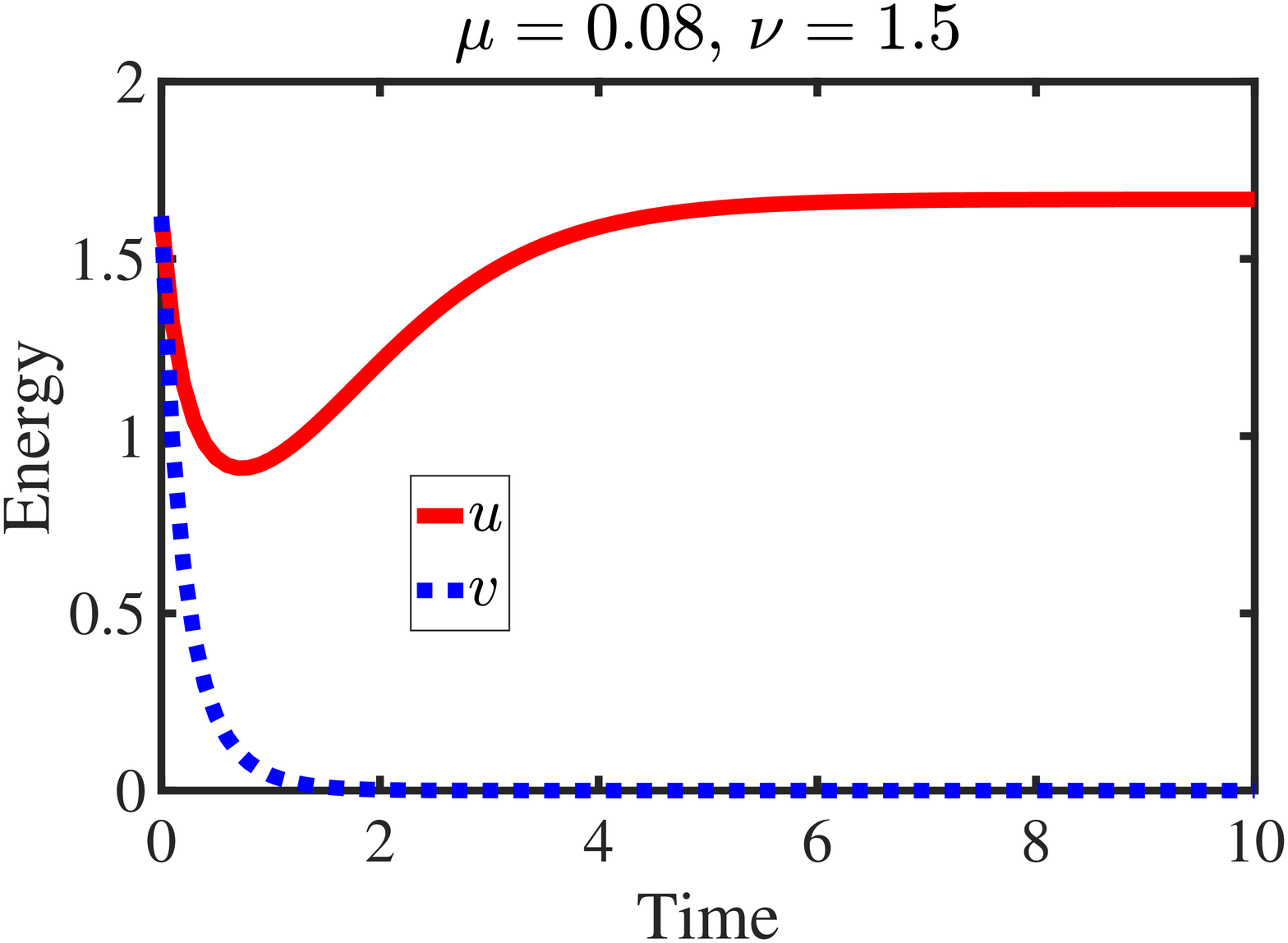}}
		\subfloat[]{\includegraphics[scale=.15]{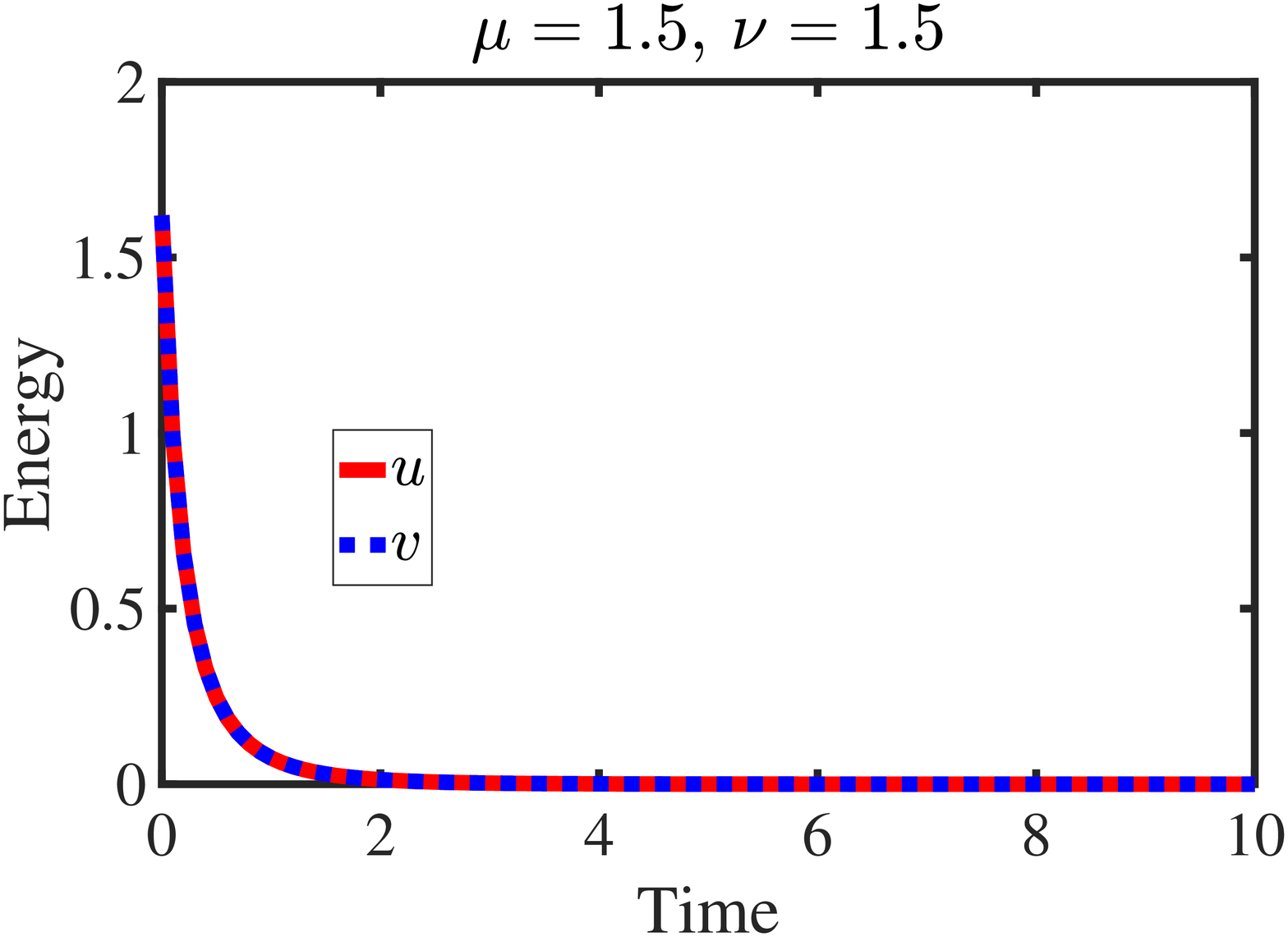}}
		\caption{Evolution of system energy for species density $u$, and $v$ with (a) $\mu=1.5$, and  $\nu=0.08$, (b) $\mu=0.08$, and $\nu=1.5$, and (c) $\mu=1.5$, and $\nu=1.5$.}
		\label{energy-mu-nu-varies-bet-1-5-to-0-08}
	\end{figure}
	\begin{figure}[H]
		\centering
		\subfloat[]{\includegraphics[scale=.16]{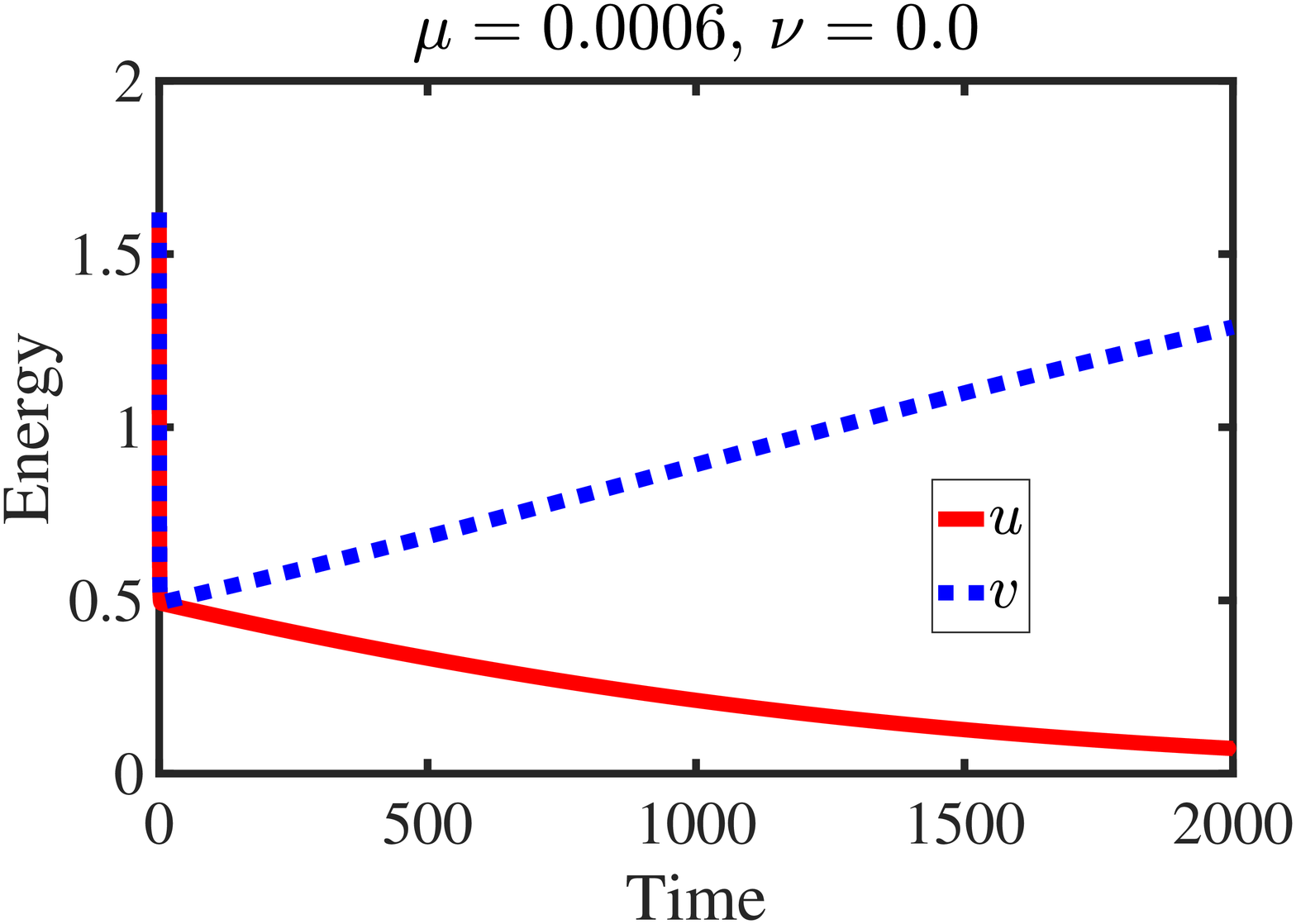}}
		\subfloat[]{\includegraphics[scale=.16]{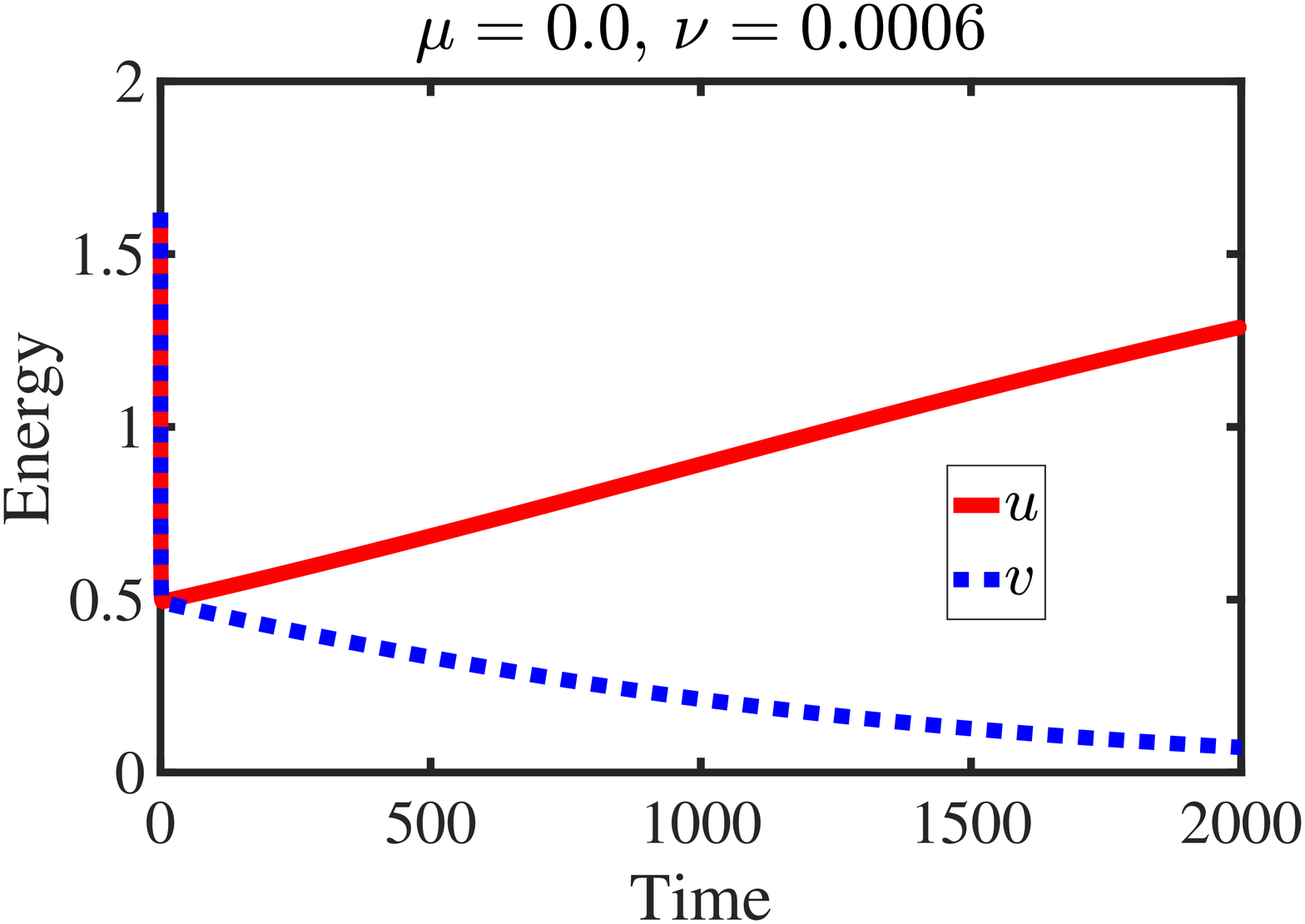}}
		\caption{{System energy for the species density $u$, and $v$ versus time for (a) $\mu=0.0006$, and $\nu=0.0$, and (b) $\mu=0.0$, and $\nu=0.0006$.}}\label{nu-0-0006-mu-0-0}
	\end{figure}
In Figure \ref{nu-0-0006-mu-0-0}, we plot the energy of the system  versus time corresponding to the species density $u$, and $v$ with the coefficients of harvesting (a) $\mu=0.0006$, and $\nu=0.0$, and (b) $\mu=0.0$, and $\nu=0.0006$. We observe the harvesting impact as an extinction of the species $u$ in (a), and the species $v$ in (b).}

{In Figure \ref{energy-mu-0-0009-nu-varies-bet-0-0005-to-0-0025}, we plot the energy of the system corresponding to the both species versus time keeping fixed the harvesting parameter $\mu=0.0009$ but varies $\nu$. We run the simulation until $t=2000$ for each cases. In Figure \ref{energy-mu-0-0009-nu-varies-bet-0-0005-to-0-0025} (a), since $\mu>\nu$, as time grows, the species density for $v$ remains always bigger than that for $u$, whereas, the scenario is opposite in Figures \ref{energy-mu-0-0009-nu-varies-bet-0-0005-to-0-0025} (b)-(f) because of $\mu<\nu$. A possible co-existence is exhibited in Figure \ref{energy-mu-0-0009-nu-varies-bet-0-0005-to-0-0025} (b), which supports the Theorem \ref{th_1}.}

\begin{figure} [H]
		\centering
		\subfloat[]{\includegraphics[scale=.15]{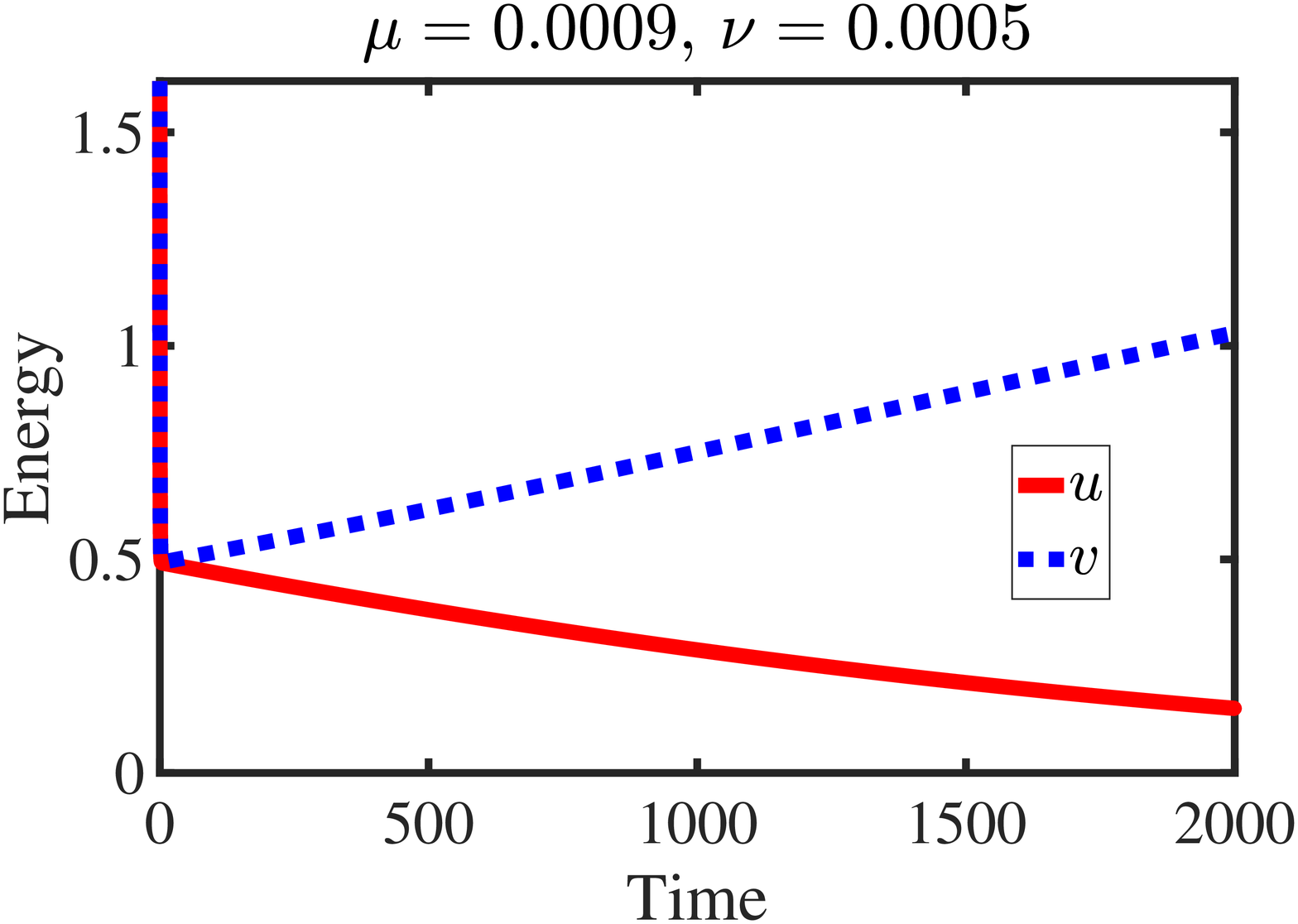}}
		\subfloat[]{\includegraphics[scale=.15]{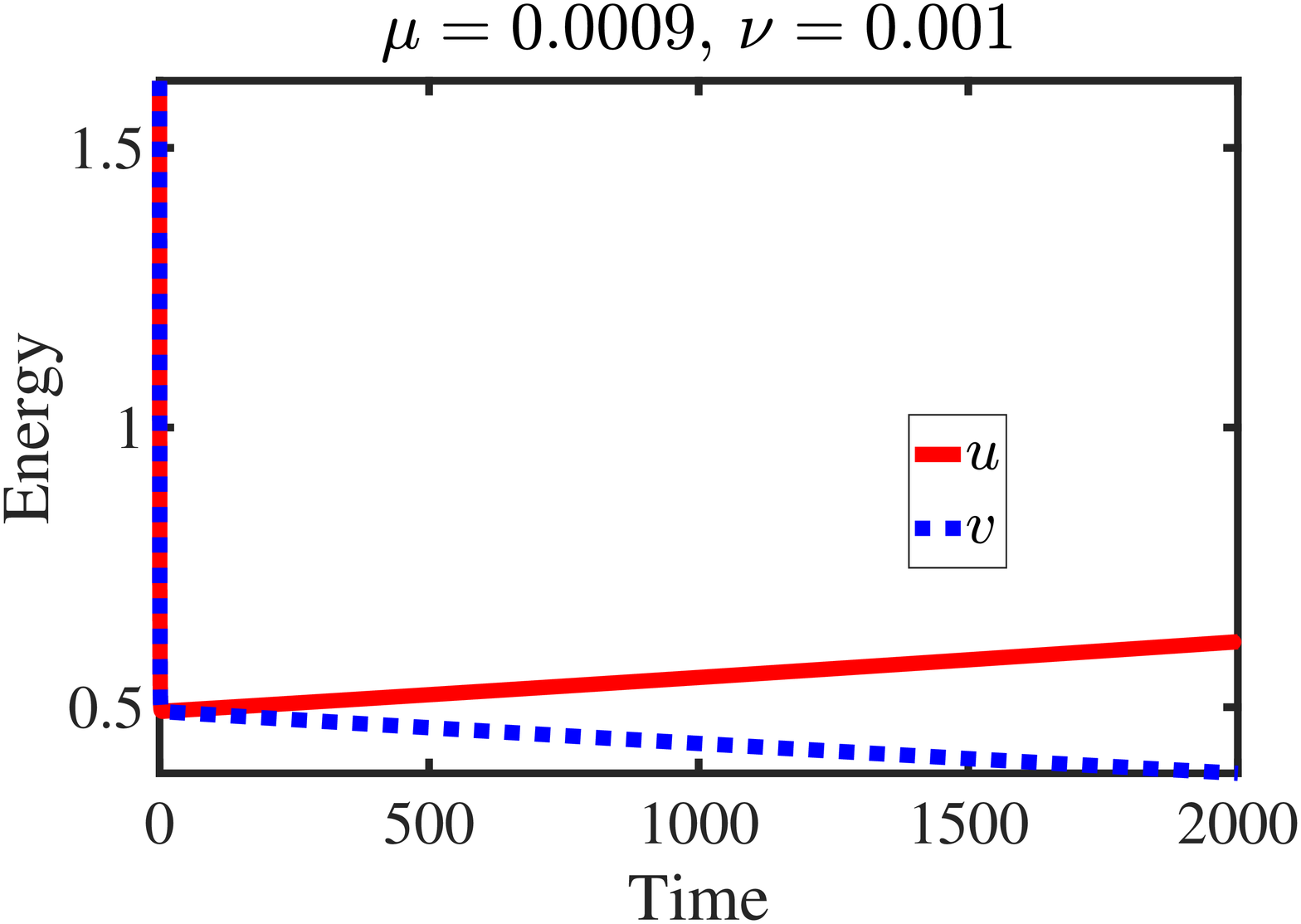}}\subfloat[]{\includegraphics[scale=.15]{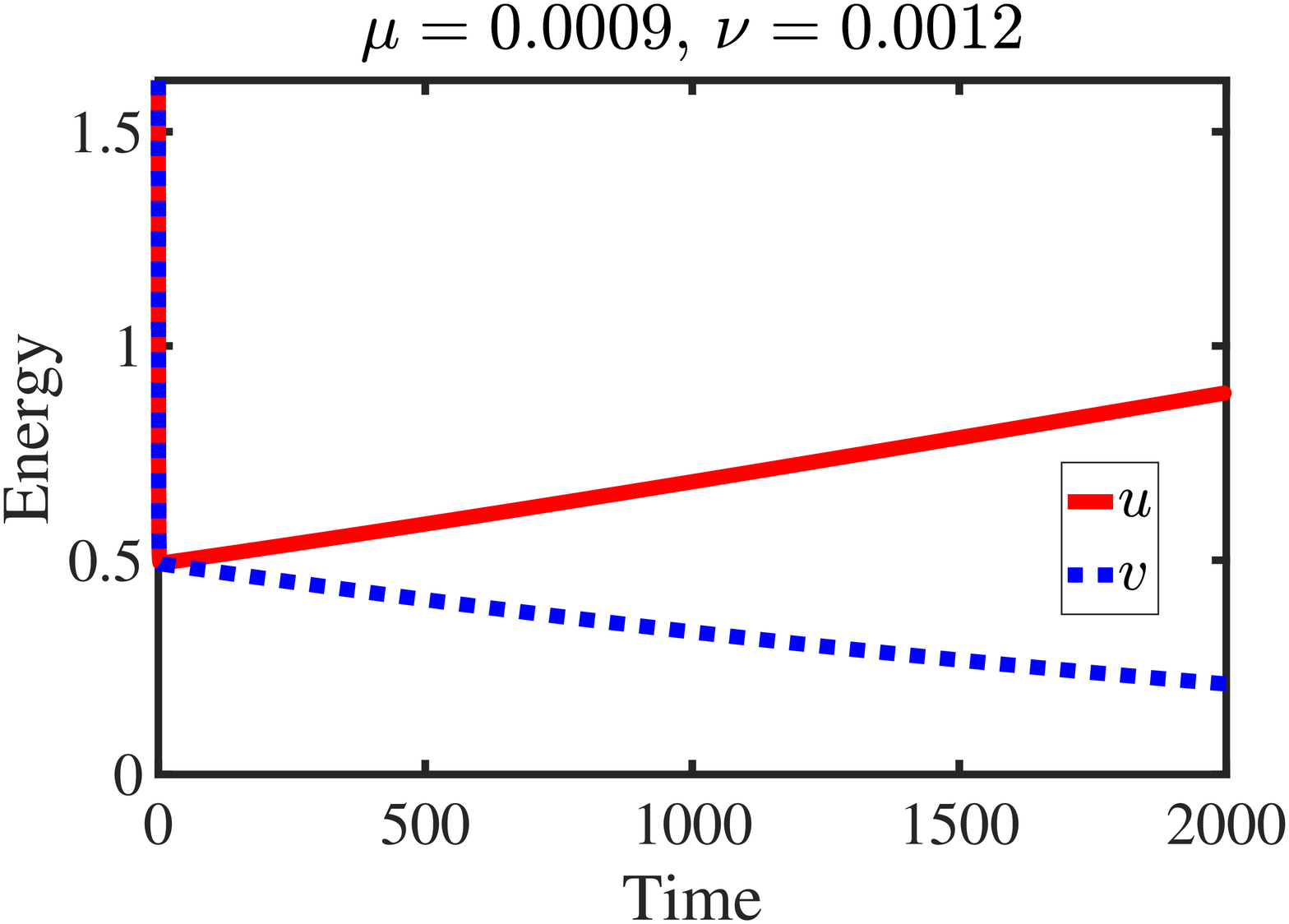}}~~\\\vspace{-2ex}
		\subfloat[]{\includegraphics[scale=.15]{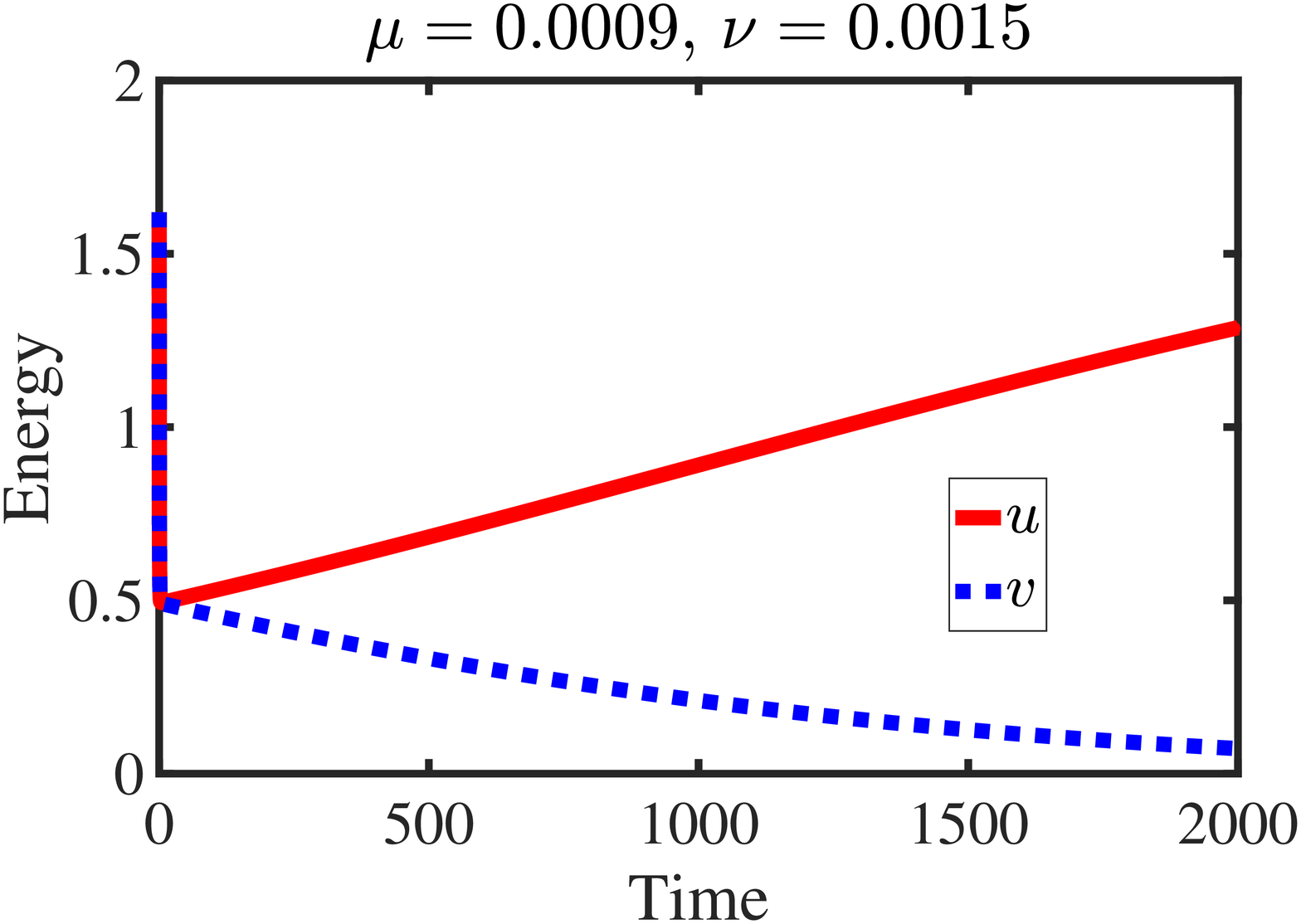}}
		\subfloat[]{\includegraphics[scale=.15]{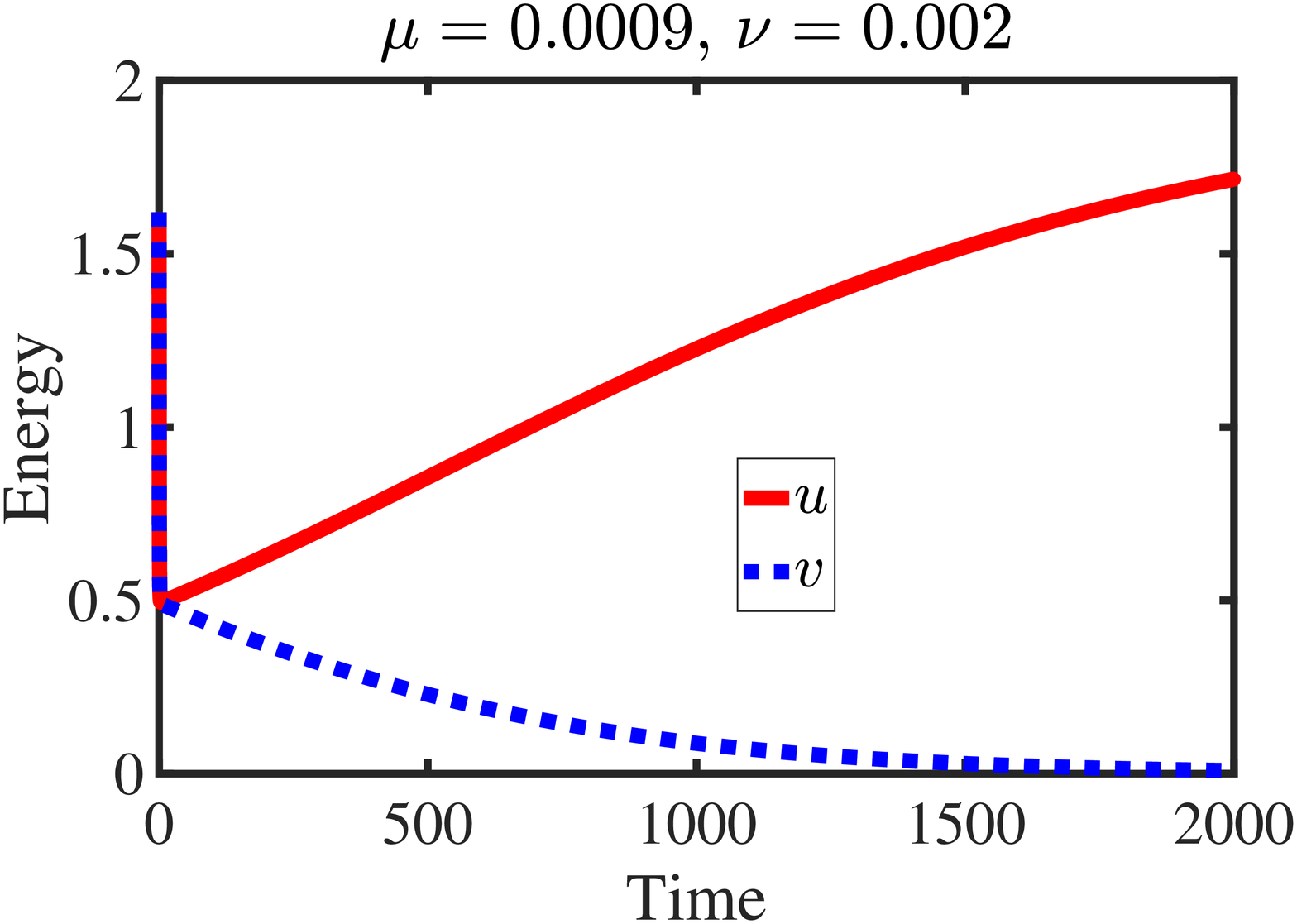}}
		\subfloat[]{\includegraphics[scale=.15]{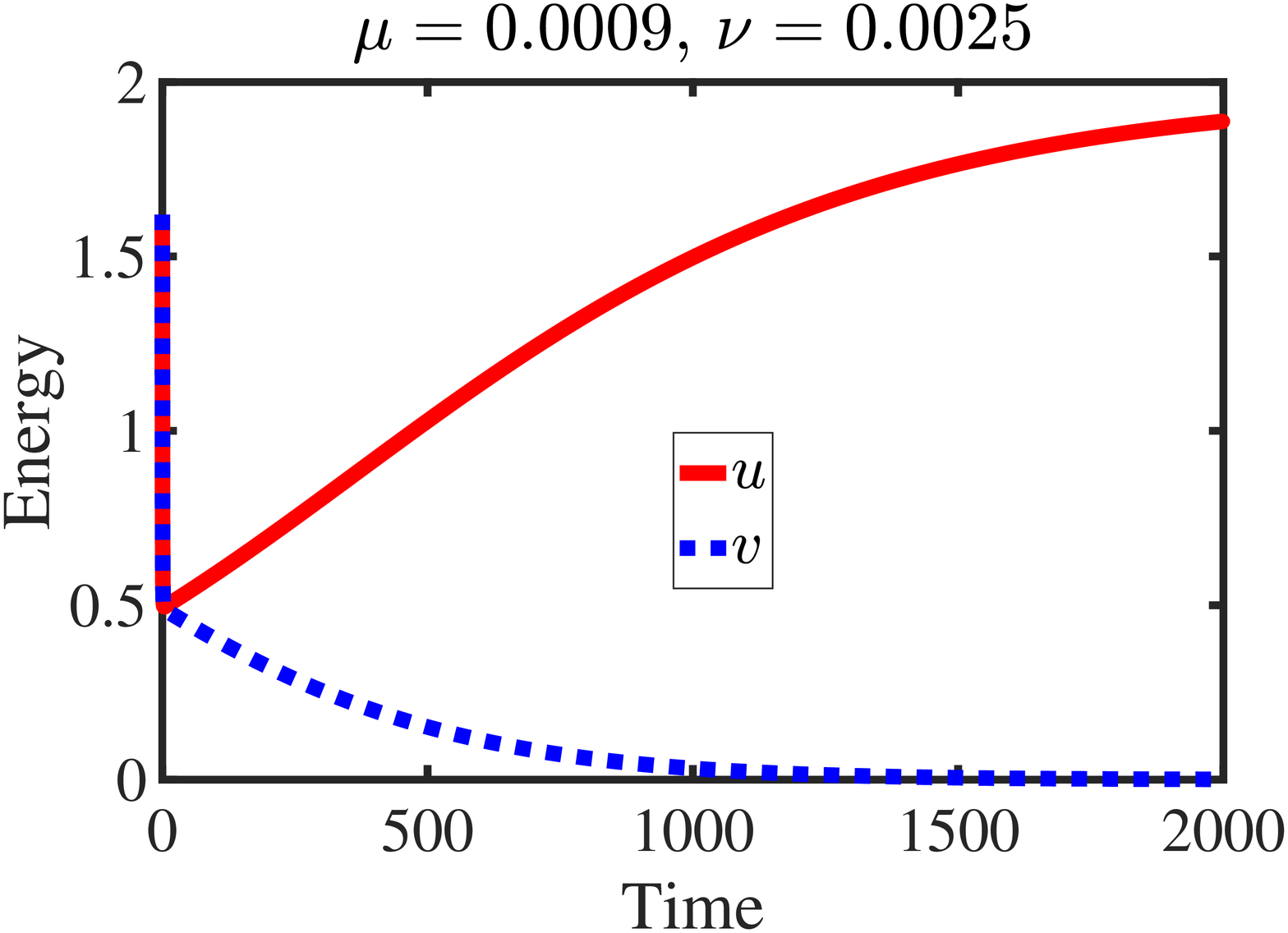}}~~
		\caption{Evolution of system energy for species density $u$, and $v$ with $\mu=0.0009$, (a) $\nu=0.0005$, (b)  $\nu=0.001$, (c)  $\nu=0.0012$, (d) $\nu=0.0015$, (e) $\nu=0.002$, and (f) $\nu=0.0025$.}
		\label{energy-mu-0-0009-nu-varies-bet-0-0005-to-0-0025}
	\end{figure}
	
	\begin{figure} [H]
		\centering
		\includegraphics[scale=.16]{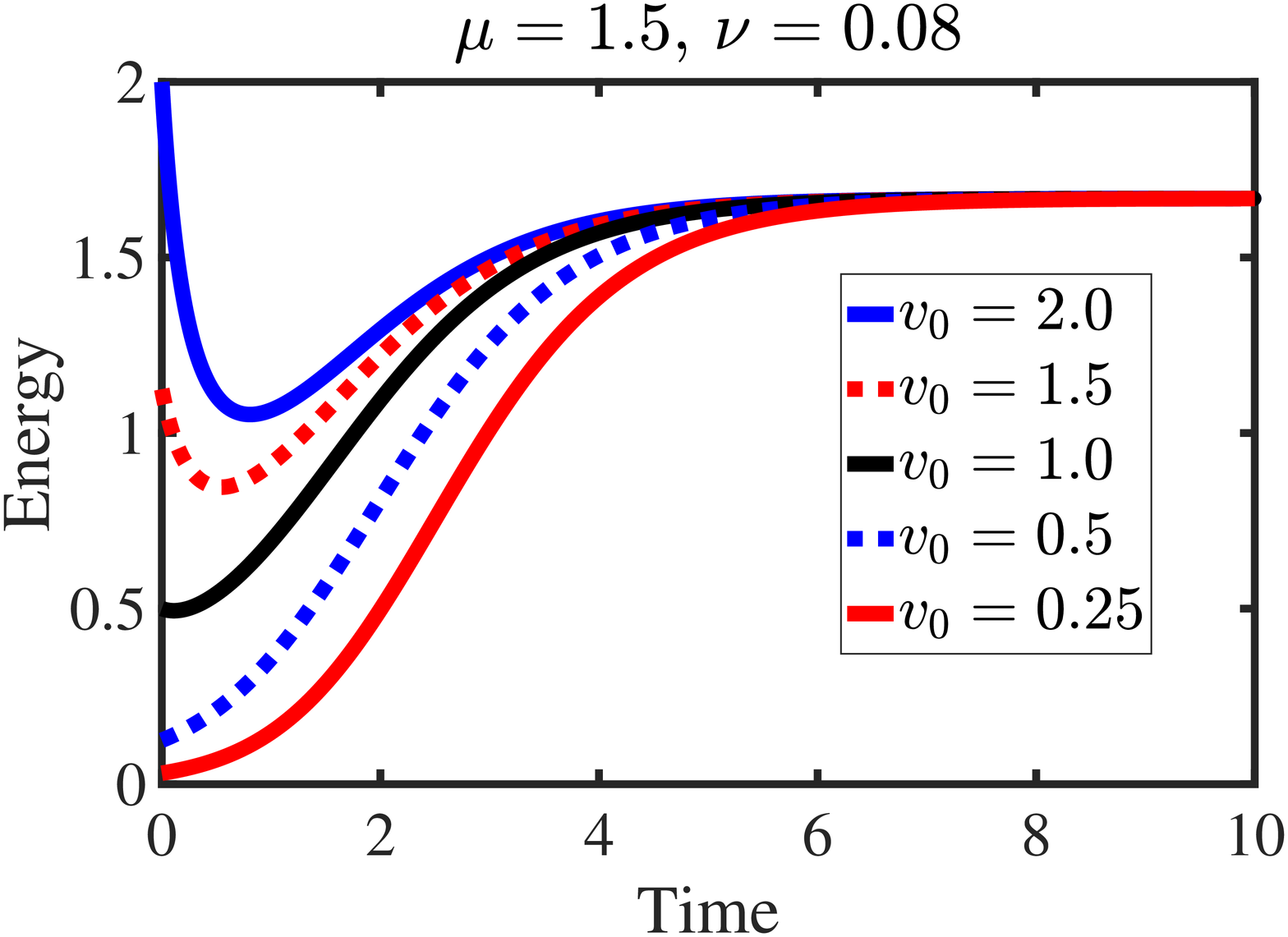}
		\caption{Stable solution $v(t,\bx)$.}
		\label{stability-v}
	\end{figure}
	{In Figure \ref{stability-v}, we plot the population density $v(t,\bx)$ versus time for various values of the initial condition $v_0$. In all the cases, we consider the initial densities for both species same, and omitted the results for $u$. We observe a unique solution as time grows if the initial conditions are positive.}

\subsubsection{Experiment 2: Space dependent intrinsic growth rate}
{In this experiment, we consider the carrying capacity, and the intrinsic growth rate as $$K(\bx)\equiv 2.5+\sin(x)\sin(y),\hspace{1mm}\text{and}\hspace{1mm}r(\bx)\equiv 1.5+\cos(x)\cos(y),$$
respectively, along with the equal initial population densities $u_0=v_0=1.2$. The system energy versus time is plotted until $t=3000$ in Figures \ref{var-r-space-depn} (a)-(b) for two different harvesting parameters pairs.}
\begin{figure} [H]
		\centering
		\subfloat[]{\includegraphics[scale=.165]{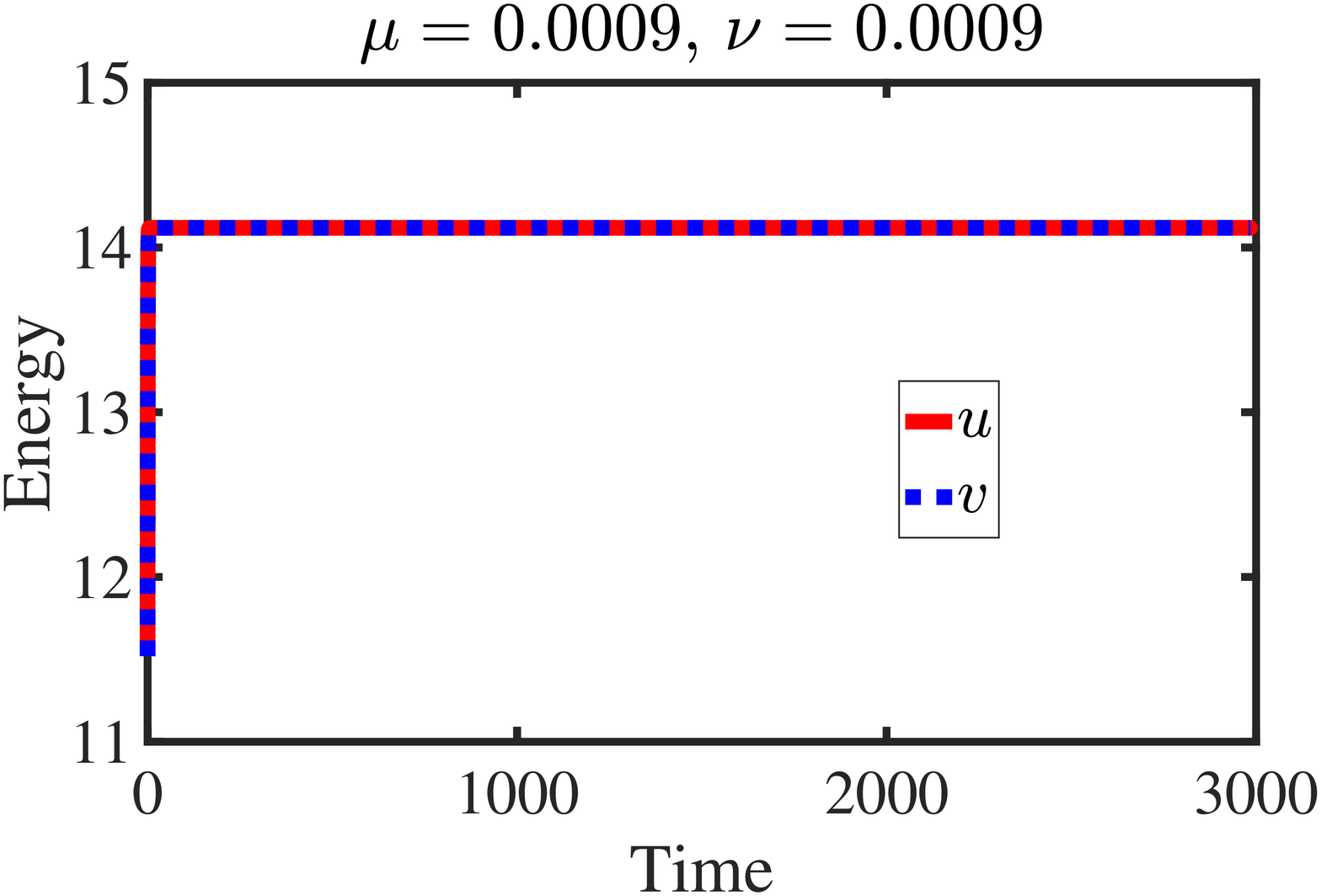}}
		\subfloat[]{\includegraphics[scale=.165]{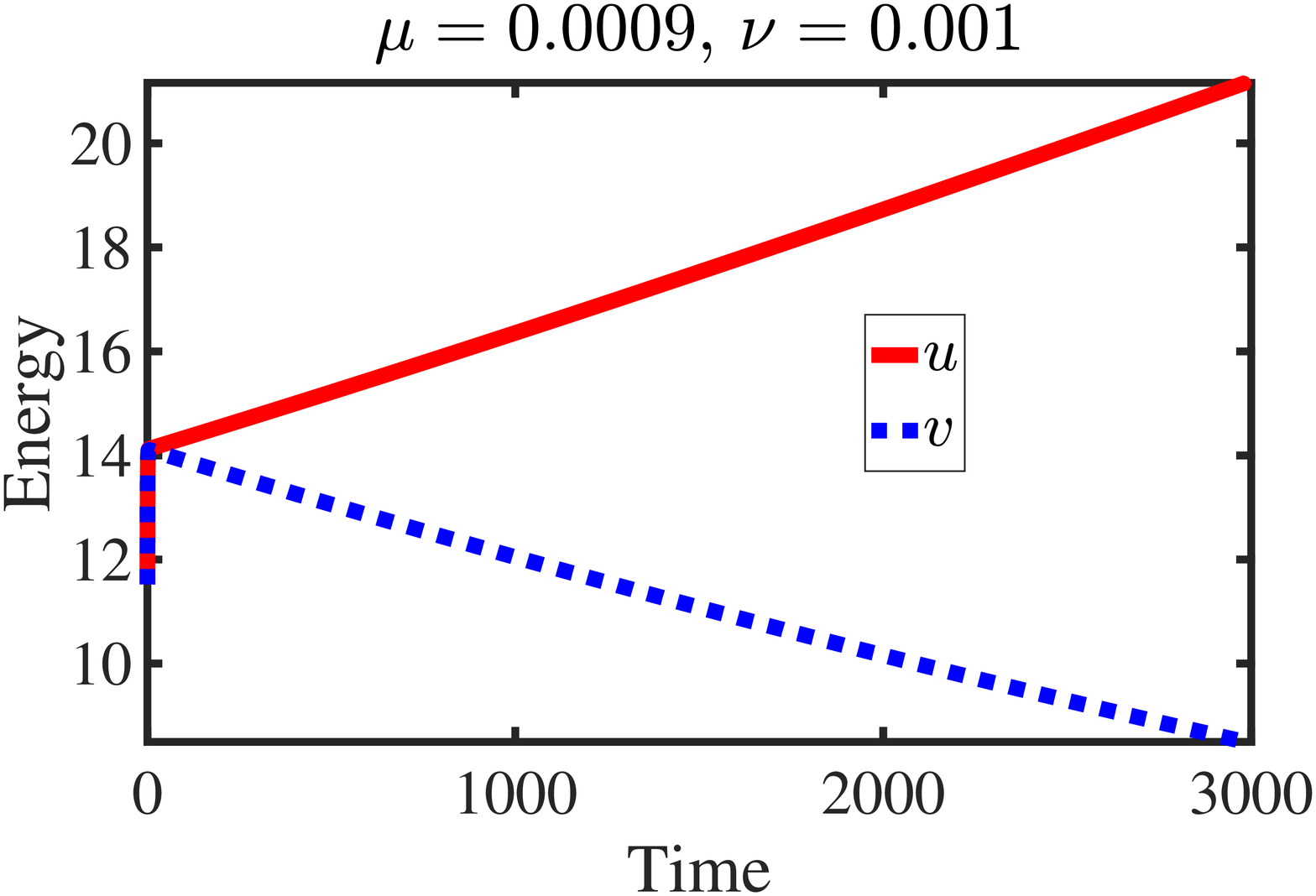}}
		\caption{Evolution of system energy for species density $u$, and $v$ with  (a) $\mu=0.0009$, and $\nu=0.0009$, and (b) $\mu=0.0009$, and $\nu=0.001$.}\label{var-r-space-depn}
	\end{figure}
{From Figure \ref{var-r-space-depn} (a), we observe, when the harvesting parameters do not exceed the intrinsic growth rate, a non-trivial solution exists. This means, the co-existence of the two species. In Figure \ref{var-r-space-depn} (b), we observe the population density of the first species remains always bigger than the second species as $\mu<\nu$. Ultimately, the second species will die out because of the competition between them.}

\subsection{Non-stationary carrying capacity}
In this section, we consider time-dependent periodic system carrying capacity together with constant and time-dependent intrinsic growth rates.
\subsubsection{Experiment 3: Constant intrinsic growth rate}
In this experiment, we consider a time-dependent carrying capacity $$K(t,\bx)\equiv (2.1+\cos(\pi x)\cos(\pi y))(1.1+\cos(t)),$$ harvesting coefficients $\mu=0.0009$, and $\nu=0.0025$, intrinsic growth rate $r(\bx)\equiv 1.0$, initial conditions $u_0=0.5$, and $v_0=1.5$ for the species $u$, and $v$, respectively. We have fixed time $t=T=13.74$, and draw the contour plots at $t=T, T+\pi/2, T+\pi, T+3\pi/2$, and $T+2\pi$, for the species density $u$, and $v$ in Figure \ref{contour-u-mu-0-0009-nu-0-0025}, and Figure \ref{contour-v-mu-0-0009-nu-0-0025}, respectively. From Figures \ref{contour-u-mu-0-0009-nu-0-0025}-\ref{contour-v-mu-0-0009-nu-0-0025}, we observe a quasi periodic behavior in both species and their co-existence.

\begin{figure} [H]
		\centering
		\subfloat[]{\includegraphics[scale=.14]{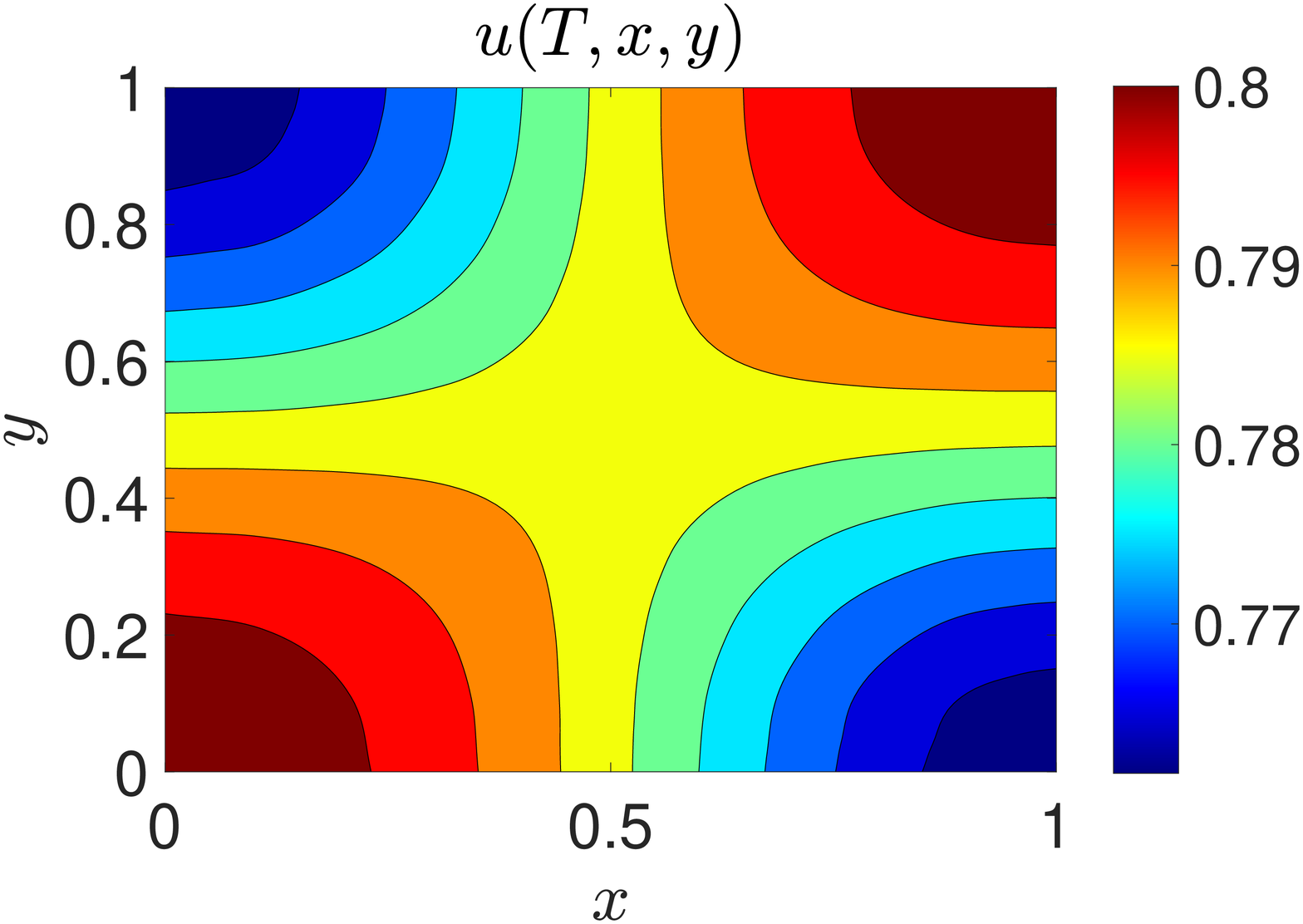}}
		\subfloat[]{\includegraphics[scale=.14]{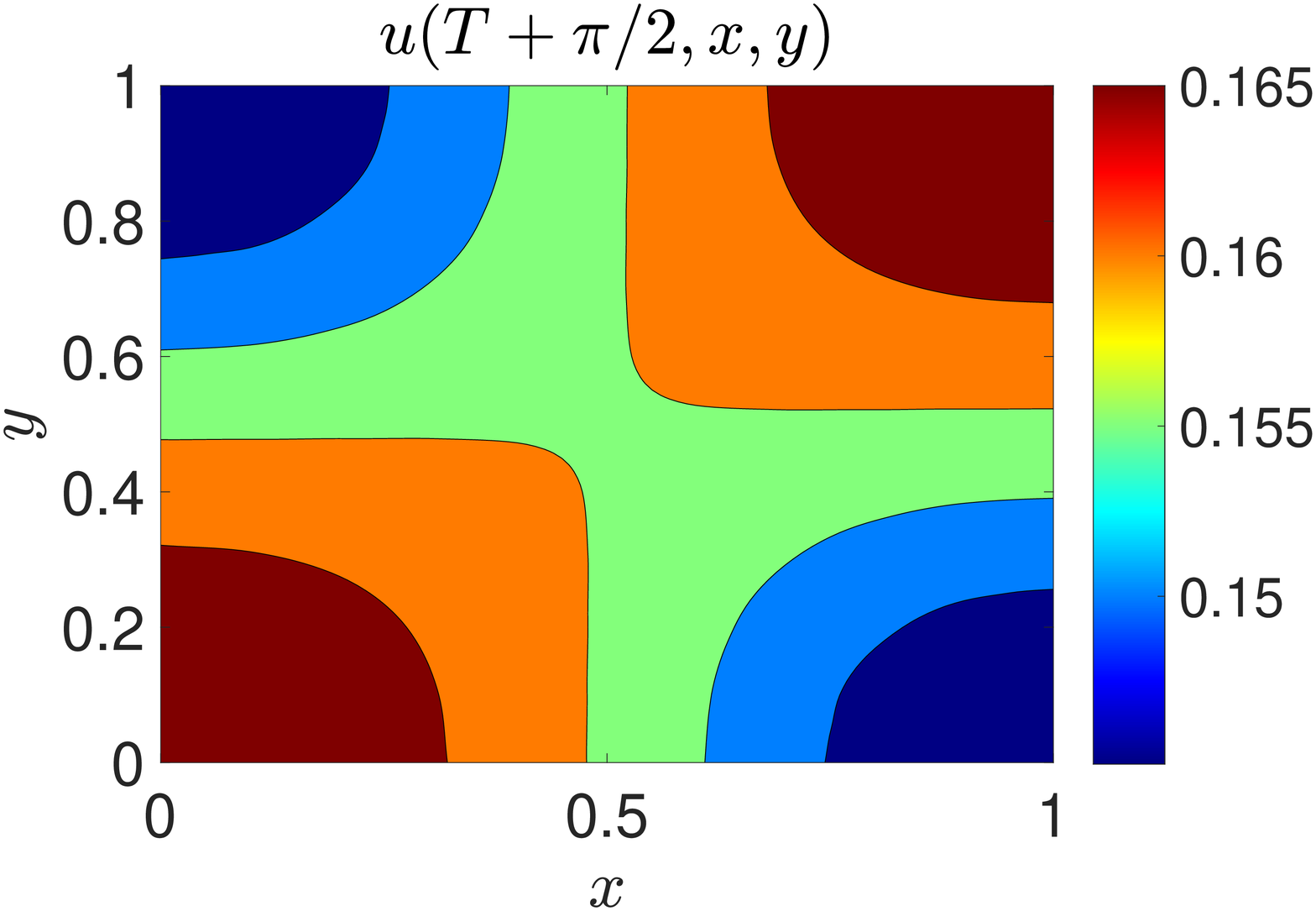}}\subfloat[]{\includegraphics[scale=.14]{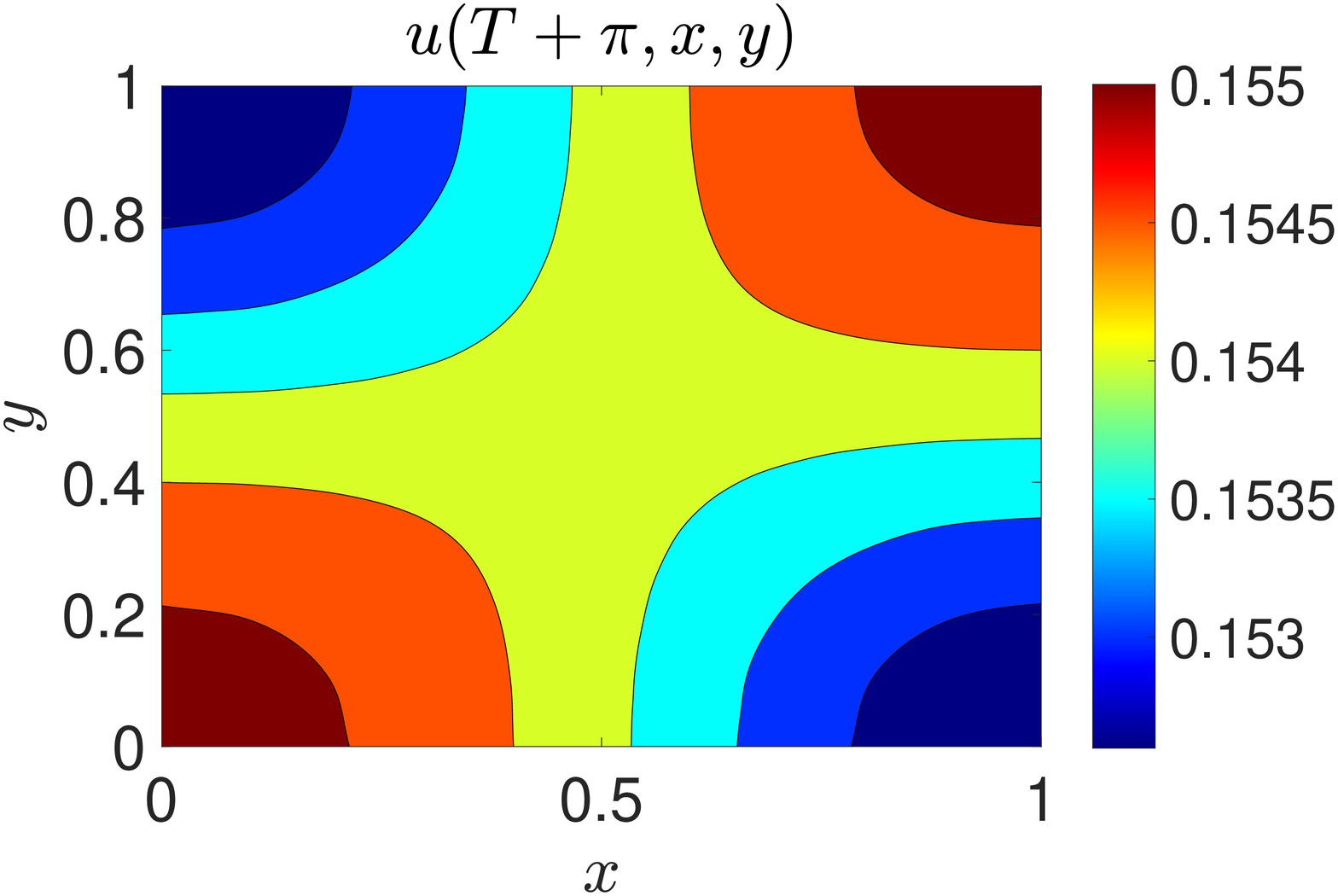}}~~\\\vspace{-2ex}
		\subfloat[]{\includegraphics[scale=.14]{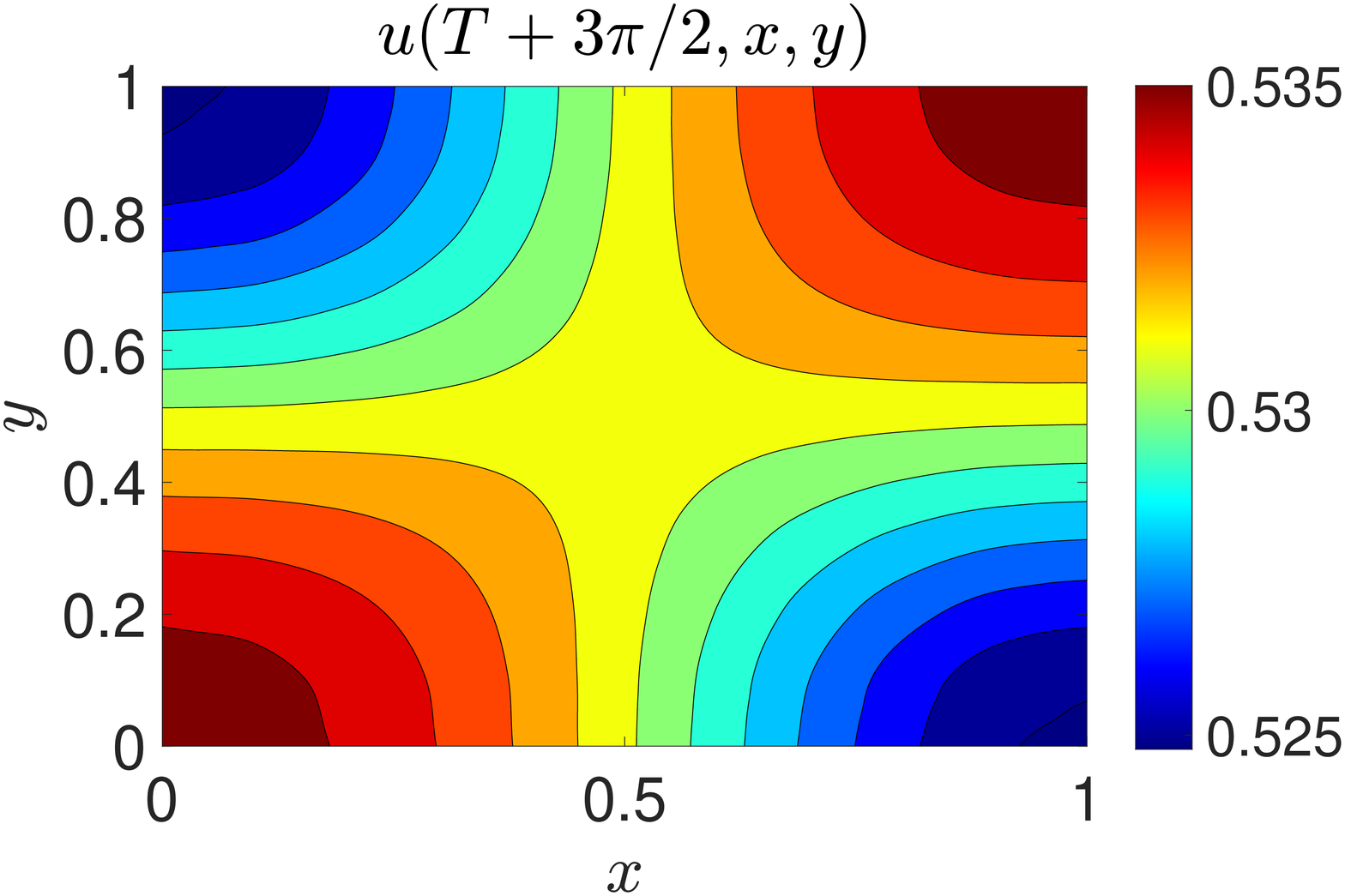}}
		\subfloat[]{\includegraphics[scale=.14]{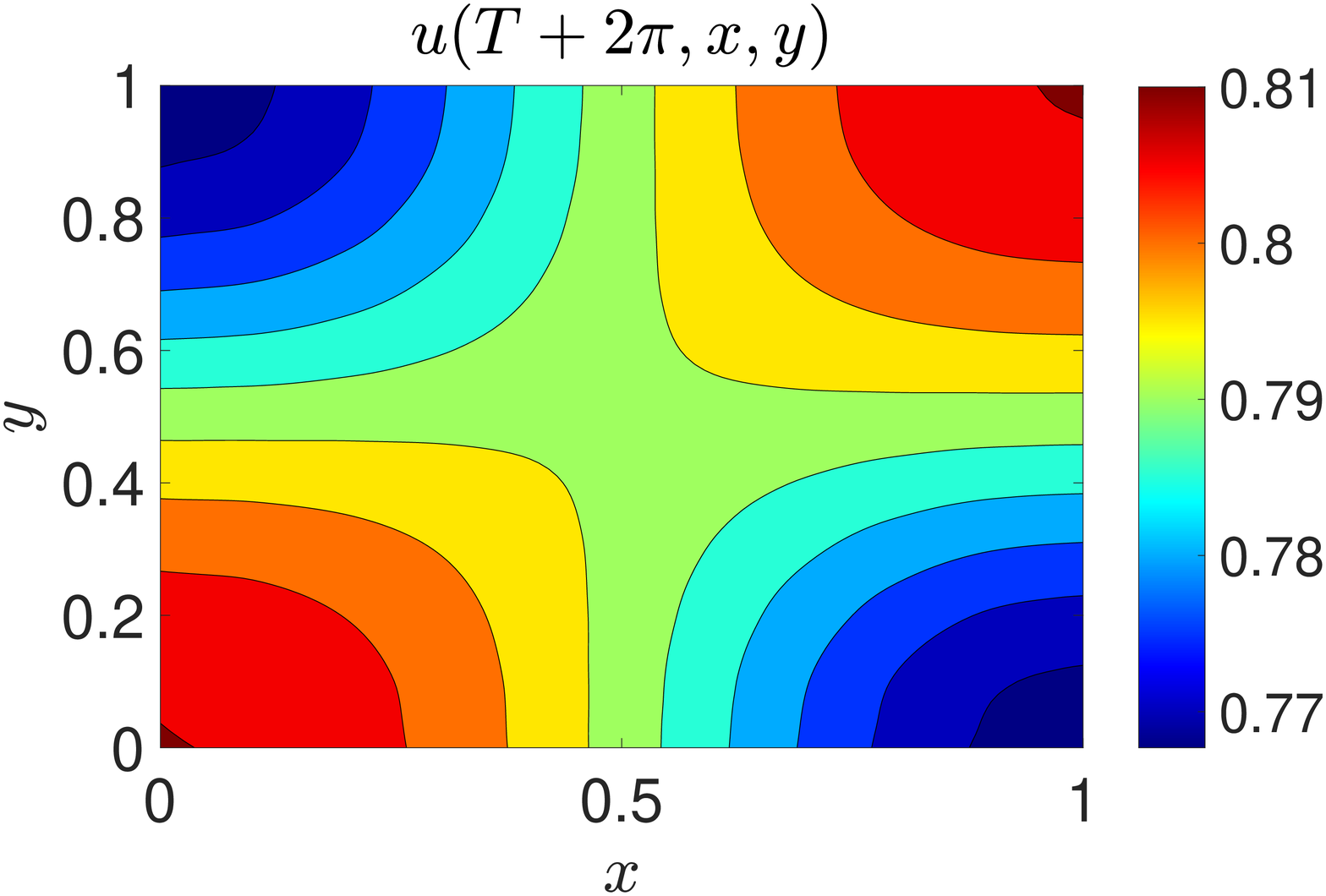}}
		\caption{Contour plot of species $u$ at five different time steps with harvesting coefficients $\mu=0.009$, and $\nu=0.0025$.}
		\label{contour-u-mu-0-0009-nu-0-0025}
	\end{figure}

\begin{figure} [H]
		\centering
		\subfloat[]{\includegraphics[scale=.14]{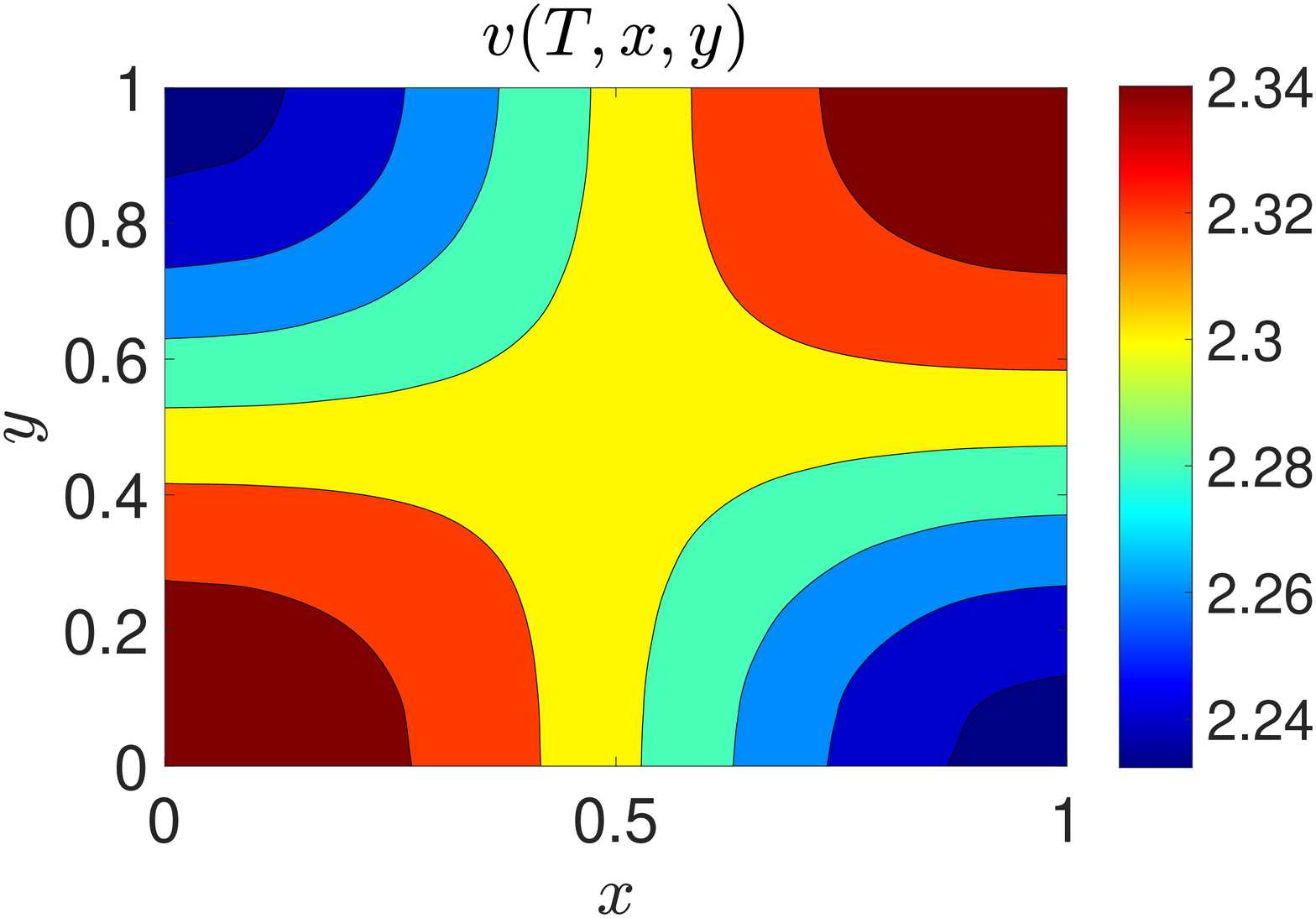}}
		\subfloat[]{\includegraphics[scale=.14]{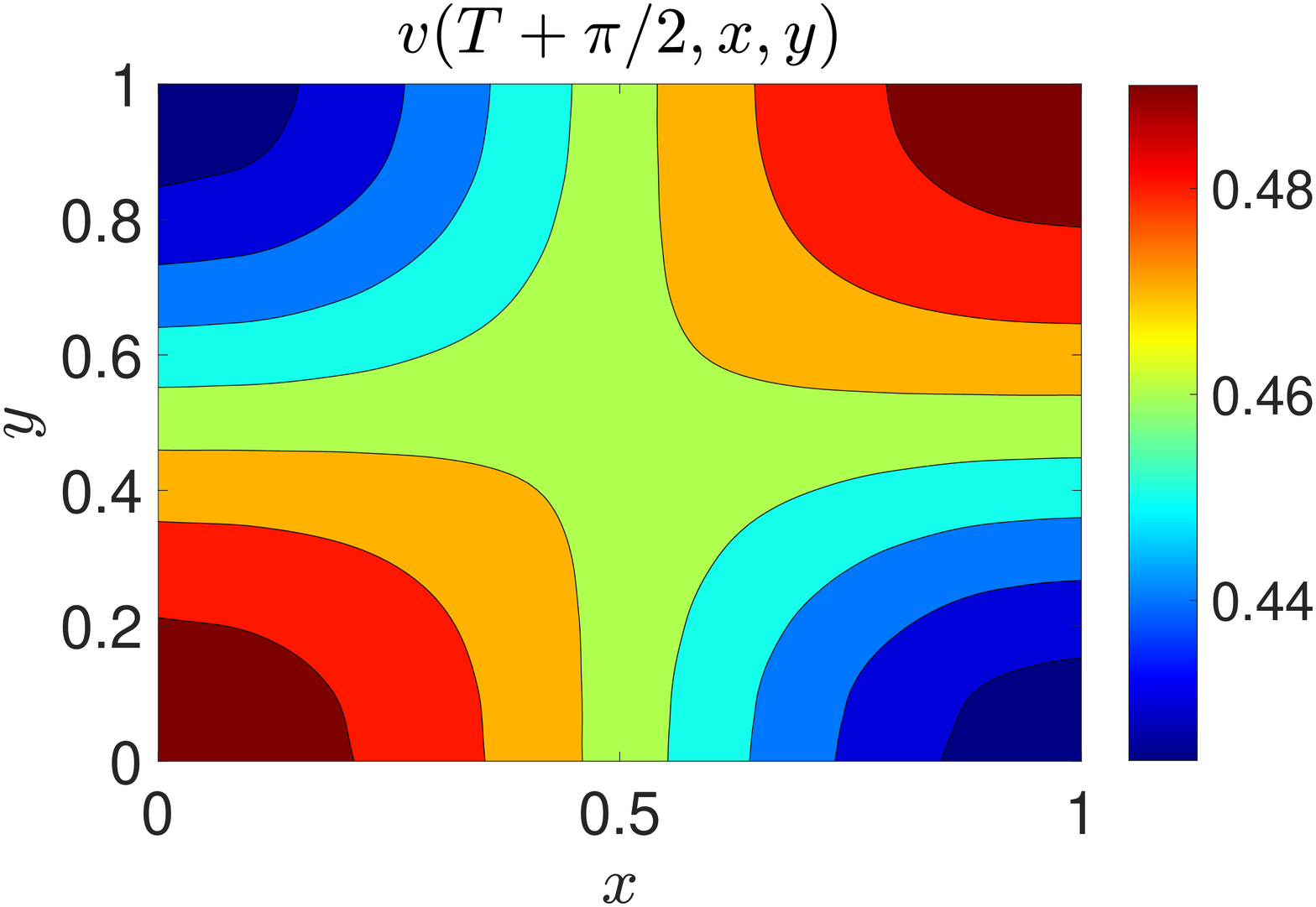}}\subfloat[]{\includegraphics[scale=.145]{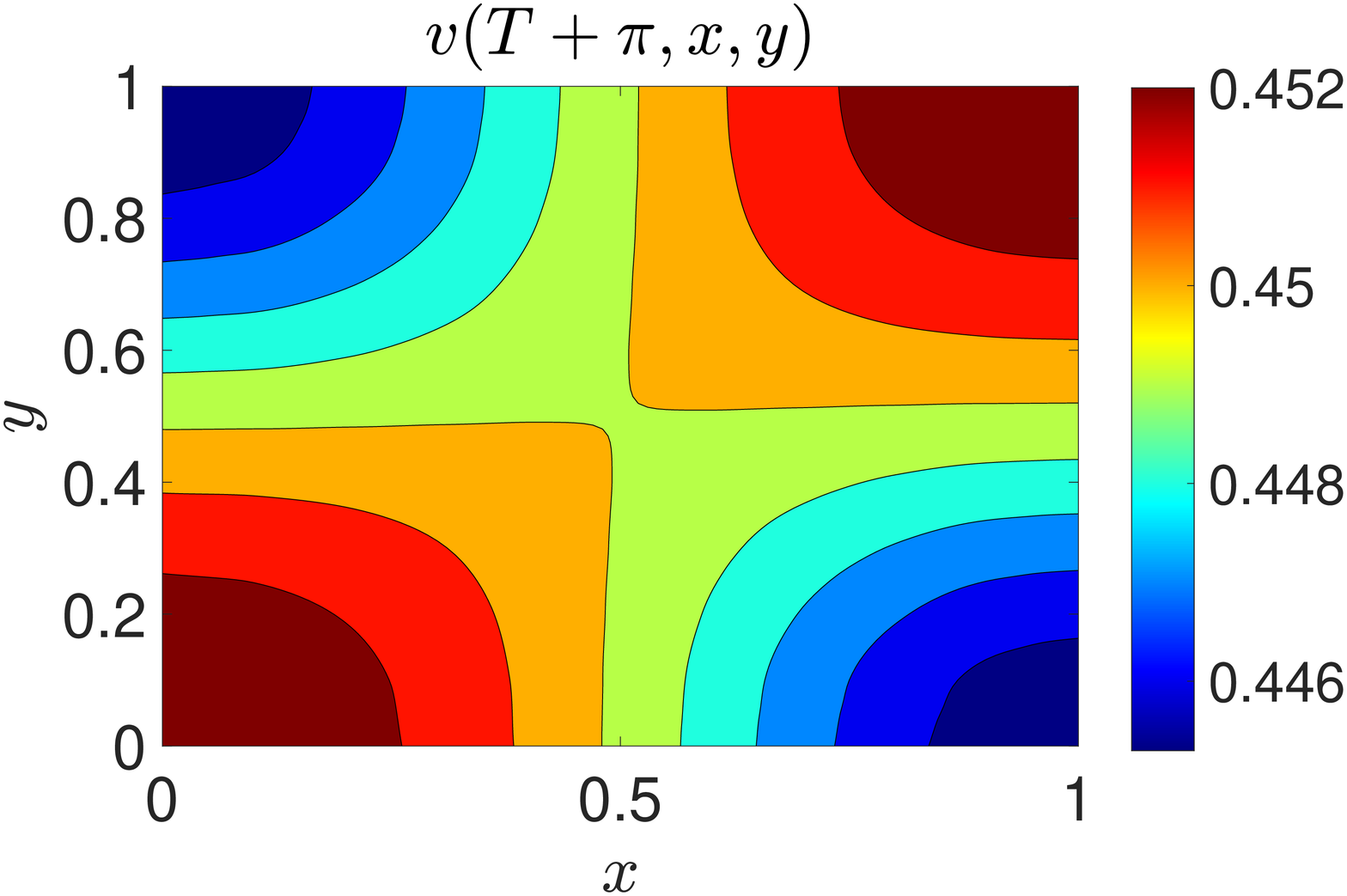}}~~\\\vspace{-2ex}
		\subfloat[]{\includegraphics[scale=.14]{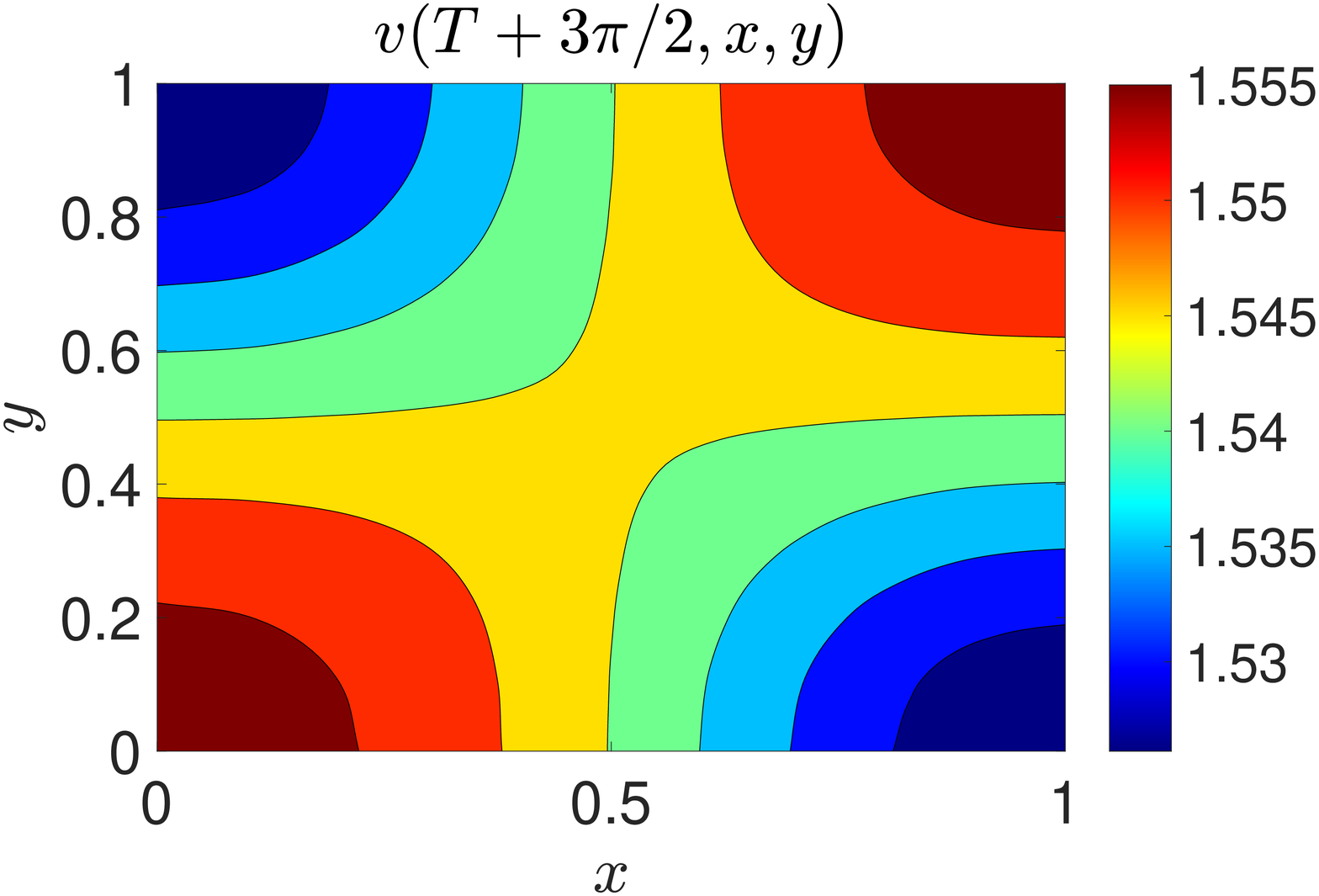}}
		\subfloat[]{\includegraphics[scale=.14]{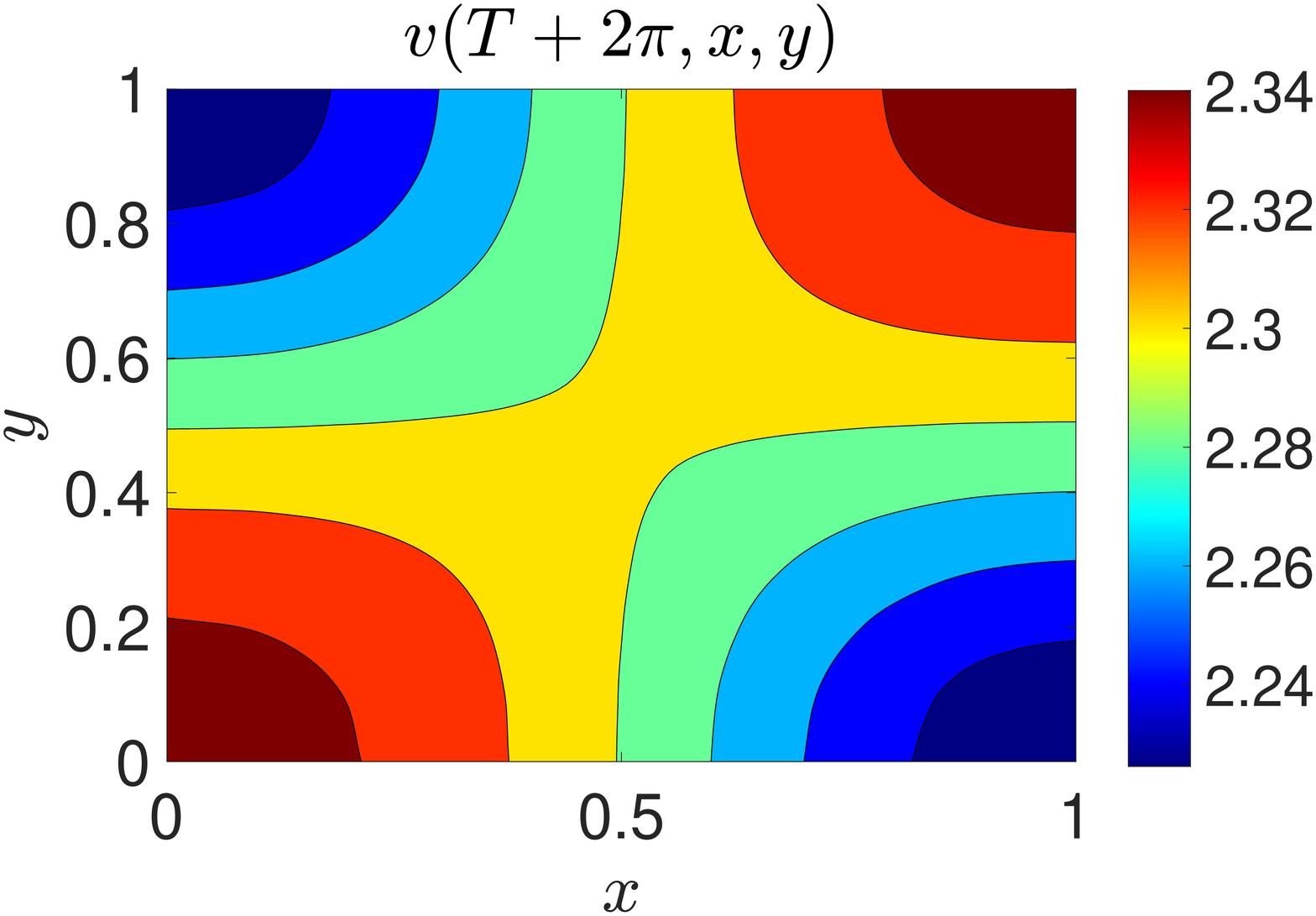}}
		\caption{Contour plot of species $v$ at five different time steps with harvesting coefficients $\mu=0.009$, and $\nu=0.0025$.}
		\label{contour-v-mu-0-0009-nu-0-0025}
	\end{figure}
We also plot the energy of the system corresponding to the species density $u$, and $v$ versus time in Figure \ref{RDE-energy-0-0009-nu-0-0025-dt-0-5-var-K-T-200}. We observe a clear co-existence of the two populations and change their density quasi periodically over time. Since, in this case, $\mu<\nu$, the amplitude of the species density $u$ increases while it decreases for $v$.
	\begin{figure} [H]
		\centering
		\subfloat[]{\includegraphics[width=0.5\textwidth,height=0.3\textwidth]{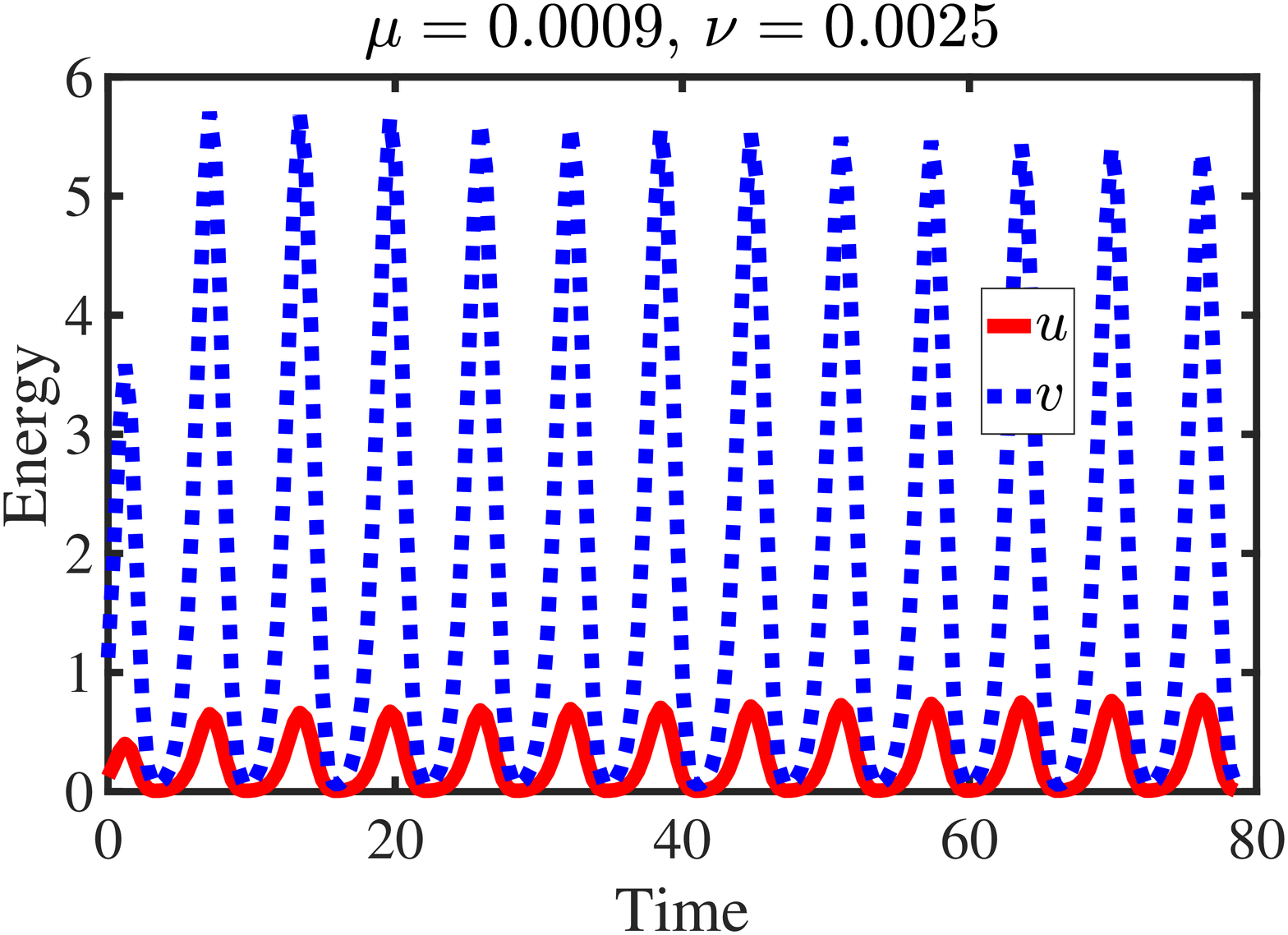}}
		\subfloat[]{\includegraphics[width=0.5\textwidth,height=0.3\textwidth]{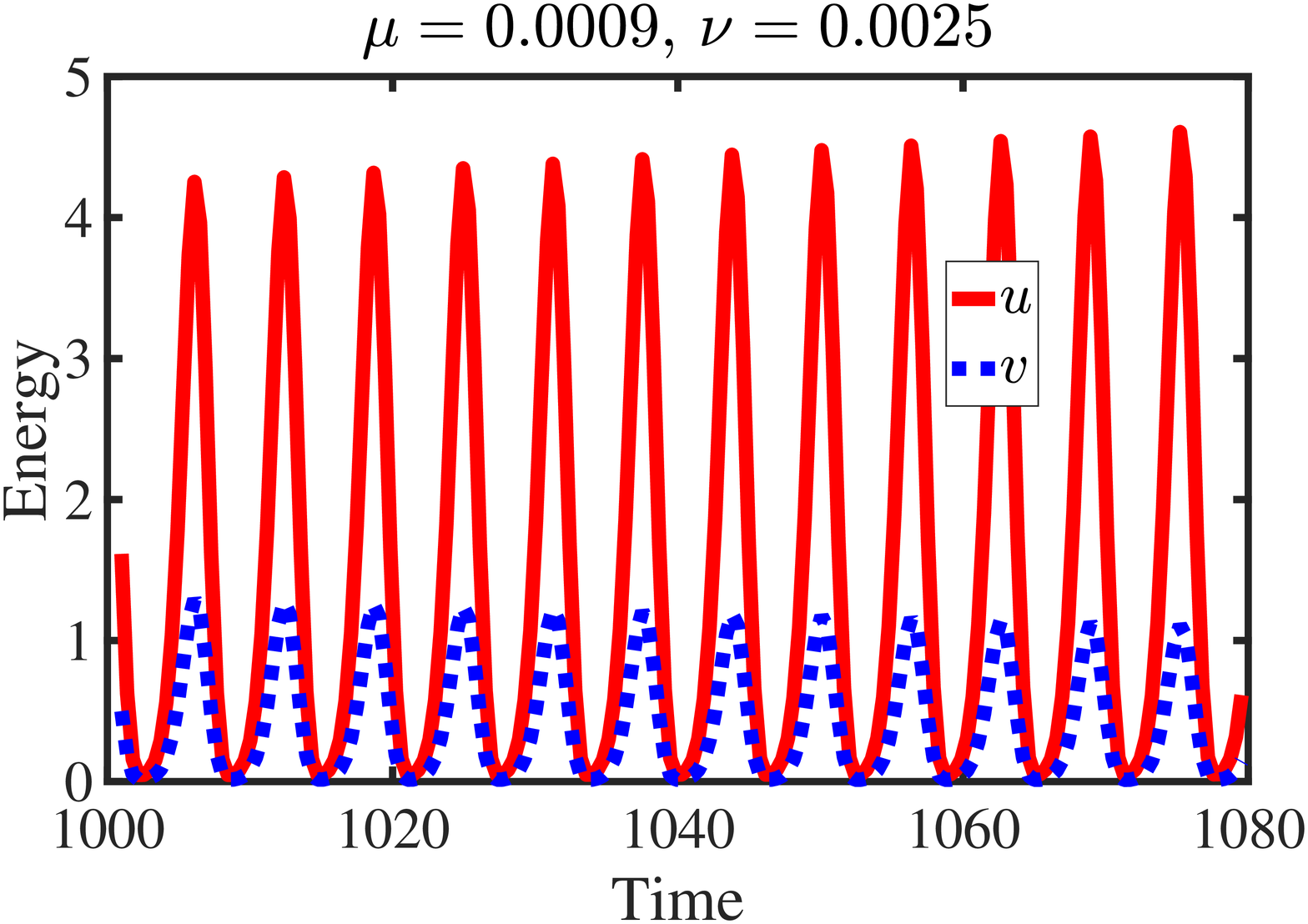}}
		\caption{Short-time (a), and long-time (b) energy of the system for the species density $u$, and $v$ with the harvesting coefficients $\mu=0.0009$, and $\nu=0.0025$.}
		\label{RDE-energy-0-0009-nu-0-0025-dt-0-5-var-K-T-200}
	\end{figure}
\subsubsection{Experiment 4: Exponentially varying carrying capacity}
{In this experiment, we consider the carrying capacity $$K(t,\bx)\equiv\big(1.2+2.5\pi^2e^{-(x-0.5)^2-(y-0.5)^2}\big)\big(1.0+0.3\cos(t)\big),$$ together with constant intrinsic growth rate $r(\bx)\equiv 1$, initial population density $u_0=v_0=1.6$, and harvesting coefficients $\mu=0.0009$, and $\nu=0.0025$.}
\begin{figure} [H]
		\centering
		\subfloat[]{\includegraphics[width=0.5\textwidth,height=0.3\textwidth]{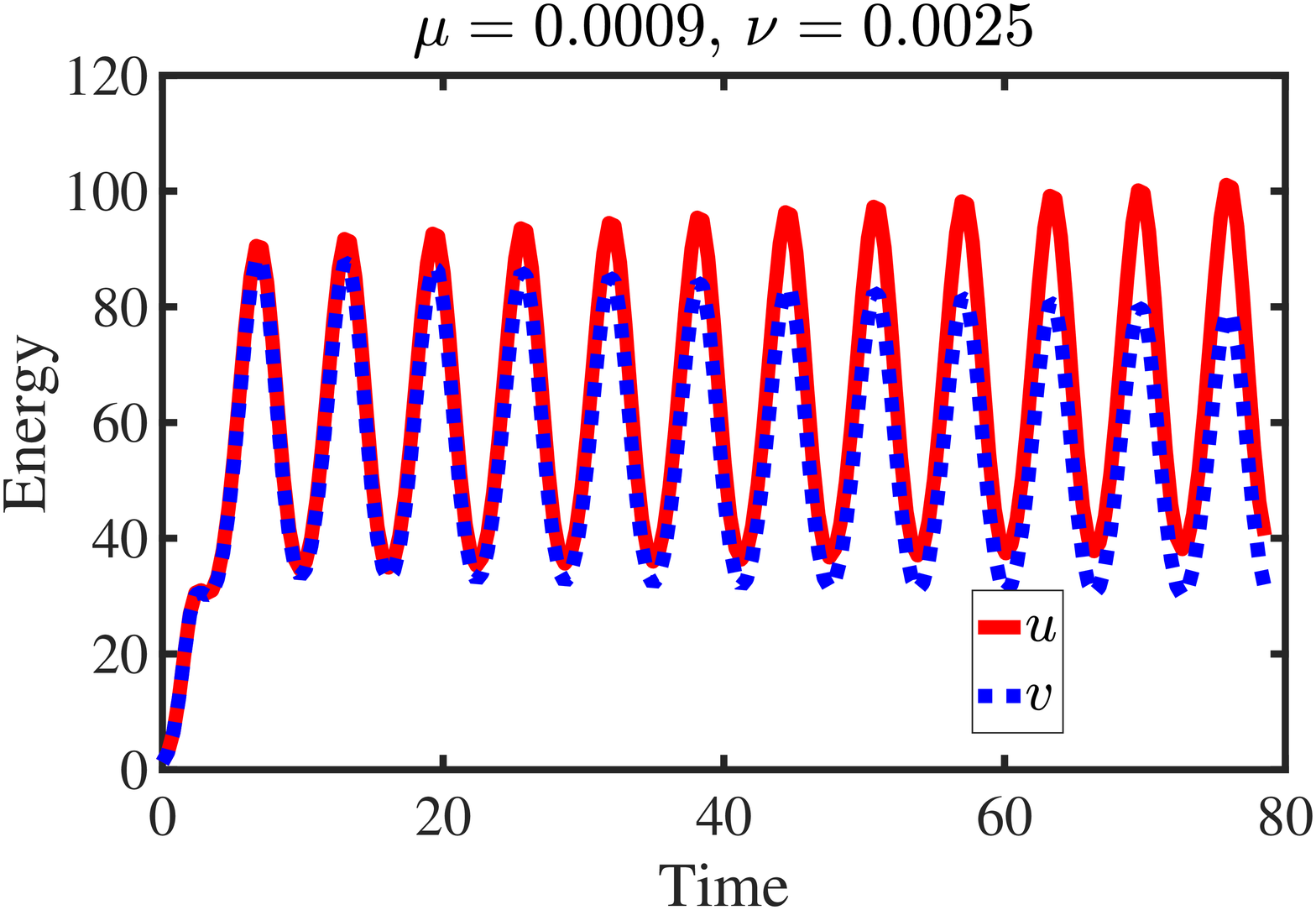}}
		\subfloat[]{\includegraphics[width=0.5\textwidth,height=0.3\textwidth]{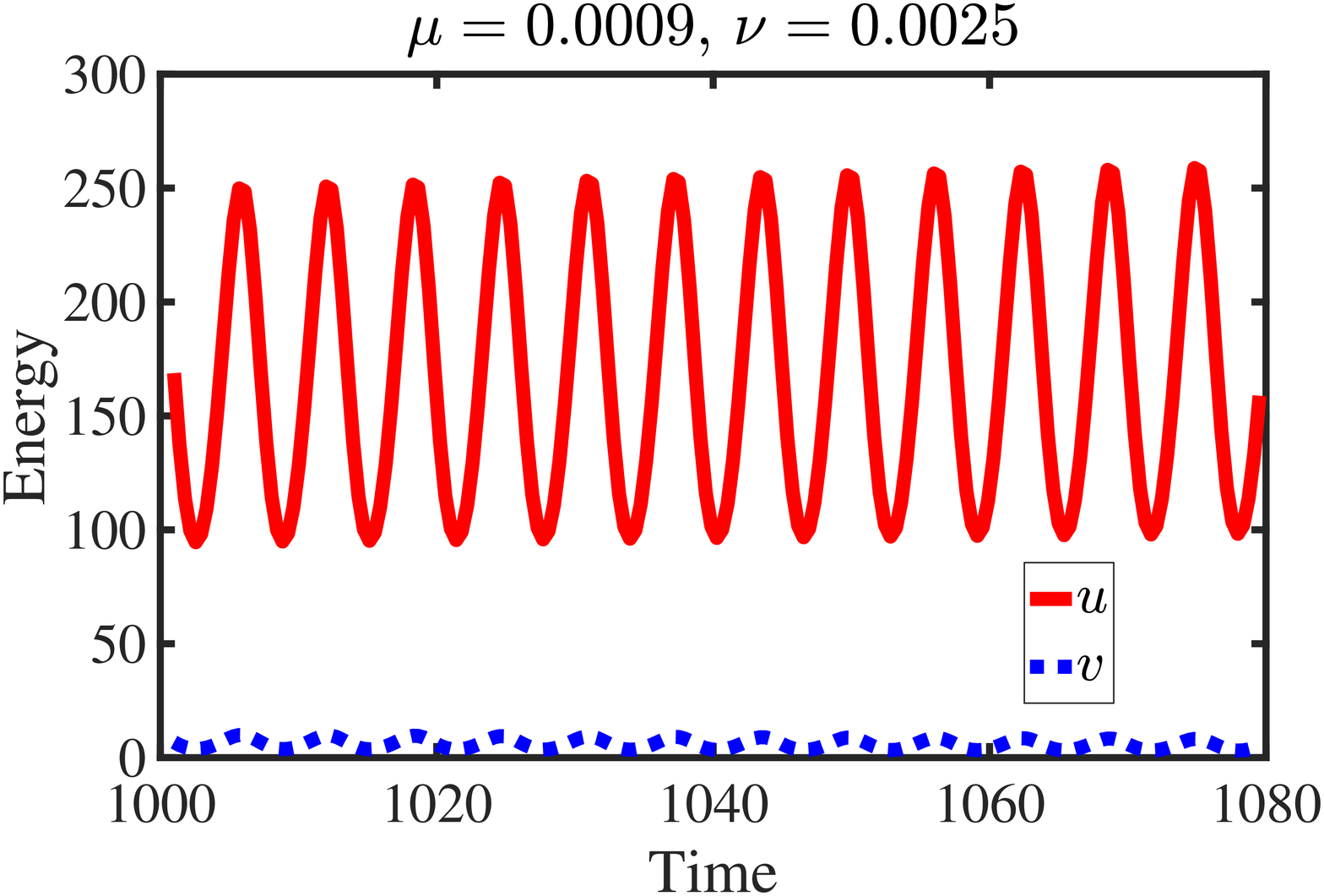}}
		\caption{Short-time (a), and long-time (b) energy of the system for the species density $u$, and $v$ with the harvesting coefficients $\mu=0.0009$, and $\nu=0.0025$.}	\label{RDE-energy-0-0009-nu-0-0025-K-time-exp}
	\end{figure}
{In Figure \ref{RDE-energy-0-0009-nu-0-0025-K-time-exp}, the system energy versus time is plotted for both short-time and long-time evolution with the harvesting coefficients $\mu=0.0009$, and $\nu=0.0025$. We observe periodic population densities for both species and eventually the species $v$ dies out but the species $u$ continues to exist.}

\begin{figure} [H]
		\centering
		\subfloat[]{\includegraphics[scale=.17]{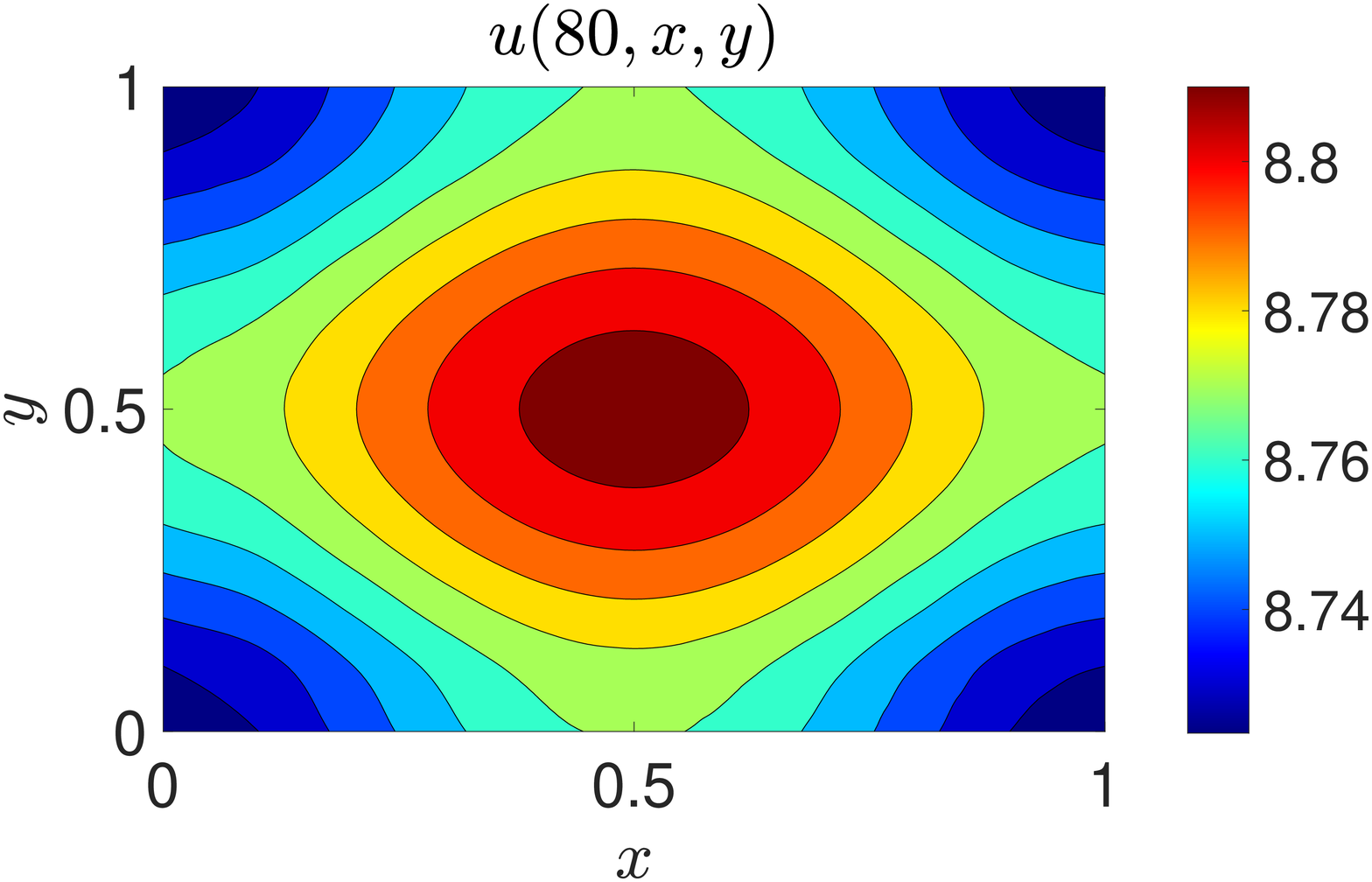}}
		\subfloat[]{\includegraphics[scale=.17]{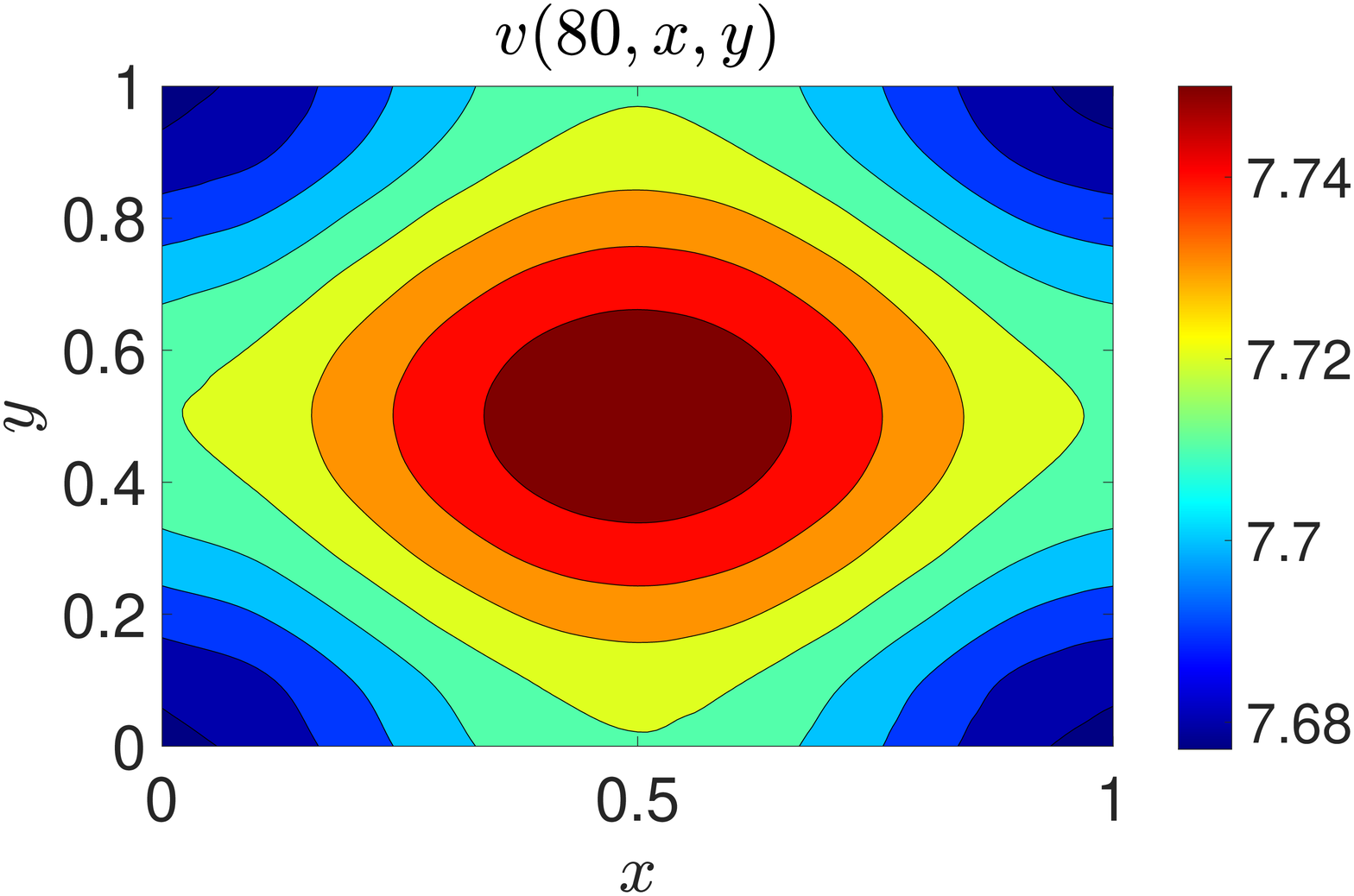}}\\
		\subfloat[]{\includegraphics[scale=.17]{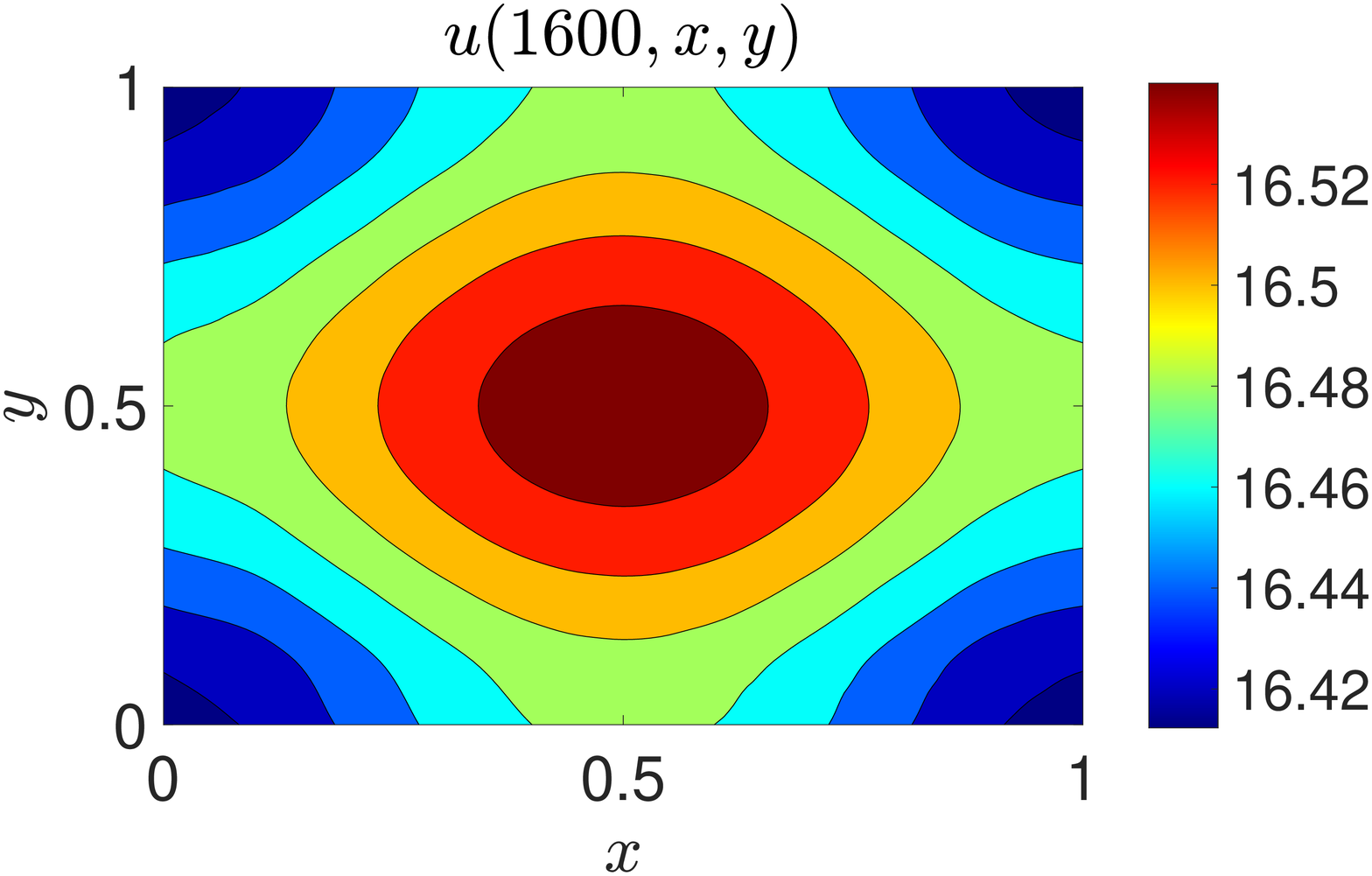}}
		\subfloat[]{\includegraphics[scale=.17]{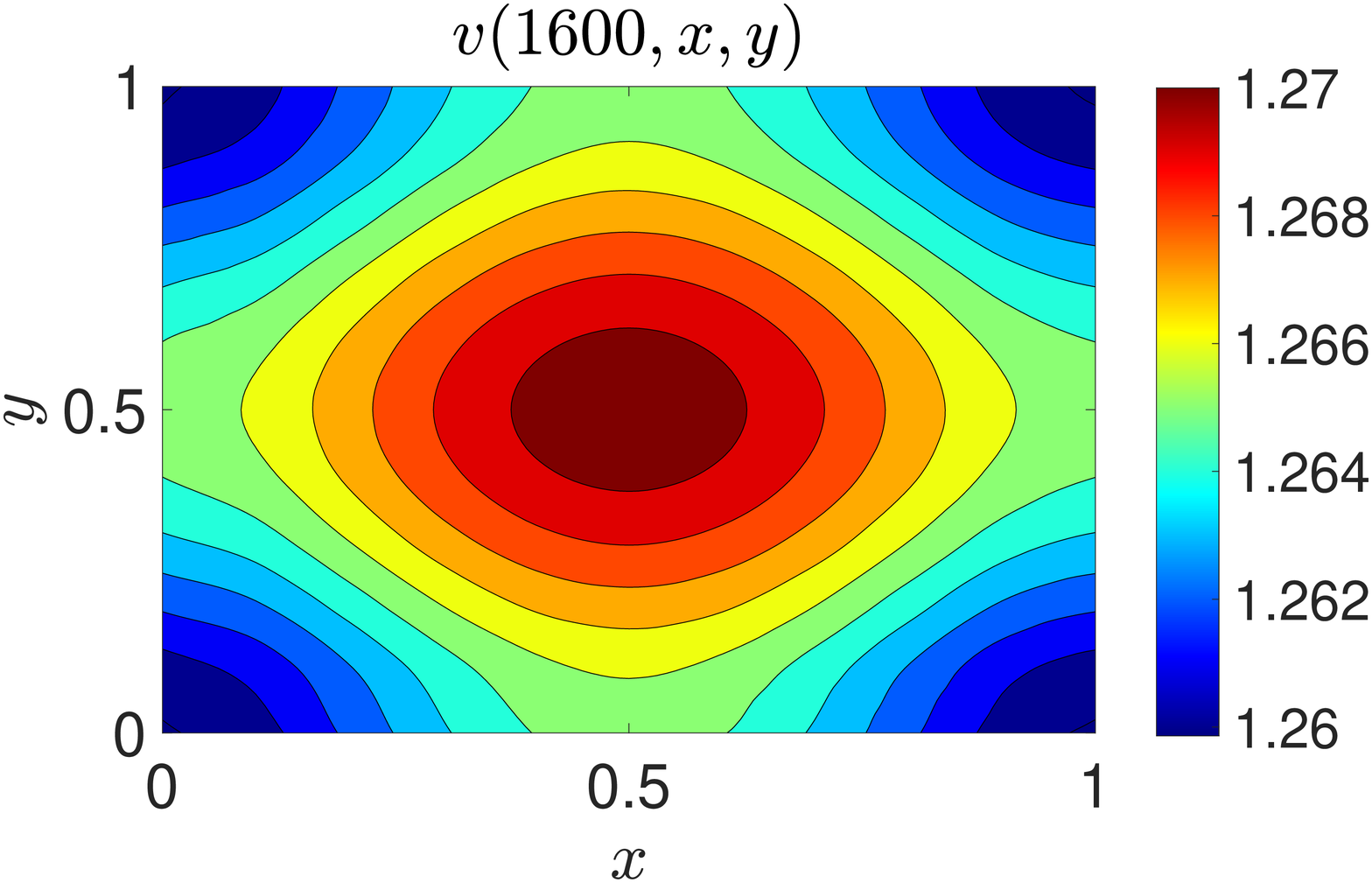}}
		\caption{Contour plot of species $v$ at time $t=80$ (top), and $t=1600$ (bottom) with harvesting coefficients $\mu=0.009$, and $\nu=0.0025$.}
		\label{contour-v-mu-0-0009-nu-0-0025-exp}
	\end{figure}
{In Figure \ref{contour-v-mu-0-0009-nu-0-0025-exp}, we represent the contour plot for both species at times $t=80$, and $t=1600$. It is observed that the highest population density is at $(0.5,0.5)$ and there is a coexistence of both species though the population density of the species $u$ remains bigger than the species $v$ at every places. This is the effect of the different harvesting parameters.}

\subsubsection{Experiment 5: Time dependent intrinsic growth rate}
{In this experiment, we consider a periodic, both time and space dependent carrying capacity and intrinsic growth rate as $$K(t,\bx)\equiv(2.5+\cos(x)\cos(y))(1.2+\cos(t)),$$
and $$r(t,\bx)\equiv(1.5+\sin(x)\sin(y))(1.2+\sin(t)),$$ respectively. We plot the system energy versus time in Figure \ref{r-time-space-k-time-space} (top) for the equal harvesting coefficients pair $\mu=\nu=0.0009$ for both short-time and long-time. Clearly, $\mu=\nu=0.0009<\inf\limits_\Omega r(t,\bx)$, and thus we observe a co-existence of the species. In Figure \ref{r-time-space-k-time-space} (bottom), the system energy versus time is plotted for  $\mu=0.0009$, and $\nu=0.001$ for both short-time and long-time. In this case, we see the presence of both species in both the short-time and long-time evolution having the effect of harvesting. That is, the amplitude of the density $u$ increases whereas for $v$ it decreases. }
\begin{figure} [H]
		\centering
		\subfloat[]{\includegraphics[width=0.5\textwidth,height=0.3\textwidth]{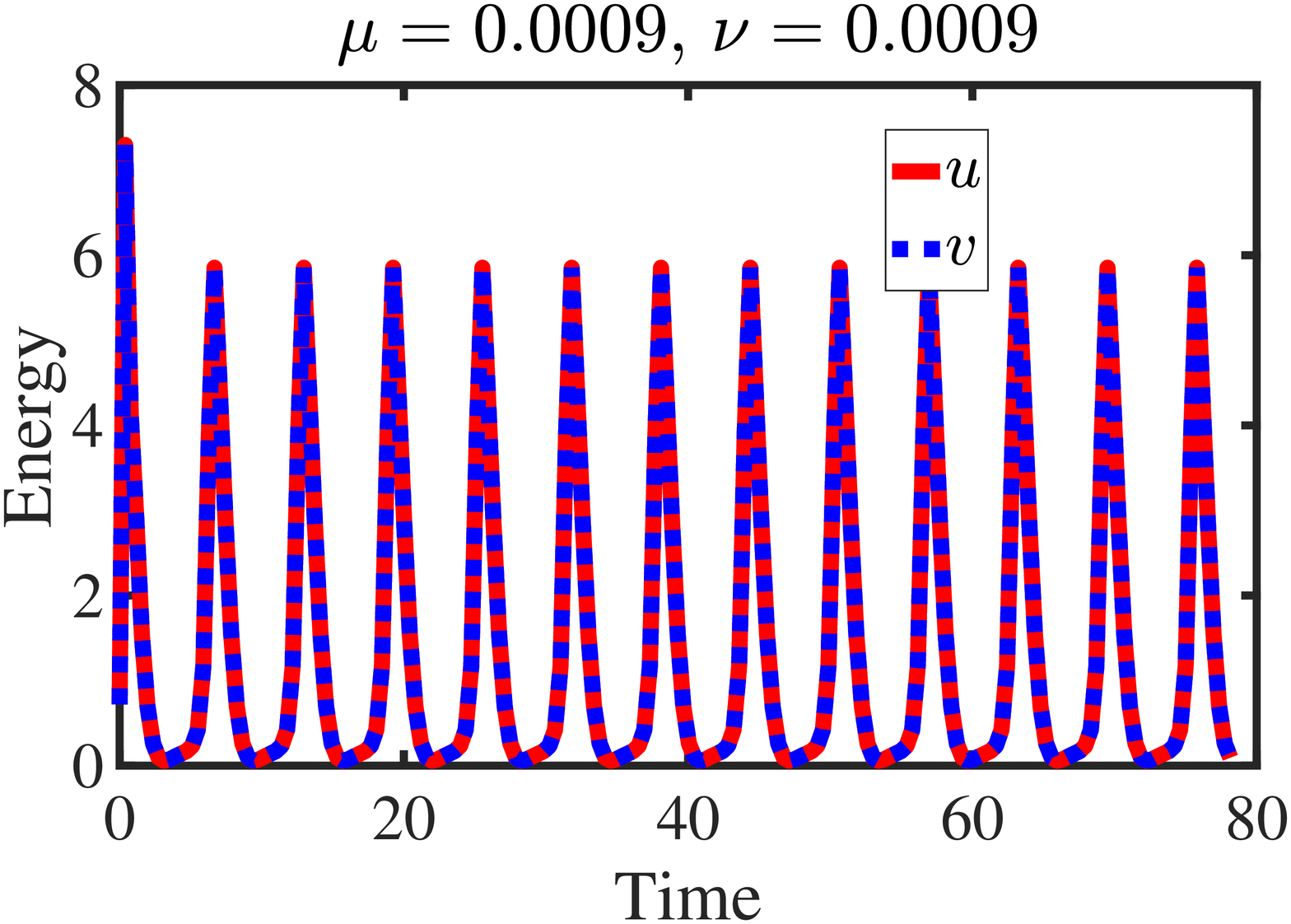}}
		\subfloat[]{\includegraphics[width=0.5\textwidth,height=0.3\textwidth]{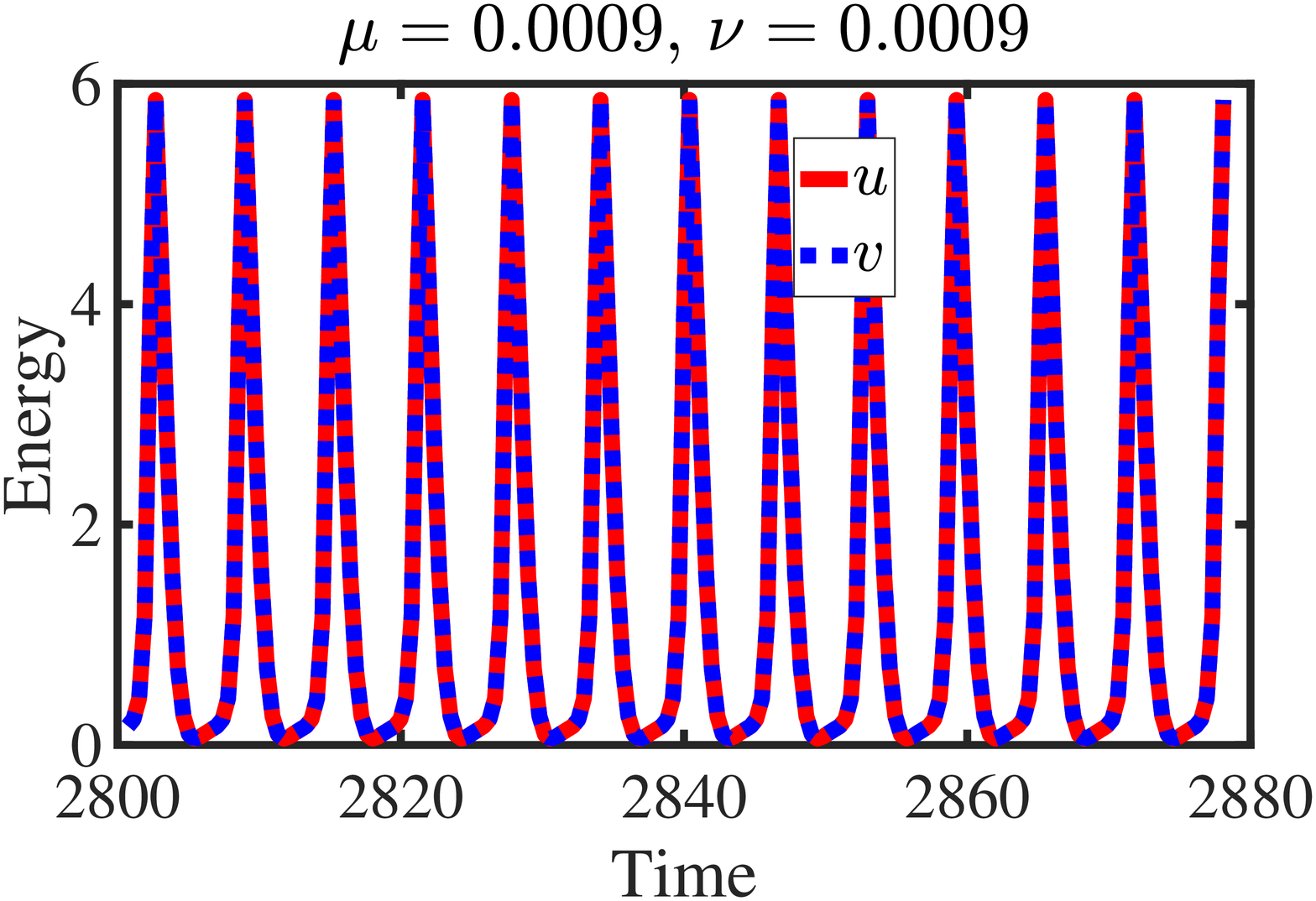}}\\
		\subfloat[]{\includegraphics[width=0.5\textwidth,height=0.3\textwidth]{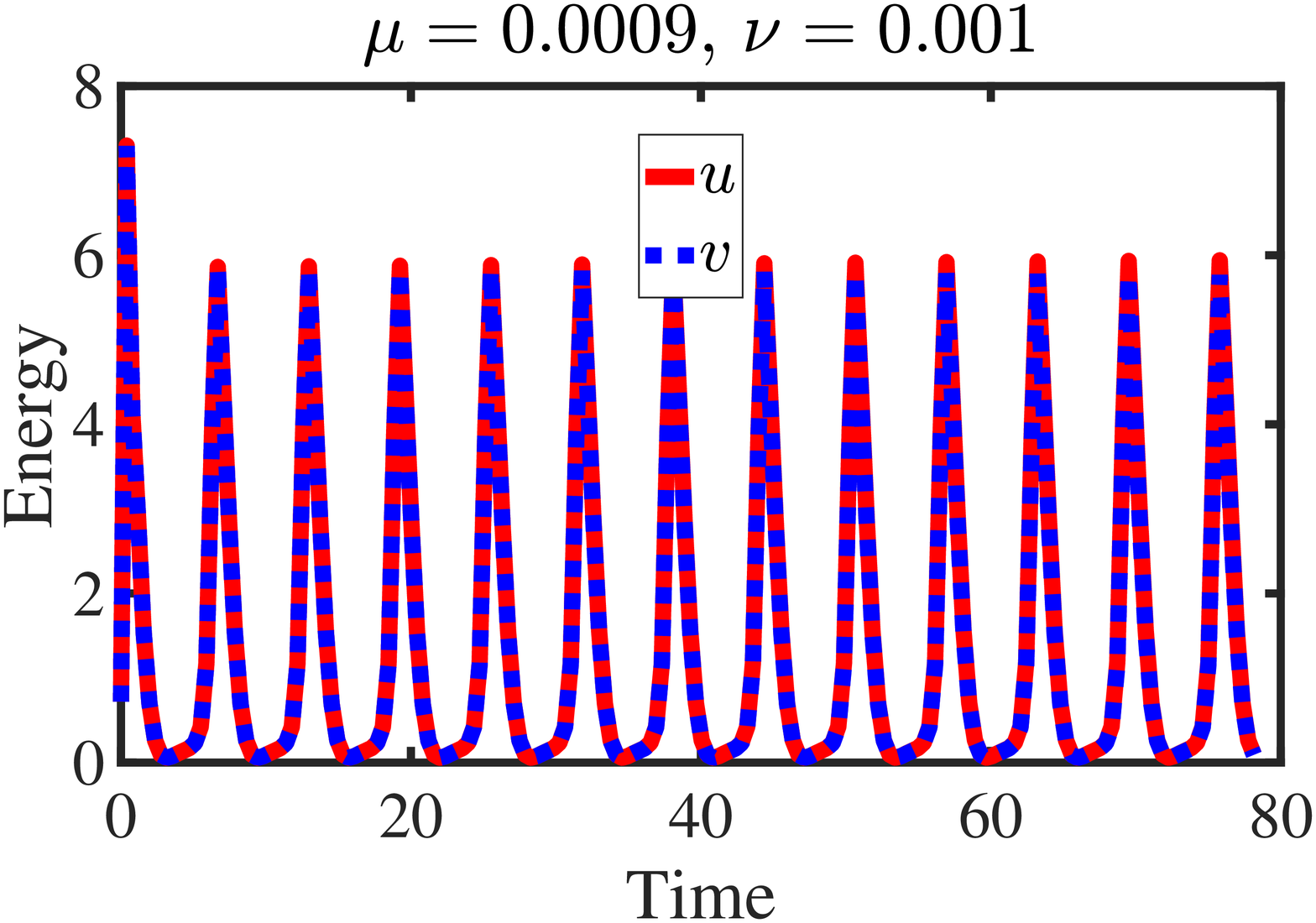}}
		\subfloat[]{\includegraphics[width=0.5\textwidth,height=0.3\textwidth]{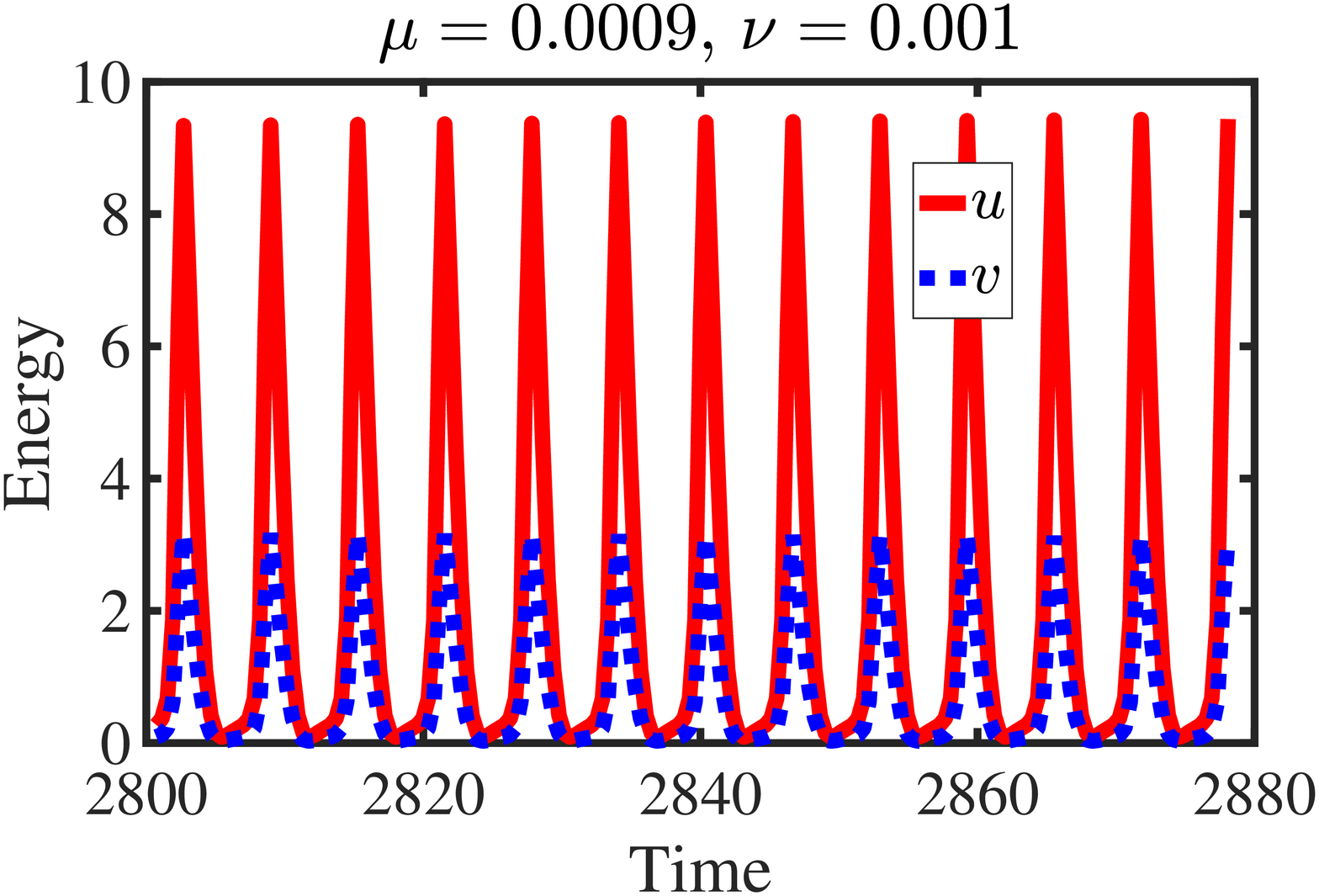}}
		\caption{Evolution of system energy for species density $u$, and $v$ for $u_0=v_0=1.2$, with periodic space and time dependent intrinsic growth rate and carrying capacity.}\label{r-time-space-k-time-space}
	\end{figure}

\section{Conclusion}\label{conclusion}
In this paper, we have studied two competing species in spatially heterogeneous environments. We observe various scenarios for several harvesting rates. 
When the harvesting rate does not surpass the intrinsic growth rate that means $\mu,\nu \in [0,1)$ and imposes conditions on $\mu$ and $\nu$, coexistence is possible. For small values of $\mu$ and $\nu$, prey and predator population coexist which is observed analytically and numerically. Moreover, we estimate the threshold of harvesting coefficient when coexistence is possible. Further, only one species extinct when their harvesting rate is greater than their growth rate and other species persist when the harvesting rate is less than their growth rate. Both species become extinct when their harvesting rate exceeds the growth rate and the system (\ref{equ_1}) as well as (\ref{equ_2}) converges to the trivial solution.
From these analytic and numerical observations, we can conclude that prey and predator species coexist when the harvesting rate is less than their growth rate and whenever the harvesting rate exceeding their intrinsic growth rate both species dies out.

\section*{Acknowledgments}
The author, M. Kamrujjaman research, was partially supported by the University Grants Commission (UGC),  University of Dhaka, Bangladesh. Also the author Muhammad Mohebujjaman, research was supported by National Science  Foundation grant DMS-2213274, and University Research grant, TAMIU.
\section*{Conflict of interest}
The authors declare no conflict of interest. 



 \appendix
\section{Appendix}\label{A}
Let $\mathcal{T}_{t}((m_{0}(x), n_{0}(x))) = (m(t, x), n(t, x))$, which implies the operator $\mathcal{T}_{t}$ picks out the initial conditions with boundary conditions of the system
\begin{equation}\label{A.1}
	\begin{cases}
		& \displaystyle	\dfrac{\partial m_{i}}{\partial t}-\mathcal{L} m_{i} =f_{i}(x,m_{1},m_{2}),\;\;\; t>0,\;\;\; x\in\Omega,\\
		& \displaystyle\dfrac{\partial m_{i}}{\partial \eta}=0,\;\;\; x\in \partial\Omega,\\
		& \displaystyle m_{i}(0,x)=m_{i,0}(x),\;\; x\in\Omega,\;\; i=1,2,
	\end{cases}
\end{equation}
and gives the solution $(m(t, x), n(t, x))$. The operator $\mathcal{L}$ represented in the following way
\begin{equation}\label{A.2}
	\mathcal{L} m = \sum_{i,j=1}^{p} a_{ij}(t,x) \frac{\partial^2 m}{\partial x_{i}x_{j}}+\sum_{i=1}^{p} b_{i}(t,x)\frac{\partial m}{\partial x_{i}}
\end{equation}
as well as uniformly elliptic with H\"older continuous coefficients, moreover, there subsist two positive real numbers, namely, $\lambda$ and $\Lambda$ such that for any vector $\zeta =(\zeta_{1},...,\zeta_{n})\in \mathbb{R}^{n}$,
\begin{equation}\label{A.3}
	\lambda \mid \zeta \mid^{2}\leq\sum_{i,j=1}^{p} a_{ij}(t,x)\zeta_{i}\zeta_{j}\leq \Lambda\mid \zeta\mid^{2},\;\;\;(t,x)\in [0, \mathcal{T}]\times(\overline{\Omega}).
\end{equation}
The proof of the following theorem is given in \cite{Ref_2}.
\begin{Th}\label{th_A.1}
	Suppose $\mathcal{T}_t$ is defined by $\mathcal{T}_{t}((m_{0}(x), n_{0}(x))) = (m(t, x), n(t, x))$, where $(m(t, x), n(t, x))$ is a solution to the first equation of (\ref{A.1}). Assume the following cases hold:
	\begin{enumerate}
		\item [1.] $\mathcal{T}_t$  is strictly order preserving, which implies that $m_1(x)\geq n_1(x)$ and $m_2(x)\leq n_2(x)$ indicate that $\mathcal{T}_t(m_1(x))\geq \mathcal{T}_t(n_1(x))$ and $\mathcal{T}_t(m_2(x)) \leq \mathcal{T}_t(n_2(x))$.
		\item [2.]  $\mathcal{T}_t(0, 0) = 0 $ for all $t > 0$ and $(0, 0)$ is a repelling equilibrium. Then there subsists a neighbourhood $\mathcal{M}$ of $(0,0)$ in $X^{+}$ imply that for each $(m_1,n_1)\in ,\; (m_1,n_1)\neq (0,0)$, there is $t_0 >0$ imply that $\mathcal{T}_{t_{0}}(m_1, n_1) \notin \mathcal{M} $.
		\item [3.] $\mathcal{T}_t((m_1, 0)) = (\mathcal{T}_t(m_1), 0)$ and $\mathcal{T}_t(m_1) \geq 0$, if $m_1 \geq 0$ such that there exists  $\widetilde{m}_{1}> 0$ imply that $\mathcal{T}_t((\widetilde{m}_{1}, 0)) = (\widetilde{m}_{1}, 0)$ for any $t \geq 0$. Same cases holds for $(0, \widetilde{m}_{2})$.
		\item [4.] When $m_{i,0}>0,\; i=1,2$ which implies $\mathcal{T}_t(m_{i,0})>0$. When $m_{1}(x)>n_{1}(x)$ and $m_{2}(x)<n_{2}(x)$ which implies $\mathcal{T}_t(m_1(x)) > \mathcal{T}_t(n_1(x))$ and $\mathcal{T}_t(m_2(x)) < \mathcal{T}_t(n_2(x))$.
	\end{enumerate}
	Therefore, there exactly one of the following conditions hold:
	\begin{enumerate}
		\item [(a)] There exists a positive steady state $(m_{1,s}, n_{1,s})$ of (\ref{A.1}).
		\item [(b)]  $(m_1, n_1) \rightarrow (\widetilde{m}_{1}, 0)$ when $t\rightarrow \infty$ for all $((m_{1,0}(x), n_{1,0}(x))\in \mathcal{I} =
		 \langle 0,\widetilde{m}_{1}\rangle  \times \langle 0, \widetilde{m}_{2}\rangle$. The $\langle .,. \rangle$ represents as an interval. 
		\item [(c)] $(m_1, n_1) \rightarrow (0,\widetilde{m}_{2})$ when $t\rightarrow \infty$ for all $((m_{1,0}(x), n_{1,0}(x))\in \mathcal{I} = \langle 0,\widetilde{m}_{1}\rangle  \times \rangle 0, \widetilde{m}_{2}\rangle$.
	\end{enumerate}
	Further, when (b), (c) holds then for all $(m_{1,0}, n_{1,0})\in X^{+} \setminus \mathcal{I}$ and $m_{1,0}, n_{1,0}\neq 0$ either $(m,n)\rightarrow (\widetilde{m}_{1}, 0)$ or $(m_1, n_1) \rightarrow (0,\widetilde{m}_{2})$ when $t\rightarrow \infty$.
\end{Th}


The following definition represents quasimonotone nonincreasing function [\cite{Ref_4}, Definition 8.1.1], [\cite{Ref_6}, Definition 1].

\begin{Def}\label{def_21}
	The function $g_i(m_1, m_2)$ is said to quasimonotone nonincreasing if $g_i$ be nonincreasing in $m_j$ for $i \neq j$. The vector-function ${\bf g} = (g_1, g_2)$ is said to quasimonotone nonincreasing in the domain $\mathcal{J}_1 \times \mathcal{J}_2$ moreover, both $g_1(m_1, m_2), \;g_2(m_1, m_2)$ are quasimonotone nonincreasing for $(s, n) \in \mathcal{J}_1 \times \mathcal{J}_2$.
\end{Def}

The following theorem represents the existence-uniqueness for parabolic paired systems which is discussed in \cite{Ref_4}.
\begin{Th}\label{th_A.2}
	Let $\mathbf{S}_\rho\equiv \left\{\left(m_1,n_1\right)\in C \left([0,\infty) \times\overline{\Omega}\right);\; 0\leq m_1\leq \rho_m,\; 0\leq n_1\leq \rho_n\right\} $ where $\rho_{m,n}\equiv const $. Assume $(f_1,f_2)$ in (\ref{A.1}) is quasimonotone nonincreasing Lipshitz functions in $\mathbf{S}_\rho$. Let $f_{1,2}$ satisfy 
	\begin{align*}
		&	f_1(t,x,\rho_1,0)\leq 0\leq f_1(t,x,0,\rho_2),\\
		& f_2(t,x,0,\rho_2)\leq 0\leq f_2(t,x,\rho_1,0).
	\end{align*}
	for any $x\in \Omega,\; t>0 $. Then for any $(m_{1,0},n_{1,0}) \in\mathbf{S}_\rho$, then there subsists and stays in $\mathbf{S}_\rho$ for every $x\in \Omega,\; t>0 $ a unique solution of (\ref{A.1}) $\mathbf{m}=(m_1,n_2)\in\mathbf{S}_\rho$ and $m_i(t,x)>0 $ for $x\in \Omega,\; t>0 $ since $m_{i,0}\not\equiv 0,\; i=1,2$.
\end{Th}

Pertaining to stability features of systems with unique equilibria, the following theorem plays a crucial role (see, [\cite{Ref_4}, Theorem 10.5.3]). 
\begin{Th}\label{th_A.3}
	Assume $\mathbf{\widetilde m}=(\widetilde m_1,\widetilde m_2)$ and $\mathbf{\widehat m}=(\widehat m_1,\widehat m_2)$ is ordered upper and lower solutions of (\ref{A.1}). Suppose $(f_1,f_2)$ is quasimonotone non-increasing (non-decreasing) for $\langle \mathbf{\widetilde m},\mathbf{\widehat m}\rangle \equiv \widetilde m_1\leq m_1\leq \widehat m_1,\;\widetilde m_2\leq m_2\leq \widehat m_2 $. 
	When the solution $(m_s, n_s)$ is unique in $\langle \mathbf{\widetilde m},\mathbf{\widehat m}\rangle$ and the initial conditions in $\langle\mathbf{\widetilde m},\mathbf{\widehat m}\rangle$, the solution $(m_1, m_2)$ of  (\ref{A.1}) converges to $(m_s,n_s)$ when $t \rightarrow\infty$, moreover, it is also valid on the contrary.
\end{Th}

The following theorem represents the Gr\"onwall inequality theorem which is discussed in \cite{Ref_8}.
\begin{Th}\label{th_A.4}
	Let $\sigma < \tau$ and also assume that $\varphi, \vartheta$ and $\theta$ are continuous integrable functions which is defined on the interval $[\sigma, \tau]$ and $\varphi$ be differentiable on $(\sigma, \tau)$. Consider $t \in[\sigma, \tau]$,
	\begin{align*}
		\vartheta(t)\leq \varphi(t) +\int_{\sigma}^{t} \theta(s)\vartheta(s)ds,
	\end{align*}
	hence, we obtain
	\begin{align*}
		\vartheta(t)\leq \varphi(t)exp\left(\int_{\sigma}^{t} \theta(s)ds\right).
	\end{align*}
\end{Th}

\end{document}